\documentclass[reqno]{article}
\usepackage{amssymb}
\usepackage{amsmath}
\usepackage{amsthm}

\usepackage{enumerate}
\usepackage{srcltx}
\usepackage[mathscr]{euscript}
\usepackage{times}
\usepackage{lipsum}

\usepackage{tikz}

\AtEndDocument{\bigskip{\footnotesize%
 \noindent \textsc{IMPA, Estrada Dona Castorina 110, CEP 22460 Rio de
   Janeiro, Brasil} \par 
\noindent \textsc{CNRS UMR 6085, Universit\'e de Rouen, Avenue de
   l'Universit\'e, BP.12,} \par
\noindent \textsc{Technop\^ole du Madril\-let, F76801
   Saint-\'Etienne-du-Rouvray, France} \par  
\noindent  \textrm{E-mail address}: \texttt{landim@impa.br}
}}

\makeatletter
\@addtoreset{equation}{section}
\makeatother

\newtheorem{theorem}{Theorem}[section]
\newtheorem{lemma}[theorem]{Lemma}
\newtheorem{proposition}[theorem]{Proposition}
\newtheorem{corollary}[theorem]{Corollary}

\newtheorem{definition}[theorem]{Definition}
\newtheorem{remark}[theorem]{Remark}

\newcounter{as}

\newcommand{\mc}[1]{{\mathcal #1}}
\newcommand{\mf}[1]{{\mathfrak #1}}
\newcommand{\mb}[1]{{\mathbf #1}}
\newcommand{\bb}[1]{{\mathbb #1}}
\newcommand{\bs}[1]{{\boldsymbol #1}}
\newcommand{\ms}[1]{{\mathscr #1}}

\newcommand{\<}{\langle}
\renewcommand{\>}{\rangle}
\renewcommand{\Cap}{{\rm cap}}

\title{Metastable Markov chains}
  
\author{C. Landim}

\date{}

\begin{document}

\maketitle

\begin{abstract}
  We review recent results on the metastable behavior of
  continuous-time Markov chains derived through the characterization
  of Markov chains as unique solutions of martingale problems.
\end{abstract}

\tableofcontents

\bigskip

We present in this review recent developments in the theory of
metastable Markov chains. The goal of the theory consists in
describing the evolution of a Markov chain by a simpler dynamics,
typically one whose state-space is much smaller than the original one,
preserving the ``macroscopic'' features of the original process. 

To illustrate the problem, we present in the next section an example
which motivates the definitions of metastability introduced in Section
\ref{sec1}. We then develop three general methods, based on the
characterization of Markov chains as solutions of a martingale
problems, to derive the metastable behavior of these dynamics.

There are two recent and compulsory monographs on this subject. The
first one, by Olivieri and Vares \cite{ov05}, addresses the problem
from the perspective of the large deviations theory, and the second
one, by Bovier and Den Hollander \cite{bh15}, uses potential theoretic
tools. We do not recall these approaches here and refer the reader to
the books. The reader will also find there physical motivations, an
historical account and an exhaustive list of references, three aspects
which are overlooked here. We tried, though, to include in the
references the articles published after 2015.

Throughout the article, all new notation and concepts are introduced
in blue. We believe this will help the reader who may want to skip
some introductory parts. We present in Section \ref{sec13} and
\ref{sec10} all results on Markov chains and potential theory used in
the article. Comments on the method presented in this review are left
to the end of Subsection \ref{ssec1.3}.

\section{A random walk in a graph}
\label{sec0}

We present in this section an example of a Markov chain to motivate
three different definitions of metastability.  Denote by
$\color{blue} E_N$, $N\ge 1$, the set shown in Figure
\ref{fig1}. In this picture, each large square represents a
$d$-dimensional discrete cube of length $N$,
${\color{blue}\Lambda_N} = \{1, \dots, N\}^d$, $d\ge 2$.  Each pair of
neighboring cubes has one and only one common point. In particular,
$E_N$ has $4 (N^d-1)$ elements.  Elements of $E_N$ are
represented by the Greek letters $\eta$, $\xi$, $\zeta$, and are
called points or configurations.

\begin{figure}[h]
\centering
\begin{tikzpicture}[scale = .2]
\foreach \x in {0, ..., 8}
\foreach \y in {0, ..., 8}
\draw (\x,\y) -- (\x,\y+1) -- (\x+1,\y+1) -- (\x+1,\y) -- (\x,\y);
\foreach \x in {9, ..., 17}
\foreach \y in {9, ..., 17}
\draw (\x,\y) -- (\x,\y+1) -- (\x+1,\y+1) -- (\x+1,\y) -- (\x,\y);
\foreach \x in {-9, ..., -1}
\foreach \y in {9, ..., 17}
\draw (\x,\y) -- (\x,\y+1) -- (\x+1,\y+1) -- (\x+1,\y) -- (\x,\y);
\foreach \x in {0, ..., 8}
\foreach \y in {18, ..., 26}
\draw (\x,\y) -- (\x,\y+1) -- (\x+1,\y+1) -- (\x+1,\y) -- (\x,\y);
\draw[<->, red] (-10,9) -- (-10,18);
\draw[red]  (-10, 14) node[anchor=east] {$N$};
\draw (0, 4) node[anchor=east] {$E_{2,N}$};
\draw (9, 13) node[anchor=east] {$E_{1,N}$};
\draw (-12, 13) node[anchor=east] {$E_{3,N}$};
\draw (0,22) node[anchor=east] {$E_{0,N}$};
\end{tikzpicture}
\caption{The set $E_N$}
\label{fig1}
\end{figure}
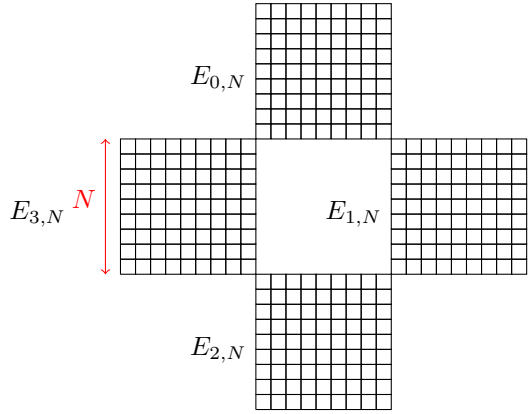

Let $\color{blue} E_{j,N}$, $0\le j\le 3$, be copies of
$\Lambda_N$. The set $E_N$ is formed by the union of the sets
$E_{j,N}$ in which some corner points have been identified. We
denote by $E_{0,N}$ the north cube and proceed labeling the sets
in the clockwise order so that $E_{3,N}$ represents the west
cube.

Denote by $\color{blue}\eta_N(t)$ the continuous-time,
$E_N$-valued, Markov chain which waits a mean-one exponential
time at each configuration and then jumps uniformly to one of the
neighbor points. This Markov chain is clearly irreducible.  Denote by
$\color{blue} \text{deg }(\eta)$, $\eta\in E_N$, the degree of the
configuration $\eta$, that is the number of neighbors.  The measure
$\color{blue}\pi_N$, defined by
$\pi_N(\eta) = Z^{-1}_N \text{deg }(\eta)$, where $Z_N$ is the
normalizing constant which turns $\pi_N$ a probability measure,
satisfies the detailed balance conditions, and is therefore the unique
stationary state.

The purpose of this section is to provide a \emph{\color{blue}
  synthetic description} of the Markov chain $\eta_N(t)$. In this
example, the reduced model is evident. Denote by
${\color{blue}\Upsilon_N} : E_N \to \{0, 1,2,3\}$ the projection
which sends a configuration in $E_{j,N}$ to $j$:
\begin{equation*}
\Upsilon_N(\eta) \;=\; \sum_{k=0}^3 k \, \chi_{E_{k,N}} 
(\eta) \;,
\end{equation*}
where $\color{blue}\chi_A$ stands for the indicator function of the
set $A$.  The value of $\Upsilon_N$ at the intersections of the cubes
is not important and can be set arbitrarily.

The derivation of the asymptotic evolution of the \emph{\color{blue}
 coarse-grained model}
\begin{equation*}
Y_N(t) \;=\; \Upsilon_N(\eta_N(t))
\end{equation*}
is based on properties of random walks evolving on discrete
cubes. Denote by $\color{blue}z_N(t)$ the symmetric, continuous-time
random walk on $\Lambda_N$ [the process $\eta_N(t)$ restricted to
$\Lambda_N$], and by $\color{blue}\pi_{\Lambda_N}$ its stationary
state, the probability measure which gives weights proportional to the
degree of the vertices. It is well known, cf. \cite[Proposition
10.13]{LevPerWil09}, that the mixing time of $z_N(t)$ is of order
$N^2$ and that the time needed to hit a point at distance $N$ is of
order $\color{blue}\alpha_N = N^2\log N$ in dimension $2$, and
$\color{blue} \alpha_N = N^d$ in dimension $d\ge 3$.

Assume that the chain starts at the center of the cube
$E_{j,N}$. Denote by $\color{blue}B$ the set of points which
belong to more than one cube, called hereafter the \emph{\color{blue}
  intersection points}, and by $H^N_{B}$ the hitting time of $B$:
\begin{equation*}
H^N_{B} \;=\; \inf\{t\ge 0 : \eta_N(t) \in B\}\;.
\end{equation*}
Since the mixing time is of order $N^2$ and the hitting time $H^N_{B}$
is of a much larger order, the chain \emph{\color{blue} equilibrates},
or \emph{\color{blue} thermalizes}, before reaching one of the corners
of $E_{j,N}$. This mean that the distribution of the chain approaches
$\pi_{\Lambda_N}$ before attaining $B$. In particular, $\eta_N(t)$
looses track of its starting point before hitting one of the corners,
and it reaches one of the two intersection points with a probability
close to $1/2$.

After thermalizing inside the cube $E_{j,N}$, the random walk
$\eta_N(t)$ wanders around $E_{j,N}$ for a length of time of order
$\alpha_N$, and then attains a point in the intersection of
$E_{j,N}$ with $E_{j\pm 1,N}$, where summation is performed
modulo $4$. Denote this point by $\xi$, and assume, to fix ideas, that
it belongs to $E_{j,N} \cap E_{j+ 1,N}$.

Fix a sequence $\color{blue} (\ell_N: N\ge 1)$ such that
$\ell_N\to\infty$, $\ell_N/N\to 0$. The precise choice of $\ell_N$ is
not important. Denote by $\color{blue} V_N$ the set of points in
$E_N$ which are at an Euclidean distance $\ell_N$ or less from
$\xi$. After hitting $\xi$, the random walk performs some short
excursions from $\xi$ to $\xi$ which remain in $V_N$. Some of these
excursions are contained in the set $E_{j,N}$ and some in
$E_{j+ 1,N}$.

It takes a time of order $\ell^2_N$ for $\eta_N(t)$ to escape from
$V_N$, that is, to reach a point in $V^c_N$, the complement of
$V_N$. Note that $\ell^2_N$ is much smaller than $\alpha_N$ and so the
escape time from $V_N$ is negligible in this time-scale.

Starting from a point at the external boundary of $V_N$, it takes a
time of order $N^2 \log \ell_N$ in dimension $2$ and $N^d$ in
dimension $d\ge 3$ to hit again the set $B$. Since this time is much
longer than the mixing time, once in $V_N^c$, before hitting the set
$B$ again, the process equilibrates inside the cube. Thus, we are back
to the initial situation, and we can iterate the previous argument to
provide a complete description of the evolution of the random walk
$\eta_N(t)$ among the cubes.

According to the previous analysis, the evolution of the random walk
can be described as follows. Starting from a point not too close from
the corners, the random walk equilibrates in the cube from where it
starts before it reaches one of the intersection points. Since it has
equilibrated, it reaches one of the two boundary points with equal
probability. Then, after some short excursion close to the
intersection point, it escapes from the corner to one of the
neighboring cubes, with equal probability due to the symmetry of the
set $E_N$. In particular, with probability $1/2$ the random walk
returns to the cube from which it came when it hit the intersection
point. The escape time being much shorter than the equilibration time,
the small excursions around the intersection can be neglected in the
asymptotic regime. After escaping, the process equilibrates in the
cube where it is and we may iterate the description of the evolution.

Loss of memory being the essence of Markovian evolution, in the
time-scale $\alpha_N$, the coarse-grained, speeded-up process 
\begin{equation*}
\mb Y_N(t)  \;:=\; Y_N(t \alpha_N) \;=\; \Upsilon_N(\eta_N(t \alpha_N))
\end{equation*}
should evolve as a $\color{blue} S:=\{0, 1,2,3\}$-valued,
continuous-time Markov chain $\mb Y(t)$ with holding rates equal to
some $\lambda>0$ and jump probabilities given by $p(j,j\pm 1) = 1/2$.

In which sense can $\mb Y_N(t)$ converge to a Markov chain?  Figure
\ref{fig2} presents a typical realization of the process $\mb Y_N(t)$.
The process remains a time interval of order $\alpha_N$ at a point
$x\in S$ until $\eta_N(t)$ reaches an intersection point. At this
time, $\eta_N(t)$ performs very short excursions [in the time scale
$\alpha_N$] in both neighboring squares. These short excursions are
represented in Figure \ref{fig2} by the bold rectangles to indicate a
large number of oscillations in a very short time interval. After many
short excursions the random walk escapes from the boundary and remains
in one of the neighboring cubes for a new time interval of order
$\alpha_N$.

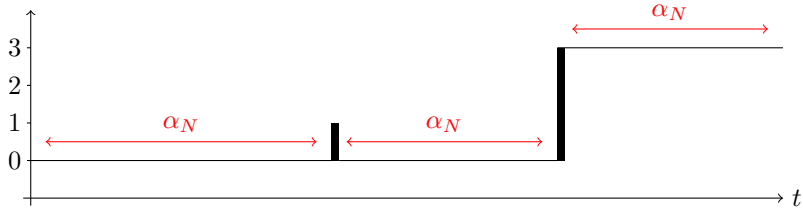
\begin{figure}[h]
  \centering
\begin{tikzpicture}[scale = .5]
\draw[->] (-.2, 0) -- (20,0) node[right] {$t$};
\draw[->] (0,-.2) -- (0,5);
\foreach \y in {0, ..., 3}
\draw (-.1,1+\y) -- (0,1+\y) node[left] {$\y$};
\draw (0, 1) -- (8,1);
\fill[black!100] (8,1) -- (8.2,1) -- (8.2,2) -- (8,2) -- (8,1);
\draw (8.2, 1) -- (14,1);
\fill[black!100] (14,1) -- (14.2,1) -- (14.2,4) -- (14,4) -- (14,1);
\draw (14, 4) -- (20,4);
\draw [<->, red] (0.4,1.5) -- (7.6,1.5);
\draw [red] (4,1.5) node[above] {$\alpha_N$};
\draw [<->, red] (8.4,1.5) -- (13.6,1.5);
\draw [red] (11,1.5) node[above] {$\alpha_N$};
\draw [<->, red] (14.4,4.5) -- (19.6,4.5);
\draw [red] (17,4.5) node[above] {$\alpha_N$};
\end{tikzpicture}
\caption{A typical trajectory of the process $X_N(t)$. The red arrows
  indicate the length of the time intervals which are of order
  $\alpha_N$.}
\label{fig2}
\end{figure}

These fluctuations in very short time intervals, represented by the black
rectangles in Figure \ref{fig2}, rule out the possibility that
$\mb Y_N(t)$ converges in any of the Skorohod topologies. Thus, either
we content ourselves with the convergence of the finite-dimensional
distributions or we need to adjust the trajectories of $\mb Y_N(t)$
by removing these short excursions. 

The first step consists in introducing a set
$\color{blue} \Delta_N \subset E_N$ to separate the squares
$E_{j,N}$.  This procedure is illustrated in Figure \ref{fig4},
where $\ms E^j_N$ represents $E_{j,N} \setminus \Delta_N$. The
set $\Delta_N$ is not unique. We only require that it is small enough
for the fraction of time spent in $\Delta_N$ to be negligible, but
large enough for the process, starting from a point outside of
$\Delta_N$, to equilibrate before it hits an intersection point.

In the example of this section, the set $\color{blue} \ms E^k_N$ can
be the points of $E_{k,N}$ which are at distance at least
$\ell_N$ from the intersection points, or, as in Figure \ref{fig4},
the set of points at distance greater than $\ell_N$ from the faces of
the cubes. Here, as above, $\ell_N$ is a sequence such that
$\ell_N\to\infty$, $\ell_N/N \to 0$.

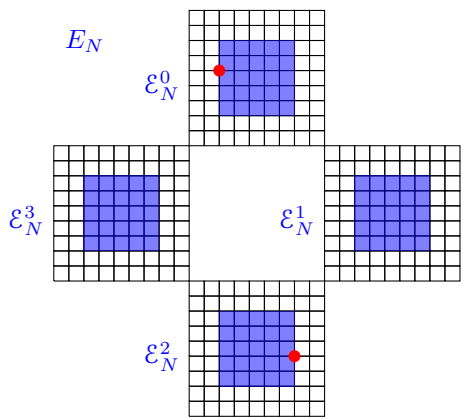
\begin{figure}[h]
  \centering
\begin{tikzpicture}[scale = .20]
\foreach \x in {0, ..., 8}
\foreach \y in {0, ..., 8}
\draw (\x,\y) -- (\x,\y+1) -- (\x+1,\y+1) -- (\x+1,\y) -- (\x,\y);
\foreach \x in {9, ..., 17}
\foreach \y in {9, ..., 17}
\draw (\x,\y) -- (\x,\y+1) -- (\x+1,\y+1) -- (\x+1,\y) -- (\x,\y);
\foreach \x in {-9, ..., -1}
\foreach \y in {9, ..., 17}
\draw (\x,\y) -- (\x,\y+1) -- (\x+1,\y+1) -- (\x+1,\y) -- (\x,\y);
\foreach \x in {0, ..., 8}
\foreach \y in {18, ..., 26}
\draw (\x,\y) -- (\x,\y+1) -- (\x+1,\y+1) -- (\x+1,\y) -- (\x,\y);
\fill [rectangle,blue,opacity=.5] (2,2)--(2,7)--(7,7)--(7,2);
\fill [rectangle,blue,opacity=.5] (11,11)--(16,11)--(16,16)--(11,16);
\fill [rectangle,blue,opacity=.5] (-7,11)--(-2,11)--(-2,16)--(-7,16);
\fill [rectangle,blue,opacity=.5] (2,20)--(7,20)--(7,25)--(2,25);
\draw[blue] (0, 4) node[anchor=east] {$\ms E^2_N$};
\draw[blue] (9, 13) node[anchor=east] {$\ms E^1_N$};
\draw[blue] (-9, 13) node[anchor=east] {$\ms E^3_N$};
\draw[blue] (0,22) node[anchor=east] {$\ms E^0_N$};
\draw[blue] (-5, 25) node[anchor=east] {$E_N$};
\fill [red] (2,23) circle [radius= .4 cm];
\fill [red] (7,4) circle [radius= .4 cm];
\end{tikzpicture}
\caption{The sets $\ms E^k_N$ are indicated in blue. The two red dots
  represent points in $\ms E^0_N$ and $\ms E^2_N$. The trace process
  $\eta^{\ms E}(t)$ may jump from one to the other. It has therefore
  long jumps, in contrast with the original random walks which only
  jumps to nearest neighbors. The picture is misleading as the annulus
  around each blue square is much smaller than the square.}
\label{fig4}
\end{figure}

In the next section, we propose two different types of amendments of
the trajectories of $\eta_N(t)$ to achieve convergence in the Skorohod
topology of the coarse-grained model. 

Before we turn to that, consider the example shown in Figure
\ref{fig5}. Assume that each line has $N$ points, counting the common
intersection point. Consider a random walk evolving on this graph. The
process waits a mean-one exponential time at the end of which it jumps
to one of its neighbors with equal probability. Since one-dimensional
random walks on a set of $N$ points equilibrate in a time of order
$N^2$, and since it hits a point a distance $N$ in the same time-scale,
there is no \emph{\color{blue} separation of scales} and the argument
presented above to claim the possibility of a synthetic description
of the dynamics does not apply.

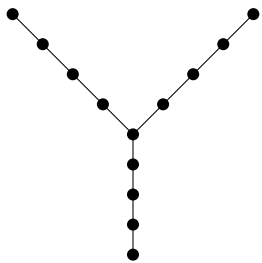
\begin{figure}[h]
  \centering
\begin{tikzpicture}[scale = .20]
\draw (0,0) -- (0,-8);
\draw (0,0) -- (8,8);
\draw (0,0) -- (-8,8);
\foreach \x in {0, ..., 4}
\fill [black] (0, -2*\x) circle [radius= .4 cm];
\foreach \x in {1, ..., 4}
\fill [black] (2*\x,2*\x) circle [radius= .4 cm];
\foreach \x in {1, ..., 4}
\fill [black] (-2*\x,2*\x) circle [radius= .4 cm];
\end{tikzpicture}
\caption{A random walk on a graph which does not have a synthetic
  description as a $3$-state Markov chain.} 
\label{fig5}
\end{figure}

\section{Metastability as model reduction}
\label{sec1}

The phenomenon described in the previous section, in which a process
remains a long time in a set in which it equilibrates before it
attains, in a very short transition, another set where the same
behavior is observed, is shared by many different types of
dynamics (cf. Section \ref{sec12} for many examples).

For this reason, we present in a general framework the adjustments
needed in the trajectory of the coarse-grained model to yield
convergence in the Skorohod topology.  Let $(E_N : N\ge 1)$ be a
sequence of finite state spaces. Elements of $E_N$ are represented by
the Greek letters $\eta$, $\xi$, $\zeta$. Denote by $\eta_N(t)$ a
continuous-time, $E_N$-valued, irreducible Markov chain.  Its
generator is represented by $\color{blue} \ms L_N$ and its unique
stationary state by $\color{blue} \pi_N$. Therefore, for every
function $f: E_N \to \bb R$,
\begin{equation*}
(\ms L_N\, f)(\eta) \;=\; \sum_{\xi\in E_N} R_N(\eta,\xi)\,
\big[\, f(\xi) - f(\eta)\, \big]\;,
\end{equation*}
where $\color{blue} R_N(\eta,\xi)$ stands for the jump rates.

For a nonempty subset $A$ of $E_N$, let $H_A$, resp. $H^+_A$, stands
for the hitting time of the set $A$, resp. the return time to $A$,
\begin{equation}
\label{1-5}
{\color{blue} H_A} \;:=\; \inf\{t\ge 0 : \eta_N(t) \in A\}\;,
\quad {\color{blue} H^+_A} \;=\; \inf\{t\ge \tau_1 : \eta_N(t) \in A\}\;.
\end{equation}
in this formula, $\tau_1$ represents the time of the first jump of
$\eta_N(t)$, $\tau_1 = \inf\{t \ge 0 : \eta_N(t) \not = \eta_N(0)\}$.

Assume that $E_N$ contains $\color{blue} \mf n>1$ disjoint sets
$\ms E^1_N, \dots, \ms E^{\mf n}_N$, called \emph{\color{blue}
  valleys}, separated by a set $\Delta_N$, so that
$\ms E^1_N, \dots, \ms E^{\mf n}_N, \Delta_N$ forms a partition of
$E_N$.  Let $\color{blue} S := \{1, \dots, \mf n\}$, and denote by
$\Phi_N\colon E_N \to S \cup \{\mf d\}$ the projection which sends a
configuration in $\ms E^j_N$, $\Delta_N$ to $j$, $\mf d$,
respectively:
\begin{equation*}
{\color{blue} \Phi_N (\eta)} \;:=\; \sum_{k=1}^{\mf n} k \; 
\chi_{\ms E^k_N} (\eta) \;+\; \mf d\, \chi_{\Delta_N} (\eta)\;.
\end{equation*}
Let $X_N(t)$ be the $(S\cup \{\mf d\})$-valued process given by
\begin{equation}
\label{A-1}
{\color{blue}  X_N(t)} \;:=\; \Phi_N\big(\, \eta_N (t) \,\big)\;.
\end{equation}

In the example of the previous section, the trajectory of $X_N(t) =
\Phi_N(\eta_N(t))$ resembles the one presented in Figure \ref{fig2}
with additional spikes due to very short excursions [in the time scale
$\alpha_N$] out of $\ms E^k_N$ which occur far from the intersection
points.

\subsection{Last passage} 
\label{ssec1.1}

The first adjustment of the trajectories which enables convergence in
the Skorohod topology consists in removing the fast fluctuations by
recording the last set $\ms E^k_N$ visited by $\eta_N(t)$. For $t>0$,
denote by $\eta_N(t\,-)$ the left limit of $\eta_N$ at $t$:
\begin{equation*}
\eta_N(t\,-) \;=\; \lim_{s\to t \,,\, s<t} \eta_N(s)\;.
\end{equation*}
Let $X^V_N(t)$ be given by
\begin{equation}
\label{A-2}
{\color{blue} X^V_N(t)} \;:=\; 
\Phi_N\big(\,\eta_N(\, \mf v_N (t)\, )\, \big)\;.
\end{equation}
where 
\begin{equation*}
{\color{blue} {\mf v_N (t)}} \;=\;
\begin{cases}
t & \text{if $\eta_N (t) \in \ms E_N$}\;, \\
\mf w_N (t)\,- & \text{otherwise}\;,
\end{cases}
\end{equation*}
and $\mf w_N (t)$ represents the last time before $t$ the process was
in one of the valleys $\ms E^k_N$:
\begin{equation*}
\mf w_N (t) \;:=\; \sup \{ s\le t : \eta_N(s)\in \ms E_N\}
\quad\text{and}\quad
{\color{blue} \ms E_N} \;:=\; \bigcup_{k=1}^{\mf n} \ms E^k_N
\;.
\end{equation*}
If the set on the right-hand side is empty, we set $\mf w_N
(t)=0$. This remark is not important as we will always start the
process from a configuration in $\ms E_N$. Note that $X^V_N(t) \in S$
because $\eta_N(\mf v_N (t)) \in \ms E_N$ for all $t\ge 0$ whenever
$\eta_N(0)\in \ms E_N$.

The time change $\mf v_N (t)$ removes the rapid oscillations from the
trajectory. Indeed, in the example of the previous section assume that
the process starts from a configuration in $\ms E^k_N$, and denote by
$\ms E^j_N$ the next valley visited. Recall that $H_{\ms E^j_N}$
represents the hitting time of this valley.  In the time interval $[0,
H_{\ms E^j_N})$, during the rapid excursions of the random walk
$\eta_N(t)$ in $\Delta_N$, $X^V_N(t)$ remains equal to $k$. In
particular, the fast fluctuations in the time interval $[0, H_{\ms
  E^j_N}]$ are washed out. We may iterate the argument starting from
time $H_{\ms E^j_N}$ to extend this property to the full trajectory.

Since $H_{\ms E^j_N} $ is of order $\alpha_N$, the trajectory of
$X^V_N(t)$ is formed by a sequence of time intervals of this magnitude
in which the process remains constant. The objections raised above for
the convergence in the Skorohod topology are thus overturned, and we
may expect, due to the loss of memory which emerges from the
equilibration, that in the time scale $\alpha_N$, $X^V_N(t)$ converges
to a $S$-valued Markov chain in the Skorohod topology.

\begin{definition}[Metastability according to LP]
\label{l1-1}
The Markov chain $\eta_N(t)$ is said to be metastable, in the sense of
last passage, in the time-scale $\theta_N$ if there exists a partition
$\{\ms E^1_N, \dots, \ms E^{\mf n}_N, \Delta_N\}$ of the state space
$E_N$ and a $S$-valued, continuous-time Markov chain $\bs X(t)$
such that 
\begin{enumerate}
\item[{\bf (LP1)}] For any $k\in S=\{1, \dots, \mf n\}$ and any sequence
$(\eta_N:N\ge 1)$ such that $\eta_N\in \ms E^k_N$, starting from
$\eta_N$, $\bs X^V_N(t) = X^V_N(t\, \theta_N)$ converges in the
Skorohod topology to $\bs X(t)$.
\item[{\bf (LP2)}] The time spent in $\Delta_N$ is negligible: For all $t>0$
\begin{equation*}
\lim_{N\to\infty}
\max_{\eta\in\ms E_N} \bb E^N_{\eta} \Big[\,
\int_0^t \chi_{\Delta_N} \big(\eta_N(s\theta_N)\big) 
\; ds\, \Big] \;=\; 0\;. 
\end{equation*}
\end{enumerate}
The sets $\ms E^j_N$ are called valleys and the process $\bs X(t)$ the
\emph{\color{blue} reduced model}.
\end{definition}

The main difficulty in proving such a result lies in the fact that the
process $\eta_N(\, \mf v_N (t)\, )$ is not markovian. For this reason
we propose an alternative modification of the trajectory which keeps
this property. This method requires the definition of the trace of a
process, which we present below in the context of continuous-time
Markov chains taking values in a finite state space.

\subsection{Trace process}  
\label{ssec1.2}

Let $E$ be a finite set and let $\eta(t)$ be an irreducible,
continuous-time, $E$-valued Markov chain. Denote by $\color{blue}
R(\eta,\xi)$, $\eta\not = \xi\in E$, the jump rates of this chain, by
$\color{blue}\lambda (\eta) = \sum_{\xi\in E} R(\eta,\xi)$ the holding
rates, and by $\color{blue} \pi$ the unique stationary probability
measure.

Denote by $\color{blue} D([0,\infty), E)$ the space of
right-continuous trajectories $\omega: [0,\infty) \to E$ which have
left limits endowed with the Skorohod topology \cite{Bil99}. This
notation will be used below, without further comments, replacing $E$
by another metric space. Let $\color{blue} \bb P_\eta$, $\eta\in E$,
be the probability measures on $D([0,\infty), E)$ induced by the
Markov chain $\eta(t)$ starting from $\eta$. Expectation with respect
to $\bb P_\eta$ is represented by $\color{blue} \bb E_\eta$.

Fix a non-empty, proper subset $F$ of $E$ and denote by $T_{F} (t)$,
$t\ge 0$, the total time the process $\eta(t)$ spends in $F$ on the
time-interval $[0,t]$:
\begin{equation*}
{\color{blue} T_F(t)}
\;:=\; \int_0^{t} \chi_F( \eta(s)) \; ds\; ,
\end{equation*}
where, we recall, $\chi_F$ represents the indicator function of the
set $F$.  Denote by $S_F(t)$ the generalized inverse of the additive
functional $T_F(t)$:
\begin{equation}
\label{1-3}
{\color{blue} S_F(t)} \;:=\; \sup\{s\ge 0 : T_F(s) \le t\}\;. 
\end{equation}
The irreducibility guarantees that for all $t>0$, $S_F(t)$ is finite
almost surely.

The process $T_F$ is continuous. It is either constant, when the chain
visits configurations which do not belong to $F$, or it increases
linearly. Figure \ref{fig3} illustrates this behavior.  Denote by
$\eta_F(t)$ the \emph{\color{blue} trace} of the chain $\eta(t)$ on
the set $F$, defined by $\color{blue} \eta_F(t) := \eta(S_F(t))$.
Taking the trace of the process corresponds to changing the axis of
time in Figure \ref{fig3}. When the process hits $F^c$, time is frozen
until $\eta(t)$ reaches $F$ again, at which time the clock is
restarted. In particular, $\eta_F(t)$ takes values in the set $F$.

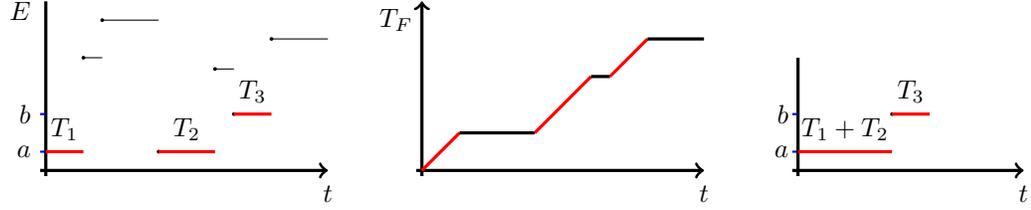
\begin{figure}[h]
\centering
\begin{tikzpicture}[scale = .25]
\draw[very thick, ->] (-.3, 0) -- (15,0);
\draw[very thick] (0, -.3) -- (0,9);
\fill[black!100] (0,1) circle [radius= .1 cm];
\draw[very thick, red] (0,1) -- (2,1);
\fill[black!100] (2,6) circle [radius= .1 cm];
\draw (2,6) -- (3,6);
\fill[black!100] (3,8) circle [radius= .1 cm];
\draw (3,8) -- (6,8);
\fill[black!100] (6,1) circle [radius= .1 cm];
\draw[very thick, red] (6,1) -- (9,1);
\fill[black!100] (9,5.4) circle [radius= .1 cm];
\draw (9,5.4) -- (10,5.4);
\fill[black!100] (10,3) circle [radius= .1 cm];
\draw[very thick, red] (10,3) -- (12,3);
\fill[black!100] (12,7) circle [radius= .1 cm];
\draw (12,7.0) -- (15,7);
\draw (15,-.3) node[anchor = north]{$t$};
\draw (1,1.1) node[anchor = south]{$T_1$};
\draw (7.5,1.1) node[anchor = south]{$T_2$};
\draw (11,3.1) node[anchor = south]{$T_3$};
\draw (-.3,8.5) node[anchor = east]{$E$};
\draw (-.3,1) node[anchor = east]{$a$};
\draw (-.3,3) node[anchor = east]{$b$};
\draw[thick, blue] (-.3, 1) -- (0,1);
\draw[thick, blue] (-.3, 3) -- (0,3);
\draw[very thick, ->] (19.7, 0) -- (35,0);
\draw[very thick, ->] (20, -.3) -- (20,9);
\draw[very thick, red] (20, 0) -- (22,2);
\draw[very thick] (22,2) -- (26,2);
\draw[very thick, red] (26,2) -- (29,5);
\draw[very thick] (29,5) -- (30,5);
\draw[very thick, red] (30,5) -- (32,7);
\draw[very thick] (32,7) -- (35,7);
\draw (20,8) node[anchor = east]{$T_F$};
\draw (35,-.3) node[anchor = north]{$t$};
\draw[very thick, ->] (39.7, 0) -- (52,0);
\draw[very thick] (40, -.3) -- (40,6);
\draw (52,-.3) node[anchor = north]{$t$};
\draw[thick, blue] (39.7, 1) -- (40,1);
\draw[thick, blue] (39.7, 3) -- (40,3);
\draw (40,1) node[anchor = east]{$a$};
\draw (40,3) node[anchor = east]{$b$};
\fill[black!100] (40,1) circle [radius= .1 cm];
\draw[very thick, red] (40,1) -- (45,1);
\draw (42.5,1.1) node[anchor = south]{$T_1+T_2$};
\fill[black!100] (45,3) circle [radius= .1 cm];
\draw[very thick, red] (45,3) -- (47,3);
\draw (46,3.1) node[anchor = south]{$T_3$};
\end{tikzpicture}
\caption{An example of the transformation which maps the chain
  $\eta(t)$ into its trace on the set $\{a,b\}$. The first graph shows
  the trajectory of $\eta(t)$, the second one the function $T_F(t)$
  for $F=\{a,b\}$, and the third one the trajectory $\eta_F(t) = \eta
  (S_F(t))$. Note that $S_F(t)$ is obtained from $T_F(t)$ by inverting
  the roles of the $x$ and $y$ axes.}
\label{fig3}
\end{figure}

It can be proven \cite[Section 6]{bl2} that $\eta_F(t)$ is an
irreducible, continuous-time, $F$-valued Markov chain. The
jump rates of the chain $\eta_F(t)$, denoted by $R_F(\eta,\xi)$, are
given by
\begin{equation}
\label{1-1}
{\color{blue} R_F(\eta,\xi)} \;:=\; \lambda(\eta) \, \bb P_\eta[H^+_{F}
= H_\xi] \;, \quad \eta\,,\, \xi\,\in\, F\;, 
\quad \eta\,\not=\, \xi\;, 
\end{equation}
where the hitting time $H_A$ and the return time $H^+_A$ have been
introduced in \eqref{1-5}.

The unique stationary probability measure of the trace chain, denoted
by $\pi_F(\eta)$, is the measure $\pi$ conditioned to $F$:
\begin{equation}
\label{1-4}
\pi_F(\eta) \;=\; \frac{\pi(\eta)}{\pi(F)}\;, \quad \eta\,\in\, F\;.
\end{equation}
Moreover, $\pi_F$ is reversible if so is $\pi$ \cite{bl2}.

\subsection{Metastability} 
\label{ssec1.3}

We return to the chain $\eta_N(t)$ introduced at the beginning of this
section. Denote by $\color{blue} \bb P^N_\eta$, $\eta\in E_N$,
the probability measures on $D([0,\infty), E_N)$ induced by the
Markov chain $\eta_N(t)$ starting from $\eta$. Expectation with
respect to $\bb P^N_\eta$ is represented by $\color{blue} \bb
E^N_\eta$.

Denote by $\color{blue} \eta^{\ms E_N}(t)$ the trace of the process
$\eta_N(t)$ on the set $\ms E_N$. As explained in Figure \ref{fig3},
by taking the trace of $\eta_N(t)$ on $\ms E_N$ we first remove from
the trajectory the time-intervals corresponding to the excursions in
$\Delta_N$ (the intervals in black in the leftmost picture of Figure
\ref{fig3}), and then, we push back the trajectory, as in the
rightmost picture of this figure. This procedure removes rapid
fluctuations from the trajectory providing an alternative definition
of metastability.

Let $\Psi_N \colon \ms E_N \to S$ the projection which sends a
configuration in $\ms E^j_N$ to $j$:
\begin{equation*}
{\color{blue} \Psi_N(\eta)} \;=\; \sum_{k=1}^{\mf n} k \; \chi_{\ms
  E^k_N} (\eta) \;. 
\end{equation*}
In contrast with $\Phi_N$, $\Psi_N$ is defined only on $\ms E_N$. Let
$X^T_N(t)$ be the process given by
\begin{equation}
\label{1-2}
{\color{blue} X^T_N(t)} \;:=\; \Psi_N\big(\, \eta^{\ms E_N}(t) \,\big)\;.
\end{equation}
Note that $X^T_N(t)$ is not a Markov chain, but just a hidden Markov
chain. It corresponds to the trace on $S$ of the process $X_N(t)$
introduced in \eqref{A-1}.

\begin{definition}[Metastability]
\label{l1-2}
The Markov chain $\eta_N(t)$ is said to be metastable in the
time-scale $\theta_N$ if there exists a partition $\{\ms E^1_N, \dots,
\ms E^{\mf n}_N, \Delta_N\}$ of the state space $E_N$ and a
$S$-valued, continuous-time Markov chain $\bs X(t)$ such that 
\begin{enumerate}
\item[{\bf (T1)}] For any $k\in S=\{1, \dots, \mf n\}$ and any
  sequence $(\eta_N:N\ge 1)$ such that $\eta_N\in \ms E^k_N$, starting
  from $\eta_N$, the process ${\color{blue}\bs X^T_N(t)} :=
  X^T_N(t\theta_N) = \Psi_N(\eta^{\ms E_N}(t\theta_N))$ converges in
  the Skorohod topology to $\bs X(t)$;
\item[{\bf (T2)}] The time spent in $\Delta_N$ is negligible: For all $t>0$
\begin{equation*}
\lim_{N\to\infty}
\max_{\eta\in\ms E_N} \bb E^N_{\eta} \Big[\,
\int_0^t \chi_{\Delta_N} \big(\eta_N(s\theta_N)\big) 
\; ds\, \Big] \;=\; 0\;. 
\end{equation*}
\end{enumerate}
\end{definition}

The first condition asserts that in the time scale $\theta_N$ the
trace on $S$ of the process $X_N(t)$ converges to a Markov chain,
while the second one states that in this time scale the amount of time
the process $X_N(t)$ spends outside $S$ is negligible, uniformly over
initial configurations in $\ms E_N$. In particular, condition (T2) can
be stated as
\begin{equation*}
\lim_{N\to\infty}
\max_{\eta\in\ms E_N} \bb E^N_{\eta} \Big[\,
\int_0^t \chi_{\mf d} \big(X_N(s\theta_N)\big) 
\; ds\, \Big] \;=\; 0\;. 
\end{equation*}

\begin{remark}
\label{A-r7}
The use of the word ``metastability'', instead of tunneling, to name
the phenomenon described in the previous section, might be
inadequate. Metastability has been used to represent the transition
from a metastable state to a stable one.  This corresponds to the case
in which the reduced model $\bs X(t)$ takes value in a set with two
elements, one being transient and the other absorbing. We allow
ourselves this abuse of nomenclature.
\end{remark}

\begin{remark}
\label{A-r1}
The same sequence of Markov chains $(\eta_N(t) : N\ge 1)$ may have
more that one metastable description. In a certain time-scale
$\alpha_N$, one may observe transitions between shallow valleys and in
a much longer time-scale $\beta_N$ transitions between deeper valleys.
\end{remark}

\begin{remark}
\label{A-r6}
There are examples of Markov chains \cite{fm1, fli1, jlt11, jlt14, bcl18} with
a countably infinite number of valleys. In these cases, the reduced
model $\bs X(t)$ is a continuous-time Markov chain in a countable
state-space. In this article, we restrict ourselves to the finite case
to avoid technical issues on the martingale problem.
\end{remark}

\begin{remark}
\label{A-r2}
One of the main features of metastability is the fast transition
between valleys. This information is encapsulated in condition (T2)
which states that the time spent outside the valleys is negligible. In
particular, the transition time between two valleys is negligible in
the metastable time-scale.
\end{remark}

\begin{remark}
\label{A-r8}
All results presented in this review are in asymptotic form, they
characterize the limiting behavior of the coarse-grained
model. Quantitative estimates at fixed $N$ are important in concrete
problems. For example, to describe synthetically a molecular dynamics
which can be represented as a Markov chain in a very large, but fixed,
state space.  The problem consists in finding a reduced model which
keeps the main features of the original chain. It might be interesting
to adapt the approach presented here to this framework.

The transition path theory \cite{ev06, msv09, ce14, lv14} has been
designed for this set-up, as well as the intertwining method
\cite{AveGau18, AveCasGauMel17a, AveCasGauMel17b,
  AveCasGauMel17c}. See also the results by Bianchi and Gaudilli\`ere
\cite{bg16}
\end{remark}

\begin{remark}
\label{A-r3}
In constrast with the pathwise approach \cite{cgov84, ov05}, no
attempt is made here to describe the transition path between two
valleys.
\end{remark}

\begin{remark}
\label{A-r4}
In the example of the previous section, the process $X_N(t)$ remains
constant in time-intervals of length of order $\alpha_N$. In this
sense, $\Psi_N$ can be understood as a \emph{\color{blue} slow
  variable}, since it evolves in a much longer time-scale than the
original process, and metastability as the search for slow variables
and the description of the evolution of these slow variables.
\end{remark}

\begin{remark}
\label{A-r5}
In most examples, as the Ising model at low temperature \cite{ns91,
  ns92}, metastability is observed as a result of the presence of an
energy barrier which the system has to overpass to reach a new region
of the state-space.

The example of the previous section is of different nature. In this
model, there is no energy landscape but a bottleneck which creates a
metastable behavior. Here, entropy [the number of configurations]
determines the height of the barriers. Say, for example, that three
squares are $3$-dimensional while the last one is $2$-dimensional. In
this case, in the time-scale $N^3$, one observes an evolution among
the $3$-dimensional cubes and the last square can be included in the
set $\Delta_N$ as the time spent there is of order $N^2\log N$.

In other models, as random walks in a potential field, both energy
and entropy play a role.
\end{remark}

\subsection{Finite-dimensional distributions} 
\label{ssec1.4}

Definition \ref{l1-1} describes the evolution of a modified version of
the original process, and Definition \ref{l1-2} the one of the trace.
To avoid tiny surgeries of the trajectories, we may turn to the
convergence of the finite-dimensional distributions, an alternative
adopted by Kipnis and Newman in \cite{kn85} and Sugiura \cite{s95,
  s01}.

\begin{definition}[Metastability according to FDD]
\label{A-d3}
The Markov chain $\eta_N(t)$ is said to be metastable, in the sense of
finite-dimensional distributions, in the time-scale $\theta_N$ if
there exists a partition $\{\ms E^1_N, \dots, \ms E^{\mf n}_N,
\Delta_N\}$ of the state space $E_N$ and a $S$-valued,
continuous-time Markov chain $\bs X(t)$ such that the
finite-dimensional distributions of $\bs X_N(t) := X_N(t\theta_N)$
converge to the ones of $\bs X(t)$.
\end{definition}

Note that while $\bs X_N(t)$ takes value in $S\cup \{\mf d\}$, $\bs
X(t)$ is $S$-valued. \smallskip

The article is organized as follows.  We present, in Sections
\ref{sec3}--\ref{sec15}, a general scheme to derive the metastable
behavior of a Markov chain in the sense of Definition \ref{l1-2} for
dynamics which ``visit points''. This approach is based on the
characterization of Markov chains as solutions of martingale problems,
examined in Section \ref{sec2}. In the following two sections, an
alternative approach is proposed for dynamics in which the entropy
plays a role in the metastable behavior. In Section \ref{sec14}, we
discuss tightness. In Section \ref{sec8}, we show that conditions
(T1), (T2) entail the metastability in the sense of the last passage,
and, in Section \ref{sec9}, we prove that these conditions together
with property \eqref{9-1} lead to the convergence of the
finite-dimensional distributions. In Section \ref{sec13} and
\ref{sec10} we recall some general results on Markov chains and
potential theory used in the article. In the last section, we list
some dynamics which fall within the scope of the theory.

\section{Martingale problems}
\label{sec2}

The proof of condition (T1) in Definition \ref{l1-2} relies on the
uniqueness of solutions of martingale problems, the subject of this
section. To avoid technical problems, we restrict ourselves to the
context continuous-time Markov chains taking values in a finite
state-space $E$. We refer to the classical books \cite{StrVar06,
  EthKur86} for further details.

Recall the notation introduced in Subsection \ref{ssec1.2}. Assume
that the Markov chain $\eta(t)$ is defined on the probability space
$\color{blue} (\Omega, \mc F, \bb P)$, where $\Omega = D([0,\infty),
E)$ and $\mc F$ represents the Borel $\sigma$-algebra of
$D([0,\infty), E)$. Let $({\color{blue} \mc F^o_t}: t\ge 0)$ be the
filtration generated by $\{\eta(s) : 0\le s\le t\}$.

Denote by $L$ the generator of the Markov chain $\eta(t)$: for every
function $f:E \to \bb R$,
\begin{equation}
\label{2-1}
{\color{blue} (Lf)(\eta)} \;=\; \sum_{\xi\in E} R(\eta,\xi) \, [f(\xi)
- f(\eta)]\;. 
\end{equation}
It is well known that for every $f:E \to \bb R$,
\begin{equation}
\label{2-2}
{\color{blue} M^f(t)} \:=\; f(\eta(t)) \;-\; f(\eta(0)) 
\;-\; \int_0^t (Lf)(\eta(s))\; ds
\end{equation}
is a zero-mean martingale in $(\Omega, (\mc F^o_t), \bb P)$.

\smallskip It turns out that the converse is true. Let $A$ be the
generator of an $E$-valued, irreducible, continuous-time Markov chain,
and $\nu$ a probability measure on $E$.

\begin{definition}[The martingale problem $(A,\nu)$]
\label{3-d1}
A probability measures $\mb P$ on $(\Omega, \mc F)$ is a solution of
the martingale problem associated to the generator $A$ and the measure
$\nu$ if for every $f:E\to \bb R$ the process $M^f$ given by
\eqref{2-2} [with $L$ replaced by $A$] is a martingale in
$(\Omega, (\mc F^o_t), \mb P)$ and $\mb P[\eta(0)=\eta] = \nu(\eta)$
for all $\eta\in E$.
\end{definition}

Next result is a particular case of Theorem 4.4.1 in \cite{EthKur86}.

\begin{theorem}
\label{3-l1}
Let $A$ be the generator of an $E$-valued, irreducible,
continuous-time Markov chain.  For every probability measure $\nu$ in
$E$, there exists a unique solution of the martingale problem
associated to the generator $A$ and the measure $\nu$. Moreover, under
this solution, the process $\eta(t)$ is the continuous-time Markov
chain whose generator is $A$.
\end{theorem}

This result provides a simple strategy to prove condition (T1) of
Definition \ref{l1-2}.  Fix $k\in S$, a sequence
$\eta_N\in \ms E^k_N$, and denote by $\color{blue} \mb P^N$ the
probability measure on $D([0,\infty), S)$ induced by the process
$\bs X^T_N(t) = X^T_N(t \theta_N)$ and the measure $\bb P^N_{\eta_N}$.
Prove first that the sequence $\mb P^N$ is tight. Then, to
characterize the limit points, show that they solve a martingale
problem $(L,\delta_k)$, where $L$ is the generator of a $S$-valued
Markov chain [guessed a priori] and $\color{blue} \delta_k$ the
probability measure on $S$ concentrated on $k$. Tightness is postponed
to Section \ref{sec14} and uniqueness is discussed in the next
sections.

\section{The martingale approach}
\label{sec3}

We carry out in this section the strategy outlined in the previous
section to prove the uniqueness of limit points of the sequence $\bs
X^T_N(t)$.  It is based on the uniqueness of solutions of martingale
problems, presented above, and on the fact that limits of martingales
are martingales recalled below.

Let $(\Omega, \mc F, \bb P)$ be a probability space, $(\mc F_t : t\ge
0)$ a filtration, and $(M_N: N\ge 1)$ a sequence of martingales
measurable with respect to the filtration.

\begin{lemma}
\label{D-l1}
Assume that for each $t\ge 0$, $M_N(t)$ converges in $L^1(\bb P)$ to a
random variable $M(t)$. Then, $M(t)$ is a martingale with respect to
the filtration $(\mc F_t : t\ge 0)$. 
\end{lemma}

\begin{proof}
Fix $0\le s< t$ and a bounded random variable $Y$, measurable with
respect to $\mc F_s$. Since $M_N$ is a martingale,
\begin{equation*}
\bb E\big[\, M_N(t)\, Y\,\big]\;=\; \bb E\big[\, M_N(s)\, Y\,\big]\;.
\end{equation*}
As $Y$ is bounded and $M_N(t)$, $M_N(s)$ converge in $L^1(\bb P)$ to 
$M(t)$, $M(s)$, respectively. The same identity holds with $M_N$
replaced by $M$. Moreover, $M(t)$, $M(s)$ belong to $L^1(\bb P)$. Since this
identity is in foce for all bounded random variable $Y$, $\bb E [\,
M(t)\,|\, \mc F_s \, ]\,=\, M(s)$, as claimed.
\end{proof}

Fix $k\in S$, a configuration $\eta_N$ in $\ms E^k_N$, and denote by
$\mb P^N$ the probability measure on $D([0,\infty), S)$ induced by the
process $\bs X^T_N(t)$ and the measure $\bb P^N_{\eta_N}$. The main
result of this section asserts that all limit points of the sequence
$\bs X^T_N(t)$ solve a martingale problem $(L,\delta_k)$ if we can
prove a \emph{\color{blue} local ergodic theorem} and calculate the
limit of the \emph{\color{blue} coarse-grained jump function},
properties (P1) and (P2) formulated at the end of this section.

Fix a function $F:S \to \bb R$.  As the trace process is a Markov
chain, \eqref{2-2} applied to the function $F\circ \Psi_N$ yields that
under $\bb P^{N}_{\eta_N}$
\begin{equation}
\label{3-2}
M_N(t) \:=\; F\big(\, \Psi_N (\eta^{\ms E_N}_{t\theta_N}) \,\big) 
\;-\; F\big(\, \Psi_N (\eta^{\ms E_N}_{0})\,\big) \;-\; \int_0^{t\theta_N} 
[\, \ms L_{\ms E_N} (F \circ \Psi_N)\, ] \, (\eta^{\ms
  E_N}_{s})\; ds
\end{equation}
is a martingale. In this formula, $\color{blue} \ms L_{\ms E_N}$
represents the generator of the trace process $\eta^{\ms E_N}(t)$.
Since $\bs X^T_N(t) = \Psi_N (\eta^{\ms E_N}(t\theta_N))$, changing
variables this expression becomes
\begin{equation*}
F\big(\, \bs X^T_N(t) \,\big) 
\;-\; F\big(\, \bs X^T_N(0) \,\big) \;-\; 
\int_0^t \, \theta_N\, 
[\, \ms L_{\ms E_N} (F \circ \Psi_N)\, ] \, (\eta^{\ms
  E_N}(s\theta_N) )\; ds\;.
\end{equation*}

Denote by $\color{blue} R^{T}_N (\eta,\xi)$ the jump rates of the
trace chain $\eta^{\ms E_N}(t)$. The expression inside of the integral
can be written as
\begin{equation*}
\theta_N\, \sum_{\xi\in\mc E_N} R^{T}_N (\zeta ,\xi) \,
\big\{\, (F \circ \Psi_N) \, (\xi) 
- (F \circ \Psi_N) \, (\zeta)  \big\}\quad\text{where}\quad 
\zeta \,=\, \eta^{\ms E_N}(s\theta_N)\;.  
\end{equation*}
Writing $\ms E_N$ as $\cup_\ell \ms E^\ell_N$, since $\Psi_N(\xi) = \ell$ for
$\xi\in \ms E^\ell_N$, this expression is equal to
\begin{align*}
& \theta_N\, \sum_{\ell\in S} \sum_{\xi\in\mc E^\ell_N} 
R^{T}_N \big(\, \eta^{\ms E_N}(s \theta_N) \,,\, \xi \, \big ) \,
\big\{\, F(\ell) \,-\, F(\bs X^T_N(s))  \,\big\} \\
&\quad =\;
\sum_{\ell\in S} R^{(\ell)}_N \big(\, \eta^{\ms E_N} (s\theta_N) \,
  \big)\, 
\big\{\, F(\ell) \,-\, F(\bs X^T_N(s))  \,\big\} \;,
\end{align*}
where $R^{(\ell)}_N(\zeta)$ represent the jump rate from the
configuration $\zeta$ to the set $\ms E^\ell_N$ for the trace process
speeded-up by $\theta_N$:
\begin{equation}
\label{3-1}
{\color{blue} R^{(\ell)}_N (\zeta)} \;=\; 
\theta_N\, \sum_{\xi\in\ms E^\ell_N} R^{T}_N (\zeta ,\xi) \;.
\end{equation}

Up to this point, we proved that the martingale $M_N(t)$ is equal to
\begin{equation*}
F\big(\, \bs X^T_N(t) \,\big) 
\;-\; F\big(\, \bs X^T_N(0) \,\big) \;-\; \int_0^{t} 
\sum_{\ell\in S} R^{(\ell)}_N \big(\, \eta^{\ms E_N} (s\theta_N) \,
\big) \,
\big\{\, F(\ell) \,-\, F(\bs X^T_N(s))  \,\big\} \; ds\;.
\end{equation*}
If the functions $R^{(\ell)}_N$ were constant over the sets
$\ms E^\ell_N$, $R^{(\ell)}_N (\eta) = r^{(\ell)}_N (\Psi_N(\eta))$
for some $r^{(\ell)}_N : S \to \bb R_+$, the martingale $M_N(t)$ could
be written in terms of the process $\bs X^T_N(s)$:
\begin{equation*}
M_N(t) \;=\; F (\bs X^T_N(t)) \;-\;  F(\bs X^T_N(0)) \;-\;
\int_0^t  \sum_{\ell\in S} r^{(\ell)}_N (\bs X^T_N(s)) \, 
\big\{\, F(\ell) \,-\, F(\bs X^T_N(s)) \,\big\} \; ds\;.
\end{equation*}
Furthermore, if for all $j\not = \ell\in S$, the sequences
$r^{(\ell)}_N (j)$ converged to some $\bs r(j,\ell) \in \bb R_+$, one
could replace in the previous formula $r^{(\ell)}_N (\bs X^T_N(s))$ by 
$\bs r (\bs X^T_N(s),\ell)$ at the cost of a small error.

Therefore, under the two previous conditions, up to a negligible
error,
\begin{equation}
\label{3-5}
F (\bs X^T_N(t)) \;-\;  F(\bs X^T_N(0)) \;-\;
\int_0^t  \sum_{\ell\in S} \bs r (\bs X^T_N(s), \ell) \, 
\big\{\, F(\ell) \,-\, F(\bs X^T_N(s)) \,\big\} \; ds
\end{equation}
is a martingale.  

Denote by $\mb P$ a limit point of the sequence $\mb P^N$. Let $X(t)$
represent the coordinate process of $D([0,\infty), S)$:
\begin{equation*}
{\color{blue} X(t, \omega)} \;=\; \omega(t)\;, \quad \omega\,\in\, D([0,\infty),
S)\;, \quad t\ge 0\;.
\end{equation*}
Assume that $\mb P [\, X(t-) \,=\, X(t)\,] = 1$ for all $t>0$, where
$\color{blue} X(t-) = \lim_{s<t \,,\, s\to t} X(s)$.

Suppose, without loss of generality, that $\mb P^N$ converges to
$\mb P$.  Let $L$ be the generator of the $S$-valued Markov chain
associated to the jump rates $r$. As $\mb P [\, X(t-) \,=\, X(t)\,]=1$,
the finite-dimensional projections are continuous (cf. equation (13.3)
in \cite{Bil99}).  Thus, since the expression in \eqref{3-5} is
uniformly bounded, we may pass to the limit and conclude from Lemma
\ref{D-l1} that
\begin{equation*}
F (X(t)) \;-\;  F(X (0)) \;-\;
\int_0^t  (L\, F) (X (s))  \; ds\;
\end{equation*}
is a martingale under the measure $\mb P$. 
Moreover, as $\eta_N \in \ms E^k_N$, $\mb P^N[X(0)=k]=1$ for all $N$
so that $\mb P[X(0)=k]=1$. Therefore, $\mb P$ is a solution of the
$(L,\delta_k)$ martingale problem. By Theorem \ref{3-l1}, this
property characterizes $\mb P$, and under this measure the coordinate
process is a continuous-time Markov chain whose generator is $L$.

We summarize the conclusions of the previous analysis in Theorem
\ref{3-t1} below. We first formulate the main hypotheses.

\begin{itemize}
\item[(P1)] (Local ergodicity). The mean rate functions
  $R^{(\ell)}_N (\eta)$, introduced in \eqref{3-1}, can be replaced by
  coarse-grained functions $r^{(\ell)}_N$. More precisely, there exist
  sequences of functions $r^{(\ell)}_N: \ms E_N \to \bb R_+$,
  $\ell\in S$, which are constant on the sets $\ms E^j_N$, $j\in S$,
  and such that for every function $F: S \to \bb R$, $t>0$, and
  sequence $\eta_N\in \ms E_N$,
\begin{equation}
\label{3-4}
\lim_{N\to\infty} \bb E_{\eta_N} 
\Big[ \int_0^t  F(\bs X^T_N(s)) \, 
\big\{\, R^{(\ell)}_N \big(\, \eta^{\ms E_N}(s \theta_N)\,\big)  
\,-\,  r^{(\ell)}_N (\bs X^T_N(s)) \,\big\} \; ds\; \Big] \;=\; 0\;.
\end{equation}

\item[(P2)] (The coarse-grained jump rates). The sequence of functions
  $r^{(\ell)}_N$, $\ell\in S$, called the \emph{\color{blue}
    coarse-grained jump functions}, converge. More precisely, since
  these functions are constant over the valleys $\ms E^k_N$, they can
  be written as
  \begin{equation*}
    r^{(\ell)}_N (\eta)\;=\; \sum_{k\in S} \bs r_N(k,\ell)\; 
    \chi_{\ms E^k_N} (\eta) 
  \end{equation*}
  for some non-negative real numbers $\bs r_N(k,\ell)$, named the
  \emph{\color{blue} coarse-grained jump rates}.  Note from the
  formula for the martingale $M_N(t)$ that the values of $\bs
  r_N(\ell,\ell)$ are unimportant. We assume that these rates
  converge: There exist $\bs r(j,\ell) \in [0,\infty)$, such that for
  all $j\not = \ell \in S$,
\begin{equation}
\label{3-3}
\lim_{N\to\infty} \bs r_N (j,\ell) \;=\; \bs r(j,\ell) \;.
\end{equation}
\end{itemize}

\begin{theorem}
\label{3-t1}
Fix $k\in S$, a sequence $\eta_N\in \ms E^k_N$, and denote by
$\mb P^N$ the probability measure on $D([0,\infty), S)$ induced by the
process $\bs X^T_N(t)$ and the measure $\bb P^N_{\eta_N}$. Assume that
conditions \eqref{3-4} and \eqref{3-3} are in force. Then, every limit
point $\mb P$ of the sequence $\mb P^N$ such that
\begin{equation*}
\mb P\big[\, X(t-) \,=\, X(t)\,\big] \;=\; 1 \quad\text{for all}\;\; t>0\;.
\end{equation*}
solves the $(L,\delta_k)$ martingale problem, where $L$ is the
generator of the $S$-valued Markov chain whose jump rates are
$\bs r(j,\ell)$.
\end{theorem}

Note that we do not need to prove property (P1) with an absolute value
inside the expectation. This observation simplifies considerably the
proof of this replacement.

We present in Sections \ref{sec6}, \ref{sec7} sufficient conditions,
formulated in terms of the stationary state and of capacities between
the sets $\ms E^j_N$, for conditions (P1), (P2) to hold. In Sections
\ref{sec5}, \ref{sec11} we propose alternative proofs of the
uniqueness of limit points for the sequence $\mb P^N$.

\section{Local ergodicity}
\label{sec6}

In this section, we provide sufficient conditions, formulated in terms
of the stationary state and of capacities, to replace the
jump rates $R^{(k)}_N$, introduced in \eqref{3-1}, by coarse-grained
jump functions which are constant on each set $\ms E^k_N$. We assume that
the reader is acquainted with the results on potential theory of
Markov chains, recapitulated in Section \ref{sec10}

Recall from \eqref{1-5} the definition of the hitting time $H_{\ms A}$
and the return time $H^+_{\ms A}$ of a subset $\ms A$ of $E_N$.  For
two non-empty, disjoint subsets $\ms A$, $\ms B$ of $E_N$, denote by
$\Cap_N(\ms A, \ms B)$ the capacity between $\ms A$ and $\ms B$:
\begin{equation}
\label{5-5}
{\color{blue}\Cap_N(\ms A, \ms B)} \;=\; \sum_{\eta\in \ms A} \pi_N(\eta)\,
\lambda_N(\eta)\, \bb P^N_{\eta} \big[\, H_{\ms B} < H^+_{\ms
  A}\,\big]\;, 
\end{equation}
where $\lambda_N(\eta)$ stands for the holding rate at $\eta$ of the
Markov chain $\eta_N(t)$:
$\color{blue} \lambda_N(\eta) = \sum_{\xi\in E_N} R_N(\eta,\xi)$. 

Recall that 
\begin{equation*}
R^{(k)}_N (\eta) \;=\; \theta_N\, \sum_{\xi\in\ms E^k_N} 
R^{T}_N (\eta ,\xi)   \;,
\end{equation*}
where $R^{T}_N (\eta ,\xi)$ represents the jump rates of the trace
process. Thus, $R^{(k)}_N (\eta)$ is the rate at which the trace
process jumps from $\eta$ to $\ms E^k_N$ multiplied by $\theta_N$.
In view of equation \eqref{1-1} for the jump rates of the trace
process, 
\begin{equation}
\label{5-1}
R^{(k)}_N (\eta) \;=\; \theta_N\, \lambda_N(\eta) \, 
\, \bb P_\eta \big[\, H^+_{\ms E_N} =
H_{\ms E^k_N} \, \big] \;=\; \theta_N\, \lambda_N(\eta) \, 
\, \bb P_\eta \big[\, H_{\ms E^k_N} < H^+_{\breve{\ms E}^k_N} \,
\big]  \;,
\end{equation}
where
\begin{equation*}
{\color{blue}\breve{\ms E}^k_N} \;:=\; \bigcup_{j\not = k} \ms E^j_N
\;, \quad k\, \in\, S \;.
\end{equation*}
In particular, $R^{(k)}_N$ vanishes in the interior of the sets
$\ms E^j_N$, where by interior we mean the set of configurations in
$\ms E^j_N$ whose neighbors belong to $\ms E^j_N$ [the configuration
$\xi\in \ms E^j_N$ such that
$\sum_{\zeta\not\in \ms E^j_N} R_N(\xi,\zeta)=0$]. This means that
$R^{(k)}_N$ is a singular function. While it vanishes in the interior
of the sets $\ms E^j_N$, it assumes a large value at the boundary
because the right-hand side of \eqref{5-1} is multiplied by
$\theta_N$.

The goal of this section is to replace the time integral of the
singular function $R^{(k)}_N$ by the time integral of a very regular
function, one which is constant at each set $\ms E^j_N$. This
replacement is expected to hold whenever the process equilibrates in
the valleys $\ms E^j_N$ before it jumps to a new one.

Let $f_N: \ms E_N \to \bb R$ be a sequence of real functions defined
on $\ms E_N$. Fix $t>0$, and consider the time integral
\begin{equation*}
\int_0^t f_N\big(\, \eta^{\ms E_N}(s\theta_N)\,\big)\; ds \;=\;
\frac 1{\theta_N} \int_0^{t\theta_N}  f_N \big(\, \eta^{\ms E_N}(s)\,\big)\;
ds \;. 
\end{equation*}
The time integral can be decomposed according to the sojourns in the
sets $\ms E^j_N$. If the process equilibrates during these visits, by
the ergodic theorem, we expect the integral of $f_N$ over these
time-intervals to be close to the integral of the mean value of
$f_N$ on these sets. Hence, let
\begin{equation}
\label{6-1}
\widehat f_N (\eta) \;=\; E_{\pi_{\ms E}} \big[\, f_N \,|\, \mc G_N
\,\big]\;, 
\end{equation}
where $\color{blue} \mc G_N$ represents the $\sigma$-algebra of
subsets of $\ms E_N$ generated by the sets $\ms E^j_N$, $j\in S$, and
$\color{blue} \pi_{\ms E}$ the stationary state of the trace process
$\eta^{\ms E_N}(t)$ [which, by \eqref{1-4}, is the stationary state
$\pi_N$ conditioned to $\ms E_N$].

Clearly,
\begin{equation*}
\widehat f_N (\eta) \;=\; \sum_{j\in S} F_N(j)\; \chi_{\ms E^j_N}
(\eta) \;, \quad\text{where}\quad
F_N(j) \;=\; \frac {1}{\pi_N(\ms E^j_N)} 
\sum_{\zeta\in \ms E^j_N} \pi_N(\zeta) \, f_N (\zeta) \;.
\end{equation*}
The function $\widehat f_N $ is the candidate, and one expects that,
under certain conditions on the sequence $f_N$,
\begin{equation*}
\frac 1{\theta_N} \int_0^{t\theta_N}  \Big\{ f_N \big(\, \eta^{\ms
  E_N}(s)\,\big) \,-\, 
\widehat f_N \big(\, \eta^{\ms E_N}(s)\,\big) \, \Big\} \; ds 
\end{equation*}
vanishes as $N\to\infty$. 

\begin{theorem}
\label{l6-1}
Let $f_N$, $g_N : \ms E_N \to \bb R$ be sequences of functions such that
\begin{itemize}
\item[(a)] For each $j\in S$, there exists a configuration $\xi^{j,N}$
  such that
\begin{equation*}
\lim_{N\to\infty} \, \frac 1{\theta_N}\, \max_{\eta\in \ms E^j_N} 
\frac{\sum_{\zeta\in \ms E^j_N} |\, f_N(\zeta)\,| \, \pi_N(\zeta)}
{\Cap_N (\eta, \xi^{j,N})} \;=\; 0\;,
\end{equation*}
where the supremum is carried over all configurations $\eta\not =
\xi^{j,N}$. 
\item[(b)] The sequence $g_N$ is uniformly bounded and is constant
  over each set $\ms E^j_N$: There exist a finite constant $C_0$ and a
  sequence of functions $G_N: S\to \bb R$ such that
\begin{equation*}
g_N(\eta) \;=\; \sum_{j\in S} G_N(j)\; \chi_{\ms E^j_N} (\eta) 
\quad\text{and}\quad
\max_{\eta\in \ms E_N} \big|\, g_N(\eta)\,\big| \;\le\; C_0   
\end{equation*}
for all $N\ge 1$.
\end{itemize}
Then, for all $t>0$,
\begin{equation*}
\lim_{N\to\infty} \max_{\eta\in \ms E_N} \, \Big|\, \bb E^N_\eta \Big[\,
\int_0^{t}  \Big\{ f_N \big(\, \eta^{\ms E_N}(s\theta_N) \,\big) \,-\, 
\widehat f_N \big(\, \eta^{\ms E_N}(s\theta_N) \,\big) \, \Big\} 
\, g_N \big(\, \eta^{\ms E_N}(s\theta_N)\,\big) \; ds \, \Big] \,
\Big| \;=\; 0\;.
\end{equation*}
\end{theorem}

In the reversible case, this result follows from Corollary 6.5 and
Proposition 6.10 in \cite{bl2} and from the hypotheses of the theorem.
In the nonreversible case, it follows from Corollary 6.5 in \cite{bl2} and
Proposition A.2 in \cite{bl7}.

\begin{remark}
\label{5-rm2}
The proof of this result takes advantage of the fact that the absolute
value is outside of the expectation.
\end{remark}

\begin{remark}
\label{rm6-2}
To turn the martingale $M_N(t)$, introduced in the previous section,
into a function of $X^T_N$, we only need to prove the previous theorem
for $f_N = R^{(j)}_N$, $j\in S$.  In this special case, by \eqref{5-5}
and \eqref{5-1},
\begin{equation*}
\frac 1{\theta_N}\, \sum_{\zeta\in \ms E^j_N} |\, f_N(\zeta)\,| \,
\pi_N(\zeta) \;=\; \Cap_N(\ms E^j_N, \breve{\ms E}^j_N)\;.
\end{equation*}
In particular, condition (a) of the theorem becomes that for all $j\in
S$, there exists $\xi^{j,N} \in \ms E^j_N$ such that
\begin{equation}
\label{6-2}
\lim_{N\to\infty} \max_{\eta\in \ms E^j_N \,,\, \eta\not = \xi^{j,N}} 
\frac{\Cap_N(\ms E^j_N, \breve{\ms E}^j_N)}{\Cap_N (\eta, \xi^{j,N})}
\;=\; 0\;.
\end{equation}
\end{remark}

\begin{remark}
\label{rm6-1}
The configuration $\xi^{j,N}$ has no special role. By Theorem 2.7 in
\cite{bl2}, if condition \eqref{6-2} holds for one configuration in
$\ms E^j_N$, it holds for all.
\end{remark}

The coarse-grained jump function, denoted by $r^{(j)}_N$ in \eqref{3-3}, is
given by $\widehat R^{(j)}_N$. Thus, by \eqref{6-1} and \eqref{5-1},
\begin{equation*}
r^{(j)}_N (\eta)\;=\; \widehat R^{(j)}_N (\eta)
\;=\; E_{\pi_{\ms E}} \big[\, R^{(j)}_N \,|\, \mc G_N
\,\big] \;=\; \sum_{k\in S} \bs r_N(k,j)\; 
\chi_{\ms E^k_N} (\eta) \;,
\end{equation*}
where, for $k\not =j$,
\begin{equation}
\label{7-1}
\begin{aligned}
\bs r_N(k,j) \; &=\; \frac {1}{\pi_N(\ms E^k_N)} 
\sum_{\zeta\in \ms E^k_N} \pi_N(\zeta) \,  R^{(j)}_N (\zeta) \\
\; &=\;  \frac {\theta_N}{\pi_N(\ms E^k_N)} 
\sum_{\zeta\in \ms E^k_N} \pi_N(\zeta) \, \lambda_N(\zeta) \, 
\bb P_\zeta \big[\, H_{\ms E^j_N} < H^+_{\breve{\ms E}^j_N} \,\big]\;.
\end{aligned}
\end{equation}

\begin{remark}
\label{5-rm1}
Hypothesis (a) of Theorem \ref{l6-1} requires the process to visit all
configurations of the valley $\ms E^j_N$ before it reaches a new
one. Dynamics which display this behavior are said to ``visit
points''.  This class includes condensing zero-range processes
\cite{bl3, l2014, agl2017, s2018}, random walks in a potential field
\cite{lmt2015, ls2016, ls2018} or models in which the valleys are
singletons as the inclusion process \cite{bdg17} or random walks
evolving among random traps \cite{fm1, fli1, jlt11, jlt14}, but it does not
contain the example of Section 1. For such dynamics, in which the
entropy plays a role in the metastable behavior, a different approach
is needed. This is discussed in Sections \ref{sec5} and \ref{sec11}.
\end{remark}

\section{The coarse-grained jump rates}
\label{sec7}

In this section, we investigate the asymptotic behavior of the
coarse-grained jump rates $\bs r_N(k,j)$, defined in \eqref{7-1}.
This is condition \eqref{3-3} of Theorem \ref{3-t1}.

\subsection{Reversible case}
In the reversible case, we may express the coarse-grained jump rates
$\bs r_N(j,k)$ in terms of capacities. If follows from the explicit
formulae \eqref{5-5}, \eqref{7-1} and from an elementary argument
taking advantage of the reversibility that
\begin{equation*}
\begin{aligned}
& \pi_N(\ms E^j) \, \bs r_N(j,k) \; \\
&\quad =\; \theta_N\,
\frac{1}{2} \, \Big\{\,\Cap_N(\ms E^j, \breve{\ms E}^j) \,+\, \Cap_N(\ms
E^k, \breve{\ms E}^k) \,-\,  \Cap_N\big({\ms E}^j\cup{\ms E}^k\,,\,
\cup_{\ell\not = j,k} {\ms E}^\ell \big)\,\Big\} \;.
\end{aligned}
\end{equation*}
Here and below we often write $\color{blue}\ms E^j$,
$\color{blue}\breve{\ms E}^j$ for $\ms E^j_N$, $\breve{\ms E}^j_N$,
respectively.  Therefore, in the reversible case, one can compute the
limit of the coarse-grained jump rates $\bs r_N(j,k)$ if one can
calculate the asymptotic behavior of $\pi_N(\ms E^j)$ and of
\begin{equation*}
\theta_N \, \Cap_N \big( \, \cup_{j\in A} \ms E^j
\,,\, \cup_{k\in B} \ms E^k\,\big)
\end{equation*}
for non-empty subset $A$, $B$ of $S$ such that
$A\cap B = \varnothing$, $A\cup B = S$.

\subsection{Nonreversible case}  

Summing over $j\not =k$ in \eqref{7-1} provides a formula for the
coarse-grained holding rates, denoted by $\bs\lambda_N(j)$:
\begin{equation}
\label{7-10}
{\color{blue} \bs\lambda_N(k)} \;:=\; \sum_{j\not = k} \bs r_N(k,j) \;=\;
\frac {\theta_N}{\pi_N(\ms E^k)} 
\sum_{\zeta\in \ms E^k} \pi_N(\zeta) \, \lambda(\zeta)
\, \bb P_\zeta \big[\, H_{\breve{\ms E}^k} < H^+_{\ms E^k} \, \big]\;.
\end{equation}
The expression on the right-hand side corresponds to the capacity
between $\ms E^k$ and $\breve{\ms E}^k$. Therefore,
\begin{equation}
\label{7-2}
\pi_N(\ms E^j)\, \bs\lambda_N(j) \;=\; 
\theta_N \, \Cap_N(\ms E^j, \breve{\ms E}^j) \;, \quad
j\in S \;.
\end{equation}

\begin{remark}
\label{r7-1}
Equation \eqref{7-2} provides a formula for the magnitude of the
scaling parameter $\theta_N$. To derive a non-trivial limit for the
coarse-grained model $\bs X^T_N$, time has to be rescaled by the
inverse of the capacity between the sets $\ms E^j$ and $\breve{\ms
  E}^j$:
\begin{equation*}
\theta_N \;\approx \; \frac{\pi_N(\ms E^j)}  
{\Cap_N(\ms E^j, \breve{\ms E}^j)} \;\cdot
\end{equation*}
\end{remark}

The asymptotic behavior of the coarse-grained holding rates can be
computed through formula \eqref{7-2} provided one can estimate the
capacities and the measures of the valleys. Once this has been done,
to compute the jump rates, it remains to estimate the jump
probabilities.

Recall from Section \ref{sec13} the definition of a collapsed
chain. Fix $j\in S$, and denote by $\color{blue}\eta^{C,j}(t)$ the
Markov chain obtained from the chain $\eta_N(t)$ by collapsing the
valley $\ms E^j$ to a point, denoted by $\mf j$. The chain
$\eta^{C,j}(t)$ takes value in ${\color{blue}E^{C,j}_N} :=
(E_N \setminus \ms E^j) \cup \{\mf j\}$.

Let $\color{blue}\bb P^{C,j}_\eta$, $\eta\in E^{C,j}_N$, be the
probability measure on $D([0,\infty), E^{C,j}_N)$ induced by the
collapsed process $\eta^{C,j}(t)$ starting from $\eta$. Expectation
with respect to $\bb P^{C,j}_\eta$ is represented by $\color{blue} \bb
E^{C,j}_\eta$. By the last formula of the proof of \cite[Proposition
3.4]{bl7}, for any $k\in S$, $k\not = j$,
\begin{equation*}
\bs p_N(j,k) \;:=\; \frac{\bs r_N(j,k)}{\bs \lambda_N(j)}\; =\; 
\bb P^{C,j}_{\mf j} \big[ \, H_{\ms E^k} < H_{\breve{\ms E}^{j,k}}\,
\big] \;, \quad \text{where}\quad
{\color{blue} \breve{\ms E}^{j,k}} := \bigcup_{\ell \in S
  \setminus \{j,k\}} \ms E^\ell_N\;.
\end{equation*}

\smallskip 

Denote by $\color{blue}\mb P_j$ the probability measure on
$D([0,\infty), S)$ induced by the reduced model $\bs X(t)$ starting
from $j$.  We present below a set of sufficient conditions which
ensure that $\bs p_N(j,k)$ converges to
$\mb P_j[H_k < H_{S\setminus \{j,k\}}]$. This approach has been
developed and gradually refined in \cite{l2014, ls2016, s2018}, and it
is based on the premise that the capacities can be calculated through
the Thomson and the Dirichlet principles.

Denote by $\color{blue} L^2(\pi_N)$ the space of square-summable
functions $f: E_N \to \bb R$ endowed with the scalar product
$\<\,\cdot\,,\,\cdot\, \>_{\pi_N}$ given by
\begin{equation*}
{\color{blue} \<\, f \,,\, g\, \>_{\pi_N}} 
\;:=\; \sum_{\eta\in E_N} f(\eta)\, g(\eta)\, \pi_N(\eta)\;.
\end{equation*}
We assume that the generator $\ms L_N$ of the Markov chain $\eta_N(t)$
satisfies a \emph{\color{blue} sector condition} with a constant $C_0$
independent of $N$: For every $f$, $g\in L^2(\pi_N)$,
\begin{equation}
\label{7-4}
\<\, \ms L_N f \,,\,  g\,\>_{\pi_N}^2 \;\le\; C_0 \, 
\<\, (-\ms L_N) f \,,\,  f\,\>_{\pi_N} \,
\<\, (-\ms L_N) g \,,\,  g\,\>_{\pi_N}\;. 
\end{equation}

Suppose that for fixed $j$, $k\in S$, $k\not = j$, 
\begin{equation}
\label{7-5}
\begin{aligned}
& \lim_{N\to\infty} \theta_N\, \Cap_N(\ms E^k, \breve{\ms E}^{j,k}) 
\,=\, \Cap_S(k, S\setminus \{j,k\})\;, \\
& \quad
\lim_{N\to\infty} \theta_N\, \Cap_N(\ms E^j, \breve{\ms E}^{j}) 
\,=\, \Cap_S(j, S\setminus \{j\})\;,
\end{aligned}
\end{equation}
where $\color{blue} \Cap_S(A,B)$, $A$, $B\subset S$, represents the
capacity with respect to the reduced model $\bs X(t)$.

We also assume that the capacities for the collapsed process
$\eta^{C,j}(t)$ can be calculated: Denote by
$\color{blue}\Cap^{C,j}_N(\ms A, \ms B)$ the capacity between $\ms A$,
$\ms B\subset E^{C,j}_N$, $\ms A \cap \ms B = \varnothing$ induced by
the collapsed process $\eta^{C,j}(t)$.  We assume that the limit of
the capacity $\Cap^{C,j}_N(\ms E^k, \breve{\ms E}^{j,k})$ coincides
with $\Cap_S(k, S\setminus \{j,k\})$:
\begin{equation}
\label{7-6}
\lim_{N\to\infty} \theta_N\, \Cap^{C,j}_N(\ms E^k, \breve{\ms E}^{j,k}) 
\,=\, \Cap_S(k, S\setminus \{j,k\})\;.
\end{equation}

The computation of the capacities requires test flows or test
functions which approximate the optimal ones in the variational
principles. It is thus implicitly assumed in hypotheses
\eqref{7-5} and \eqref{7-6} that explicit expressions for such flows
or functions are available. We assume below that there exists a
sequence of functions ${\color{blue} V^N_{j,k}}: \ms E_N \to [0,1]$
close to the \emph{\color{blue} equilibrium potential}
$h^N_{\ms E^k, \breve{\ms E}^{j,k}}$, given by
\begin{equation*}
{\color{blue} h^N_{\ms E^k, \breve{\ms E}^{j,k}} (\eta)} \;:=\; 
\bb P^N_\eta \big[\, H_{\ms E^k} < H_{\breve{\ms   E}^{j,k}} 
\,\big]\;,
\end{equation*}
in the sense that
\begin{equation}
\label{7-8}
\lim_{N\to\infty} \theta_N\, D_N(V^N_{j,k}) \;=\;
\lim_{N\to\infty} \theta_N\, 
D_N \big( h^N_{\ms E^k, \breve{\ms E}^{j,k}}\big) \;=\;
\Cap_S(k, S\setminus\{j,k\}) \;,
\end{equation}
where $D_N(f)$ stands for the Dirichlet form of $f$:
\begin{equation*}
{\color{blue}  D_N(f)} \;:=\; \<\, (-\, \ms L_N) \, f \,,\,  f\,\>_{\pi_N} \;.
\end{equation*}
The last identity in \eqref{7-8} follows from the fact, proved in
\eqref{J-3}, that
$D_N ( h^N_{\ms E^k, \breve{\ms E}^{j,k}}) = \Cap_N(\ms E^k,
\breve{\ms E}^{j,k})$ and from assumption \eqref{7-5}.

We assume, furthermore, that $V^N_{j,k}$ is constant in each valley
$\ms E^\ell_N$:
\begin{equation}
\label{7-7}
V^N_{j,k}(\eta) \;=\; \sum_{\ell \in S} \mb P_\ell[H_k < H_{S\setminus
  \{j,k\}}] \, \chi_{\ms E^\ell}(\eta)
\quad\text{for all} \quad \eta\in \ms E_N \;.
\end{equation}
Hence, $V^N_{j,k}$ is equal to $1$, $0$ in $\ms E^k_N$,
$\breve{\ms E}^{j,k}_N$, respectively, while on $\ms E^j_N$ it is
given by the probability appearing in \eqref{7-7}.

Finally, as $V^N_{j,k}$ approximates
$h^N_{\ms E^k, \breve{\ms E}^{j,k}}$, which is harmonic on
$\Delta_N \cup \ms E^j$, it is also reasonable to require
$\ms L_N \, V^N_{j,k}$ to be small in these sets. We assume that
\begin{equation}
\label{7-9}
\lim_{N\to\infty} \theta_N\,
\sum_{\eta\in \Delta_N} \big|\, (\ms L_N \, V^N_{j,k}) \, (\eta) \, \big|
\, \pi_N(\eta) \;=\; 
\lim_{N\to\infty} \theta_N\, \Big| \, \sum_{\eta\in \ms E^j}
(\ms L_N \, V^N_{j,k}) \, (\eta) \, \pi_N(\eta) \,\Big| \;=\; 0\;.
\end{equation}

\begin{proposition}
\label{l7-1}
Fix $j$, $k\in S$, $k\not = j$, and assume that conditions
\eqref{7-4}--\eqref{7-9} are in force.  Then, 
\begin{equation*}
\lim_{N\to\infty}  \bs p_N(j,k) 
\;=\; \mb P_j[H_k < H_{S\setminus \{j,k\}}] \;.
\end{equation*}
\end{proposition}

The proof of this proposition is divided in several lemmata.
Since $V^N_{j,k}$ is constant on $\ms E^j$, we may collapse it to a
function defined on $E^{C,j}_N$. Recall
from \eqref{7-7} the value of $V^N_{j,k}$ at $\ms E^j$ and let 
$V^{C,j}_{j,k}: E^{C,j}_N \to [0,1]$ be given by
\begin{equation*}
V^{C,j}_{j,k} (\mf j) \;=\; \mb P_j[H_k < H_{S\setminus
  \{j,k\}}] \;, \quad
V^{C,j}_{j,k} (\eta) \;=\; V^N_{j,k} (\eta) \;,\quad
\eta\in E^{C,j}_N \setminus \{\mf j\}\;.
\end{equation*}
The dependence of $V^{C,j}_{j,k}$ on $N$ has been omitted.

Let $\color{blue}\ms L^{C,j}_N$ be the generator of the collapsed
process $\eta^{C,j}(t)$.  For $\ms A$, $\ms B\subset E^{C,j}_N$,
$\ms A \cap \ms B = \varnothing$, denote by
$h^{(\mf j)}_{\ms A, \ms B}$ the solution of the boundary value
elliptic problem
\begin{equation*}
\left\{
\begin{aligned}
& \; (\, \ms L^{C,j}_N \, h \,)\, (\eta) \;=\;
0\;,\;\; \eta\not \in \ms A \cup \ms B \;, \\
& \; h (\eta) \;=\; \chi_{\ms A} (\eta) 
\;,\;\; \eta \,\in\, \ms A \cup \ms B \;.
\end{aligned}
\right.
\end{equation*}

Denote by $h_{j,k}: E_N \to \bb R$ the lifting of the function
$h^{(\mf j)}_{\ms E^k, \breve{\ms E}^{j,k}}$:
\begin{equation*}
h_{j,k}(\eta) \;=\; h^{(\mf j)}_{\ms E^k, \breve{\ms E}^{j,k}}(\mf j)
\;, \quad \eta\in \ms E^j\;,\qquad
h_{j,k}(\eta) \;=\; h^{(\mf j)}_{\ms E^k, \breve{\ms E}^{j,k}}(\eta) \;,\quad
\eta\in E_N \setminus \ms E^j\;.
\end{equation*}
Note that the function $h_{j,k}$ is constant and equal to
$\bs p_N(j,k)$ on the set $\ms E^j$.

\begin{lemma}
\label{l7-2}
We claim that
\begin{equation*}
\lim_{N\to\infty}   \theta_N\, D_N\big( \,h_{j,k}
\,-\, V^N_{j,k} \,\big)  \;=\; 0 \;.
\end{equation*}
\end{lemma}

\begin{proof}
Rewrite $D_N( h - V)$ as $\< \, h-V \,,\, (-\ms L_N\, [h-V]\,)\,
\>_{\pi_N}$ and compute separately the limit of the four
terms. 

By equation \eqref{L-3},
$D_N(h_{j,k}) = D_N^{C, j}( h^{(\mf j)}_{\ms E^k, \breve{\ms
    E}^{j,k}})$,
where $D_N^{C, j}$ represents the Dirichlet form associated to the
collapsed process.  By \eqref{J-3},
$D_N^{C, j}( h^{(\mf j)}_{\ms E^k, \breve{\ms E}^{j,k}}) = \Cap^{C,
  j}_N(\ms E^k, \breve{\ms E}^{j,k})$.
Hence, by assumption \eqref{7-5},
\begin{equation*}
\lim_{N\to\infty} 
\theta_N\, D_N(h_{j,k}) \;=\; 
\Cap_S(k, S\setminus \{j,k\}) \;.
\end{equation*}
By assumption \eqref{7-8}, the same result holds for $V_{j,k}$ in place of
$h_{j,k}$. 

It remains to examine the cross terms. By \eqref{L-3},
\begin{equation*}
\< \, V_{j,k} \,,\, (- \ms L_N \, h_{j,k}) \, \>_{\pi_N} \;=\;
\< \, V^{C, j}_{j,k} \,,\, (- \ms L^{C, j}_N \, 
h^{(\mf j)}_{\ms E^k, \breve{\ms E}^{j,k}}) \, \>_{\pi^{C, j}_N}\;,
\end{equation*}
where $\pi^{C, j}_N$ stands for the stationary measure $\pi_N$
collapsed at $\ms E^j$. Since
$h^{(\mf j)}_{\ms E^k, \breve{\ms E}^{j,k}}$ is harmonic on
$\Delta_N \cup \{\mf j\}$, and since $V^{C, j}_{j,k}$ vanishes
on $\breve{\ms E}^{j,k}$ and coincides with
$h^{(\mf j)}_{\ms E^k, \breve{\ms E}^{j,k}}$ on $\ms E^k$, the last
expression is equal to
\begin{align*}
& \sum_{\eta\in \ms E^k} V^{C, j}_{j,k} (\eta)\, 
(- \ms L^{C, j}_N \, h^{(\mf j)}_{\ms E^k, \breve{\ms E}^{j,k}}) 
\, (\eta)\, \pi^{C, j}_N(\eta) \\
&\quad 
\; =\; \sum_{\eta\in \ms E^k} h^{(\mf j)}_{\ms E^k, \breve{\ms E}^{j,k}}  (\eta)\, 
(- \ms L^{C, j}_N \, h^{(\mf j)}_{\ms E^k, \breve{\ms E}^{j,k}})
  \, (\eta)\, \pi^{C, j}_N(\eta)\;.
\end{align*}
Using again the harmonicity of
$h^{(\mf j)}_{\ms E^k, \breve{\ms E}^{j,k}}$ on
$\Delta_N \cup \{\mf j\}$, and the fact that it vanishes on
$\breve{\ms E}^{j,k}$, we may extend the sum to the entire set
$E^{C, j}_N$ and conclude, as at the beginning of the
proof, that
\begin{equation*}
\lim_{N\to\infty} 
\theta_N\, \< \,V_{j,k} \,,\, (- \ms L_N \, h_{j,k})
\,\>_{\pi_N} \;=\;  \Cap_S(k, S\setminus \{j,k\}) \;.
\end{equation*}

Similarly, since $h_{j,k}$ is equal to
$V_{j,k}$ on $\ms E^k \cup \breve{\ms E}^{j,k}$,
\begin{align*}
\< \, h_{j,k} \,,\, (- \ms L_N V_{j,k} \,)
\, \>_{\pi_N} \; & =\; \sum_{\eta\in \ms E^k} V_{j,k}(\eta)\,
(- \ms L_N \, V_{j,k}) \, (\eta)\, \pi_N(\eta) \\
& +\; \sum_{\eta\in \Delta_N \cup \ms E^j} h_{j,k}  (\eta)\, 
(- \ms L_N \, V_{j,k}) \, (\eta)\, \pi_N(\eta)\;.
\end{align*}
Since $V_{j,k}$ vanishes on $\breve{\ms E}^{j,k}$, the first term on
the righ-hand side is equal to
\begin{equation*}
D_N(V_{j,k}) \;-\;
\sum_{\eta\in \Delta_N \cup \ms E^j}  V_{j,k} (\eta)\, 
(- \ms L_N \, V_{j,k}) \, (\eta)\, \pi_N(\eta)\;.
\end{equation*}
Therefore,
\begin{align*}
& \< \, h_{j,k} \,,\, (- \ms L_N V_{j,k} \,)
\, \>_{\pi_N} \;=\;  D_N(V_{j,k}) \\
&\quad  +\; \sum_{\eta\in \Delta_N \cup \ms E^j} 
\{\, h_{j,k}  (\eta) \,-\, V_{j,k} (\eta) \} \, 
(- \ms L_N \, V_{j,k}) \, (\eta)\, \pi_N(\eta)\;.
\end{align*}
Since $h_{j,k}$ and $V_{j,k}$ are constant in
$\ms E^j$, non-negative and bounded by $1$, the absolute value of the
second term on the right-hand side is less than or equal to
\begin{equation*}
\sum_{\eta\in \Delta_N} \big|\, (\ms L_N \, V_{j,k}) \, (\eta) \, \big|
\, \pi_N(\eta) \;+ \; \Big| \, \sum_{\eta\in \ms E^j}
(\ms L_N \, V_{j,k}) \, (\eta) \, \pi_N(\eta) \,\Big|\;.
\end{equation*}
By condition \eqref{7-9}, this expression multiplied by $\theta_N$
converges to $0$ as $N\to\infty$. Thus, by \eqref{7-8},
\begin{equation*}
\lim_{N\to\infty}   \theta_N\,
\< \, h_{j,k} \,,\, 
(- \ms L_N V_{j,k} \,) \, \>_{\pi_N} \;=\; 
\Cap_S(k, S\setminus \{j,k\})\;.
\end{equation*}
Putting together all previous estimates yields the assertion.  
\end{proof}

Fix two non-empty subsets $\ms A$, $\ms B$ of $E_N$ such that
$\ms A \cap \ms B = \varnothing$, Recall from Section \ref{sec10} that
we represent by $\color{blue} \mf C_{1,0}(\ms A , \ms B)$ the space of
functions $f:E_N \to [0,1]$ which are equal to $1$ on $\ms A$ and $0$
on $\ms B$.  Let $f^N_{j,k} = h_{j,k} \,-\, V_{j,k}$, and note that
this function is constant on each valley $\ms E^\ell$.

\begin{lemma}
\label{l7-3}
Let $f_{j,k}(\ms E^j)$ be the value of $f_{j,k}$ at $\ms E^j$.  Then,
$\lim_{N\to\infty} f_{j,k}(\ms E^j) =0$.
\end{lemma}

\begin{proof}
The function $f_{j,k}$ vanishes on $\ms E^k \cup \breve{\ms E}^{j,k} =
\breve{\ms E}^{j}$, and it is constant on $\ms E^{j}$. Hence, if
$f_{j,k}(\ms E^j)\not = 0$, the function $F$ defined as $F(\eta) =
f_{j,k}(\eta)/f_{j,k}(\ms E^j)$ belongs to $\mf C_{1,0}(\ms E^j,
\breve{\ms E}^{j})$.

Suppose that $f_{j,k}(\ms E^j)\not =0$.  On the one hand, by Lemma
\ref{l7-2}, $\theta_N\, f_{j,k}(\ms E^j)^2 \, D_N(F) = \theta_N\,
D_N(f_{j,k}) \to 0$.  On the other hand, since $F$ belongs to $\mf
C_{1,0}(\ms E^j, \breve{\ms E}^{j})$, by \eqref{J-13}, $D_N(F) \ge
\Cap^s_N(\ms E^j, \breve{\ms E}^{j})$, where $\Cap^s_N(\ms E^j,
\breve{\ms E}^{j})$ represents the capacity associated to the
symmetric dynamics.  By the sector condition, stated in assumption
\eqref{7-4}, and Lemma \ref{J-l2}, this symmetric capacity is bounded
below by $c_0\, \Cap_N(\ms E^j, \breve{\ms E}^{j})$, where $c_0 =
1/C_0>0$. Hence, by \eqref{7-5}, $\liminf_{N\to \infty} \theta_N\,
D_N(F_N) \ge c_0\, \Cap_S(j, S\setminus \{j\}) > 0$, which proves the
assertion of the lemma.
\end{proof}

\begin{proof}[Proof of Proposition \ref{l7-1}]
By definition \eqref{7-7} of $V_{j,k}$, $V_{j,k}(\ms E^j) = \mb
P_j[H_k < H_{S\setminus \{j,k\}}]$. The claim of the proposition
follows from Lemma \ref{l7-3} and the fact that
\begin{equation*}
h_{j,k}(\ms E^j) \;=\; h^{(\mf j)}_{\ms E^k, \breve{\ms E}^{j,k}}(\mf
j) \;=\; \bs p_N(j,k)\;.
\end{equation*}
\end{proof}

\begin{remark}
\label{r7-2}
Assumption \eqref{7-4} can be replaced by the hypothesis that 
\begin{equation}
\label{7-3}
\liminf_{N\to\infty} \theta_N\, \Cap^s_N(\ms E^j, \breve{\ms E}^{j})
\;>\; 0\;.  
\end{equation}
\end{remark}

\begin{proof}
We only used the sector condition, assumption \eqref{7-4}, in the
proof of Lemma \ref{l7-3} to guarantee that $\theta_N\, D_N(F)$ is
bounded below by a strictly positive constant. Since
$D_N(F) \ge \Cap^s_N(\ms E^j, \breve{\ms E}^{j})$, by \eqref{7-3},
$\liminf_{N\to\infty} \theta_N\,D_N(F) >c_0$, as needed.
\end{proof}

\section{The negligible set $\Delta_N$}
\label{sec15}

We provide in this section sufficient conditions for assumption (LP2)
or (T2) to hold. Recall from \eqref{7-1} that $\bs r_N(k,j)$
represents the coarse-grained jump rates. Assume that they converge:
For all $j\not = k$, there exists $\bs r(k,j) \in [0,\infty)$ such
that
\begin{equation}
\label{o-2}
\lim_{N\to\infty} \bs r_N(k,j) \;=\; \bs r(k,j)\;.
\end{equation}
Recall that we represent by $\bs X(t)$ the reduced model, the
$S$-valued Markov chain whose jump rates are given by $\bs r(k,j)$.
Denote by $\color{blue} A\subset S$ the subset of $S$ formed by the
points which are absorbing for the reduced model $\bs X(t)$.  Next
result is Theorem 2.7 in \cite{bl2} and Theorem 2.1 in \cite{bl7}.

\begin{theorem}
\label{o-t1}
Assume that conditions \eqref{6-2} and \eqref{o-2} are in force.
Assume, furthermore, that for all $k\in A$, $t>0$, 
\begin{equation}
\label{o-3}
\lim_{N\to\infty}
\max_{\eta\in\ms E^k_N} \bb E^N_{\eta} \Big[\,
\int_0^t \chi_{\Delta_N} \big(\eta_N(s\theta_N)\big) 
\; ds\, \Big] \;=\; 0\;;
\end{equation}
and that for all $j\not \in A$ 
\begin{equation}
\label{o-1}
\lim_{N \to \infty} \frac{ \pi_N(\Delta_N) }{ \pi_N(\ms E^j_N) } 
\;=\; 0 \;.
\end{equation}
Then, property (LP2) is in force.
\end{theorem}

\begin{remark}
\label{o-r1} 
In the previous theorem, we may replace conditions \eqref{o-3},
\eqref{o-1} by the assumption
\begin{equation*}
\lim_{N \to \infty} \frac 1{ \sum_{k\not = j}\bs r_N(j,k)}
\, \frac{ \pi_N(\Delta_N) }{ \pi_N(\ms E^j_N) }\;=\; 0
\end{equation*}
for all $j\in S$.
\end{remark}

In some spin dynamics, the valleys are formed by few configurations
and the following simple argument applies.

\begin{lemma}
\label{o-t2}
Assume that 
\begin{equation*}
\lim_{N \to \infty} \max_{\eta\in \ms E_N} 
\frac{\pi_N(\Delta_N) }{ \pi_N(\eta) }\;=\; 0\;.
\end{equation*}
Then, condition (LP2) is in force.
\end{lemma}

\begin{proof}
Fix $t>0$. Clearly, dividing and multiplying by $\pi_N(\eta)$,
\begin{equation*}
\bb E^N_{\eta} \Big[\,
\int_0^t \chi_{\Delta_N} \big(\eta_N(s\theta_N)\big) 
\; ds\, \Big]  \;\le\; \frac{1}{ \pi_N(\eta) }\,
\bb E^N_{\pi_N} \Big[\,
\int_0^t \chi_{\Delta_N} \big(\eta_N(s\theta_N)\big) 
\; ds\, \Big] \;.
\end{equation*}
Since $\pi_N$ is the stationary state, the previous expression is
equal to $t \, \pi_N(\Delta_N) / \pi_N(\eta)$, which proves the
lemma. 
\end{proof}

\section{The Poisson equation}
\label{sec5}

We present here an alternative method to prove uniqueness of limit
points of the sequence of measures $\mb P^N$ introduced in Section
\ref{sec3}. It relies on asymptotic properties of the solutions of
Poisson equations.

Assume that we are able to foretell the dynamics of the reduced model,
and denote by $L$ its generator.  Fix a function
$F \colon S \to \bb R$, and let $G = L F$. Denote by $f$,
$g\colon E_N \to\bb R$ the function given by
\begin{equation}
\label{5-4}
f \;=\;  \sum_{k\in S} F(k) \, \chi_{\ms E^k_N}\;, \quad 
g \;=\; \sum_{k\in S} G(k) \, \chi_{\ms E^k_N}\;.  
\end{equation}
The functions $f$, $g$ are constant on each valley $\ms E^\ell_N$ and
vanish at $\Delta_N$. The method presented below relies on the
assumption that the solution $f_N$ of the Poisson equation
\begin{equation*}
\theta_N\, \ms L_N \, f_N \;=\; g
\end{equation*}
is almost constant on each set $\ms E^k_N$. A solution of this
equation exists only if $g$ has zero-mean with respect to $\pi_N$,
which is not necessarily the case. Therefore, we need first to turn
$g$ into a zero-mean function and then to consider the solution of the
Poisson equation. This is the content of conditions (A1), (A2).

Assume that there exists a sequence of function
$g_N\colon E_N \to \bb R$ such that
\begin{itemize}
\item[(A1)] $g_N$ has zero-mean with respect to $\pi_N$, vanishes on
  $\Delta_N$ and converges to $g$ uniformly on $\ms E_N$;
\item[(A2)] Denote by $f_N$ the unique solution of the Poisson
  equation
\begin{equation}
\label{5-6}
\theta_N\, \ms L_N f \;=\; g_N
\end{equation}
in $E_N$. There exists a finite constant $C_0$ such that
\begin{equation*}
\sup_{N\ge 1}\, \max_{\eta\in E_N} |f_N(\eta)| \;\le\; C_0\;,
\quad\text{and} \quad
\lim_{N\to \infty}  \max_{\eta\in \ms E_N} \big|\, f_N(\eta) 
- f(\eta)\, \big| \;=\; 0\;.
\end{equation*}
\end{itemize}

The natural candidate for $g_N$ in conditions (A1) and (A2) is the
function $g$ itself, but it does not have zero-mean. To fulfill this
condition, denote by $\color{blue}\pi$ the stationary state of the
reduced model. We expect $\pi_N(\ms E^k_N)$ to converge to
$\pi(k)$. Hence,
\begin{equation*}
\lim_{N\to \infty} E_{\pi_{N}}[g] \;=\; 
\lim_{N\to \infty} \sum_{k\in S} G(k) \,
\pi_N(\ms E^k_N) \; = \; \sum_{k\in S} (L F) (k) \,
\pi(k) \;=\; 0\;.
\end{equation*}
A reasonable candidate for $g_N$ is thus
$g \,-\, \epsilon_N \, \chi_{\ms E^1_N}$, where
$\epsilon_N = E_{\pi_N}[g]/\pi_N(\ms E^1_N)$ vanishes as $N\to\infty$
[if $\pi_N(\ms E^1_N) \to \pi(1)>0$].

Properties (A1), (A2) have been proved in \cite{et2016, st2017} for
elliptic operators on $\bb R^d$ of the form $\mc L_N f =
e^{N\, V} \nabla \cdot (e^{-N\, V} a \nabla f)$ and in
\cite{ls2017} for one-dimensional diffusions with periodic boundary
conditions. It is an open problem to prove these conditions in the
context of interacting particle systems, say for condensing zero-range
processes.

The main result of this section, Theorem \ref{5-p1} below, asserts
that conditions (A1), (A2) guarantee uniqueness of limit points of the
sequence $\mb P^N$. The proof of this result requires some
preparation.

Let $\color{blue} \bb Q^N_\eta$, $\eta\in E_N$, be the probability
measure on $D([0,\infty), E_N)$ induced by the speeded-up process
${\color{blue}\xi^N(t)} := \eta_N(t\theta_N)$ starting from
$\eta$. Keep in mind that the generator of this process is
$\theta_N\, \ms L_N$.  Denote by $(\mc F^o_t : t\ge 0)$ the
$\sigma$-algebra of subsets of $D([0,\infty), E_N)$ generated by
$\{\eta(s): 0\le s\le t\}$, where $\eta(s)$ represents the coordinate
process.  Fix $\eta\in E_N$ and denote by
$\color{blue} \{\mc F^\eta_t : t\ge 0\}$ the usual augmentation of
$\{\mc F^o_t : t\ge 0\}$ with respect to $\bb Q^N_{\eta}$. We refer to
Section III.9 of \cite{RogWil94} for a precise definition. The
advantage of $\mc F^\eta_t$ with respect to $\mc F^o_t$ is that it is
right-continuous: $\mc F^\eta_t = \cap_{s>t} \mc F^\eta_s$.

Recall from \eqref{1-3} the definition of the time change
$S_{\ms E_N}(t)$ associated to the additive functional $T_{\ms E_N}(t)$.
Clearly, for all $r\ge 0$, $t\ge 0$,
\begin{equation}
\label{5-2}
\{ S_{\ms E_N} (r) \ge t\} \;=\; \{ T_{\ms E_N} (t) \le r\}\;.
\end{equation}

\begin{lemma}
\label{5-p2}
For each $t\ge 0$ and $\eta\in E_N$, $S_{\ms E_N} (t)$ is a stopping
time with respect to the filtration $(\mc F^\eta_t: t\ge 0)$.
\end{lemma}

\begin{proof}
Fix $t\ge 0$, $r\ge 0$ and $\eta\in E_N$. By \eqref{5-2},
\begin{equation*}
\{S_{\ms E_N} (t) \le r\} \;=\; \bigcap_{q} \, \{S_{\ms E_N} (t) < r + q\}
\;=\; \bigcap_{q} \, \{T_{\ms E_N} (r + q)>t \}\;,
\end{equation*}
where the intersection is carried out over all
$q\in (0,\infty)\cap \bb Q$. By definition of $T_{\ms E_N}$,
$\{T_{\ms E_N} (r + q)>t \}$ belongs to $\mc F^\eta_{r+q}$. Hence, as the
filtration is right-continuous,
$\{S_{\ms E_N} (t) \le r\} \in \cap_q \, \mc F^\eta_{r+q} = \mc F^\eta_{r}$, which
proves the lemma.
\end{proof}

Let $\color{blue}(\mc G^{N,\eta}_t : t\ge 0)$ be the filtration given
by $\mc G^{N,\eta}_t = \mc F^\eta_{S_{\ms E_N}(t)}$, and denote by
$\color{blue} \xi^{\ms E_N}(t)$ the trace of the coordinate process $\eta(t)$
on $\ms E_N$: $\xi^{\ms E_N}(t) = \eta (S_{\ms E_N}(t))$. Clearly, the process
$\xi^{\ms E_N}(t)$ is adapted to the filtration $(\mc
G^{N,\eta}_t)$. Moreover, as the coordinate process corresponds to the
distribution of $\xi^N(t)$, $\xi^{\ms E_N}(t)$ corresponds to the
trace of the speeded-up process $\xi^N(t)$ on $\ms E_N$.

It is easy to check that we may commute the trace operation with the
acceleration of the process:
\begin{equation*}
\eta^{\ms E_N}(t \theta_N) \;=\; \xi^{\ms E_N}(t)\;.
\end{equation*}
On the left-hand side, we first computed the trace of the chain
$\eta_N(t)$ on $\ms E_N$ and then accelerated it by $\theta_N$, while
on the right-hand side we first speeded-up the chain $\eta_N(t)$ by
$\theta_N$ and then computed the trace of the result on $\ms E_N$. In
particular, the process $\bs X^T_N(t)$, introduced in assumption (T1),
corresponds to the projection of $\xi^{\ms E_N}(t)$ on $S$ through $\Psi_N$:
\begin{equation}
\label{5-3}
\bs X^T_N(t) \;=\; \Psi_N(\xi^{\ms E_N}(t))\;.
\end{equation}
Moreover, the measure $\mb P^N$ on $D([0,\infty), S)$ represents the
distribution of the process $\Psi_N(\xi^{\ms E_N}(t))$.  We may now state the
main result of this section.

\begin{theorem}
\label{5-p1}
Fix $k\in S$ and a sequence $\eta_N \in \ms E^k_N$.  Assume that
conditions (A1) and (A2) are in force for every function $F: S\to \bb
R$. Then, every limit point $\mb P$ of the sequence $\mb P^N$
such that 
\begin{equation}
\label{5-7}
\mb P\big[\, X(t-) \,=\, X(t)\,\big] \quad\text{for all}\;\; t>0\;.
\end{equation}
solves the $(L,\delta_k)$ martingale problem.
\end{theorem}

\begin{proof}
Fix a function $F: S \to \bb R$.  Let
$f_N: E_N \to \bb R$ be the function given by assumption
(A2). Then,
\begin{align*}
M_N(t) \;=& \; f_N (\eta(t)) \;-\; f_N (\eta(0))
\;-\; \int_0^t \theta_N\,  (\ms L_N f_N)(\eta(s))\, ds \\
\;= & \; f_N (\eta(t)) \;-\; f_N (\eta(0))
\;-\; \int_0^t g_N (\eta(s))\; ds
\end{align*}
is a martingale in $(D([0,\infty), E_N), (\mc F^{\eta_N}_t), \bb
Q^N_{\eta_N})$.  

Since $\{S_{\ms E_N} (t): t\ge 0\}$ are stopping times with respect to
the filtration $(\mc F^{\eta_N}_t)$, $M_N(S_{\ms E_N}(t))$ is a
martingale with respect to
$(\mc F^{\eta_N}_{S_{\ms E_N}(t)}) = (\mc G^{N,\eta_N}_t)$. Hence, by
definition of the trace process $\xi^{\ms E_N} (t)$,
\begin{equation*}
\widehat M_N(t) \;:=\; M_N(S_{\ms E_N}(t)) \;=\;  
f_N (\xi^{\ms E_N} (t)) \;-\; f_N (\xi^{\ms E_N} (0))
\;-\; \int_0^{S_{\ms E_N}(t)} g_N (\eta(s) )\; ds
\end{equation*}
is a martingale with respect to the filtration $\mc G^{N,\eta_N}_t$.
Since $g_N$ vanishes on $\Delta_N$, we may insert in the integral the
indicator function of the set $\ms E_N$. Then, a change of variables
yields that this integral is equal to
\begin{equation*}
\int_0^{S_{\ms E_N}(t)} g_N (\eta(s))\, \chi_{\ms E_N} (\eta(s)) \, ds 
\;=\; \int_0^{t} g_N (\eta (S_{\ms E_N}(s)))\; ds\;.
\end{equation*}
Therefore,
\begin{equation*}
\widehat M_N(t) \;=\; f_N (\xi^{\ms E_N} (t)) \;-\; 
f_N (\xi^{\ms E_N} (0)) \;-\; 
\int_0^{t} g_N (\xi^{\ms E_N} (s))\, ds
\end{equation*}
is a $\{\mc G^{N,\eta_N}_t\}$-martingale.

By (A1) and (A2), $g_N$, resp. $f_N$, converge to $g$, resp. $f$,
uniformly in $\ms E^N$ as $N \to 0$. Hence, since
$\xi^{\ms E_N} (s) \in \ms E^N$ for all $s\ge 0$, we may replace in
the previous equation $g_N$, $f_N$ by $g$, $f$, respectively, at a
cost which vanishes as $N\to \infty$. Therefore,
\begin{equation*}
\widehat M_N(t) \;=\; f (\xi^{\ms E_N}(t)) \;-\; f (\xi^{\ms E_N}(0)) 
\;-\; \int_0^{t} g (\xi^{\ms E_N} (s))\; ds \;+\; o_N(1)\, 
\end{equation*}
is a $\{\mc G^{N,\eta_N}_t\}$-martingale.

Since $f$ and $g$ are constant on each set $\ms E^k_N$, by \eqref{5-4}
and \eqref{5-3},
$f (\xi^{\ms E_N}(t)) = F(\Psi_N(\xi^{\ms E_N}(t))) = F(\bs X^T_N(t))$,
$g (\xi^{\ms E_N}(s)) = G(\Psi_N(\xi^{\ms E_N}(s))) = G(\bs X^T_N(s))$, and 
\begin{equation*}
F(\bs X^T_N(t)) \;-\; F(\bs X^T_N(0)) \;-\; 
\int_0^t (L F)(\bs X^T_N(s))\; ds \;+\; o_N(1) 
\end{equation*}
is a martingale because $G=LF$.

Since $\mb P^N$ corresponds to the distribution of $\bs X^T_N$,
\begin{equation*}
\widehat M (t) \;=\; F(X(t)) \;-\; F(X(0)) \;-\; \int_0^t (L
F)(X(s))\, ds 
\end{equation*}
is a martingale under $\mb P^N$ up to a small error. Let $\mb P$ be a
limit point of the sequence $\mb P^N$ satisfying \eqref{5-7}, and
assume, without loss of generality, that $\mb P^N$ converges to
$\mb P$.  By \eqref{5-7}, the one-dimensional projections are
continuous, and we may pass to the limit to obtain that
$\widehat M (t)$ is a martingale under $\mb P$.

On the other hand, as $\eta_N\in \ms E^k_N$, $\mb P^N[X(0)=k] = 1$ for
all $N$, so that $\mb P[X(0)=k] = 1$. This proves that any limit point
of the sequence $\mb P^N$ satisfying \eqref{5-7} is a solution of the
$(L,\delta_k)$ martingale problem.
\end{proof}

\section{Local ergodic theorem in $L^2$}
\label{sec11}

It is not clear whether the scheme presented in the previous section
can be applied to a large class of dynamics. The proof of condition
(A2) is unclear even for the simple example of Section \ref{sec1}.

The method presented in Sections \ref{sec3}--\ref{sec7} has also a
drawback. As the function $f = \sum_{k\in S} F(k) \chi_{\ms E^k_N}$
has a sharp interface, the jump rates $R^{(k)}_N$ which appear in the
computation of $\ms L_{\ms E_N} f$, are singular functions, vanishing
at the interior of the valleys and taking large values at the
boundary. This lack of smoothness turns the proof of the local ergodic
theorem more demanding.

Following \cite{bl9}, we propose below an alternative approach, in
which we replace the indicator function $\chi_{\ms E^k_N}$ by
``smooth'' approximations obtained by solving the resolvent equation
\begin{equation}
\label{h-1}
\big(\, I \,-\, \gamma_N \,\ms L_{\ms E_N}\, \big ) \, f \;=\;
\chi_{\ms E^k_N} \;,
\end{equation}
where $\ms L_{\ms E_N}$ represents the generator of the trace
process $\eta^{\ms E_N}(t)$, $I$ the identity and $\gamma_N$ a suitable
sequence of positive numbers. 

The resolvent equation \eqref{h-1} has a unique solution, denoted by
$u^k_N$. Equation \eqref{h-2} provides a stochastic representation of
the solution, different from the usual one given in terms of a time
integral. This guarantees existence. Uniqueness can be proven as
follows. Let $u_1$, $u_2$ be two solutions, and set $w= u_1-u_2$. The
function $w$ solves \eqref{h-1} with a right-hand side equal to
$0$. Multiply both sides of the equation by $w$ and integrate with
respect to $\pi_{\ms E}$ to get that $w=0$ because $\< \,\ms L_{\ms
  E_N}\, f \,,\, f\>_{\pi_{\ms E}} \le 0$ for all functions $f: \ms
E_N \to \bb R$.

Note that $\gamma_N\, \ms L_{\ms E_N} u^k_N$ has the same regularity
as $u^k_N$ because it is equal to $u^k_N - \chi_{\ms E^k_N}$.  We
prove in Lemmata \ref{h-s2}, \ref{h-s3} that $u^k_N$ is close to
$\chi_{\ms E^k_N}$ and that the local ergodic theorem holds for $\ms
L_{\ms E_N} u^k_N$ if $\gamma_N$ is larger than the equilibration
times in the valleys and smaller than the transition times between
valleys.

\subsection{The enlarged process} 

We assume below that the reader is familiar with the results on
enlarged and reflected chains summarized in Section \ref{sec13}.

We do not require below the process $\eta_N(t)$ to be reversible, but
we impose certain conditions on the reflected processes.  Denote by
$\color{blue} \eta^{R,k}_N(t)$ the process $\eta_N(t)$ reflected at
$\ms E^k_N$. Recall that this means that we forbid all jumps between
$\ms E^k_N$ and its complement, and consider the resulting dynamics in
$\ms E^k_N$.

Denote by $\color{blue} \pi_{\ms E^k}$ the stationary measure $\pi_N$
conditioned to $\ms E^k_N$.  We assume that for all $k\in S$ the
reflected process at $\ms E^k_N$ is \emph{\color{blue} irreducible}
and that $\pi_{\ms E^k}$ is a stationary state (and therefore the
unique stationary state up to multiplicative constants).  If the
process is reversible, the second condition follows from the first
one. By Lemma \ref{M-s4}, this is also the case in the non-reversible
setting if the valley $\ms E^k_N$ is formed by cycles.

Denote by $\ms L^{R,k}_N$ the generator of the process
$\eta^{R,k}_N(t)$ and by $\color{blue} t^{k,N}_{\rm rel}$ the
\emph{\color{blue} relaxation time} of the symmetric part of the
generator:
\begin{equation*}
{\color{blue} \frac 1{t^{k,N}_{\rm rel}}} 
\;=\; \inf_f \frac{\< \,(-\, \ms L^{R,k}_N)\,
  f \,,\, f\, \>_{\pi_{\ms E^k}}}{\< \, f \,,\, f\, \>_{\pi_{\ms E^k}}} \;,
\end{equation*}
where the infimum is carried over all zero-mean functions $f: \ms
E^k_N \to \bb R$.
\smallskip

Let $\color{blue}\ms E^{\star, k}_N$ be copies of the sets
$\ms E^k_N$, $k\in S$, and set
\begin{equation*}
{\color{blue} \ms E^\star_N} \;:=\; \bigcup_{k\in S} \ms E^{\star,
  k}_N\;, \quad
{\color{blue} \breve{\ms E}^{\star, k}_N}  \;:=\; 
\bigcup_{j\not = k} \ms E^{\star, j}_N\;.
\end{equation*}
Denote by $\color{blue} P_\star: \ms E_N \cup \ms E^\star_N \to \ms E_N
\cup \ms E^\star_N$ the application which maps a configuration in $\ms
E_N$, $\ms E^\star_N$, to its copy in $\ms E^\star_N$, $\ms E_N$,
respectively.

Fix a sequence $\gamma_N$, and denote by $\color{blue} \zeta_N(t)$ the
\emph{\color{blue} $\gamma_N$-enlargement} of the trace process
$\eta^{\ms E_N}(t)$.  The process $\zeta_N(t)$ is a Markov chain
taking values in $\ms E_N \cup \ms E^\star_N$ and whose generator,
denoted by $\color{blue} \ms L_{\ms E_N, \star}$, is given by
\begin{gather*}
(\ms L_{\ms E_N, \star} f) (\eta) \,=\, \sum_{\xi\in \ms E_N} R^T_N(\eta,\xi)
\, \big\{f(\xi)-f(\eta)\big\} \;+\; \frac 1{\gamma_N}\,
\big\{f(P_\star \eta)-f(\eta)\big\}
\;, \quad \eta\in \ms E_N \;, \\
(\ms L_{\ms E_N, \star} f) (\eta) \,=\, \frac 1{\gamma_N}\,
\big\{f(P_\star \eta)-f(\eta)\big\}
\;, \quad \eta\in \ms E^\star_N \;.
\end{gather*}
In this formula, $R^T_N(\eta,\xi)$ represents the jump rates of the
trace process $\eta^{\ms E_N}(t)$. Hence, from a configuration
$\eta\in \ms E^\star_N$ the chain may only jump to $P_\star \eta$ and
this happens at rate $1/\gamma_N$. From a configuration $\eta\in \ms
E_N$, besides the jumps of the original chain, the enlarged process
may also jump to $P_\star \eta$ and this happens at rate
$1/\gamma_N$. The parameter $\gamma_N$ will be large, which makes the
jumps between $\ms E^\star_N$ and $\ms E_N$ rare.

The stationary state of $\zeta_N(t)$, denoted by $\color{blue}
\pi^\star_{\ms E}$, is given by
\begin{equation*}
\pi^\star_{\ms E} (\eta) \;=\; \pi^\star_{\ms E} (P_\star\eta) 
\;=\; (1/2)\, \pi_{\ms E} (\eta)  \; , \quad \eta\in \ms E_N\;,
\end{equation*}
where, recall, $\pi_{\ms E}$ stands for the stationary state $\pi_N$
conditioned to $\ms E_N$.

In dynamics in which the process jumps to a new valley before visiting
all configurations in the valley, as configurations are not visited,
it makes more sense to suppose that the dynamics starts from a
distribution rather than from a configuration. Denote this initial
distribution by $\nu_N$ and assume that there exist $\ell\in S$ and a
finite constant $C_0$ such that for all $N\ge 1$
\begin{equation}
\label{h-12a}
\nu_N(\ms E^{\ell}_N)\;=\;1 \quad\text{and}\quad
E_{\pi_{\ms E}} \Big[ \, \Big( \frac {d\nu_N}{d\pi_{\ms E}} 
\Big)^2 \,\Big] \;\le\; \frac{C_0}{\pi_{\ms E}(\ms E^{\ell}_N) }\;\cdot
\end{equation}
Note that the measure $\pi_{\ms E^\ell}$ satisfies this condition.

For two non-empty, disjoint subsets $\ms A$, $\ms B$ of
$\ms E_N \cup \ms E^*_N$, denote by
$\color{blue} \Cap_\star (\ms A, \ms B)$ the capacity between $\ms A$
and $\ms B$ for the enlarged process.  Consider two sequences
$(a_N : N\ge 1)$, $(b_N : N\ge 1)$ of positive real numbers. We say
that $a_N$ is much smaller than $b_N$, $\color{blue} a_N \ll b_N$, if
$\lim_{N\to \infty} a_N/b_N =0$.

\begin{lemma}
\label{h-s2}
Fix $k\in S$, two sequences of positive numbers $\gamma_N$,
$\theta_N$, and a sequence of probability measures $\nu_N$ satisfying
\eqref{h-12a}. Assume that $\gamma_N \ll \theta_N$ and that there
exists a finite constant $C_0$ such that for all $N\ge 1$
\begin{equation}
\label{h-12b}
\frac {\Cap_\star(\ms E^{\star, k}_N, \breve{\ms E}^{\star, k}_N) }
{\pi^\star_{\ms E} (\ms E^{\star, k}_N)}
\;\le\; \frac{C_0}{\theta_N} \;\cdot
\end{equation}
Then, representing by $u^k_N$ the solution of the resolvent equation
\eqref{h-1}, 
\begin{equation*}
\lim_{N\to \infty} \sup_{t\ge 0}\, 
\bb E_{\nu_N} \Big[ \,\big |\, \chi_{\ms E^k_N} (\eta^{\ms E}
(t\theta_N)) \,-\, u^k_N (\eta^{\ms E}(t\theta_N)) \,\big|\, 
\Big] \;=\;0\;.
\end{equation*}
\end{lemma}

Denote by $\zeta^{\ms E^\star_N}(t)$ the trace of the process
$\zeta_N(t)$ on $\ms E^\star_N$, and by $\color{blue} \bb P^{\star,
  \gamma_N}_\eta$, $\eta\in \ms E_N \cup \ms E^\star_N$, the
probability measure on $D([0,\infty), \ms E_N \cup \ms E^\star_N)$
induced by the enlarged process starting from $\eta$. Let
$r^\star_N(j,k)$, $j\not = k\in S$, be the coarse-grained jump rates
at which the trace process $\zeta^{\ms E^\star_N}(t)$ jumps from $\ms
E^{\star, j}_N$ to $\ms E^{\star, k}_N$. By \eqref{7-1}, these rates
are given by
\begin{equation}
\label{h-13}
{\color{blue} r^\star_N(j,k)} \; =\; 
\frac{1}{\pi^\star_{\ms E} (\ms E^{\star, j}_N)}
\sum_{\eta\in \ms E^{\star, j}_N} \pi^\star_{\ms E} (\eta)\, \lambda_\star(\eta)
\, \bb P^{\star, \gamma_N}_\eta
\big[ \, H(\ms E^{\star, k}_N) \,<\,  H^+(\breve{\ms E}^{\star, j}_N)\, \big]\;,
\end{equation}
where $\lambda_\star(\eta)$ represents the holding rates of
$\zeta_N(t)$. Since the enlarged process jumps from $\eta\in \ms
E^{\star}_N$ to $P_\star \eta$ at rate $\gamma_N$, the previous
expression is equal to
\begin{equation*}
\frac{1}{\gamma_N\, \pi_{\ms E} (\ms E^{j}_N)}
\sum_{\eta\in \ms E^{j}_N} \pi_{\ms E} (\eta)\, \bb P^{\star, \gamma_N}_\eta
\big[ \, H (\ms E^{\star, k}_N) \,<\, H(\breve{\ms E}^{\star, j}_N)\, \big] \;.  
\end{equation*}

According to Section \ref{sec7}, in the reversible case, the
coarse-grained jump rates $r^\star_N(j,k)$ can be expressed in terms
of capacities, while in the non-reversible case they can be computed
if there are good approximations of the equilibrium potential.
Assume, from now on, that these rates converge: There exist a
time-scale $\theta_N$ and jump rates $\bs r(j,k)$ such that
\begin{equation}
\label{h-6}
\lim_{N\to \infty} \theta_N\, r^\star_N(j,k) \;=\; \bs r(j,k)
\quad\text{for all $j\not = k \in S$} \;.
\end{equation}
Condition \eqref{h-12b} follows from this hypothesis
since 
\begin{equation*}
\Cap_\star(\ms E^{\star, k}_N, \breve{\ms E}^{\star, k}_N) \;=\;
\pi^\star_{\ms E} (\ms E^{\star, k}_N)\, \sum_{\ell\not = k}
r^\star_N(k,\ell) \;.
\end{equation*}

The sequence $\theta_N$ represents the time-scale at which the process
jumps between valleys. The proof of a metastable behavior is set up on
the ground that this time-scale is much larger than the equilibration
time inside the valleys. This hypothesis is formulated here by
requiring the relaxation times of the processes reflected at a valleys
to be much smaller than $\theta_N$: for all $k\in S$,
\begin{equation}
\label{h-7}
 t^{k,N}_{\rm rel} \;\ll\; \theta_N \;.
\end{equation}

Let $\color{blue} w^k_N = \ms L_{\ms E_N} u^k_N$. Recall from
\eqref{6-1} that $\mc G_N$ represents the $\sigma$-algebra of subsets
of $\ms E_N$ generated by the sets $\ms E^j_N$, $j\in S$. Let
\begin{equation}
\label{h-4}
\widehat w^k_N (\eta) \;=\; E_{\pi_{\ms E}} \big[\, w^k_N \,|\, \mc G_N
\,\big]\;.
\end{equation}

\begin{lemma}
\label{h-s3}
Fix $\ell\in S$, and a sequence of probability measures $\nu_N$
satisfying \eqref{h-12a}.  Assume that for all $j$, $k\in S$, 
\begin{equation}
\label{h-16}
\frac{\pi_{N} (\ms E^k_N)}{\pi_{N} (\ms E^\ell_N)} 
\; t^{j,N}_{\rm rel} \;\ll\; \theta_N \;.
\end{equation}
Let $\gamma_N$ be a sequence such that $\max_{j,k} \alpha^{j,k}_N \ll
\gamma_N\ll \theta_N$, where $\alpha^{j,k}_N$ stands for the left-hand
side of \eqref{h-16}.  Then, for all $T>0$, $k\in S$,
\begin{equation*}
\lim_{N\to \infty} \bb E_{\nu_N} \Big[ \, \sup_{t\le T \theta_N}
\Big| \, \int_0^t  \Big\{  w^k_N (\eta^{\ms E_N} (s)) - 
\widehat w^k_N  (\eta^{\ms E_N}(s)) \Big\} \, ds \, \Big|\; \Big] 
\;=\;0\;.
\end{equation*}
\end{lemma}

By \eqref{h-2} and a straightforward computation,
\begin{equation*}
\widehat w^k_N (\eta) \;=\; \sum_{j\in S} \bs r^\star_N(j,k)\; 
\chi_{\ms E^j_N} (\eta) \;,
\end{equation*}
where $\bs r^\star_N(j,k)$, $j\not = k$, are the coarse-grained jump
rates introduced in \eqref{h-13}, and $\bs r^\star_N(j,j) = -\,
\sum_{k\not = j} \bs r^\star_N(j,k)$. Thus, for every function $F:S
\to \bb R$,
\begin{equation}
\label{h-3}
\sum_{k\in S} F(k)\, \widehat w^k_N (\eta) \;=\;
\sum_{j\in S} \chi_{\ms E^j_N} (\eta) 
\sum_{k\not = j} \bs r^\star_N(j,k) \, [\, F(k) - F(j)\,] \;.
\end{equation}

Fix $\ell\in S$ and a sequence of probability measures satisfying
conditions \eqref{h-12a}.  Let $\mb P^N$ be the probability measure on
$D([0,\infty), S)$ induced by the process $\bs X^T_N$ and the measure
$\nu_N$. Next theorem is the main result of this section.

\begin{theorem}
\label{h-s1}
Fix $\ell\in S$ and a sequence $\nu_N$ of probability measures
satisfying \eqref{h-12a}. Assume that conditions \eqref{h-6} and
\eqref{h-16} are in force. Then, every limit
point $\mb P$ of the sequence $\mb P^N$ such that
\begin{equation}
\label{h-14}
\mb P\big[\, X(t-) \,=\, X(t)\,\big] \quad\text{for all}\;\; t>0\;.
\end{equation}
solves the $(L,\delta_\ell)$ martingale problem, where $L$ is the
generator of the $S$-valued Markov chain whose jump rates are
$\bs r(j,k)$.
\end{theorem}

Theorem \ref{h-s1} describes the asymptotic evolution of the trace of
the Markov $\eta(t)$ on $\ms E_N$. The next lemma shows that in the
time scale $\theta_N$ the time spent on the complement of $\ms E_N$ is
negligible. The proof is similar to the one of Lemma \ref{o-t2} and
uses Schwarz inequality and assumption \eqref{h-12a} to replace $\nu_N$
by $\pi_N$.

\begin{lemma}
\label{h-s6}
Assume that
\begin{equation*}
\lim_{N\to\infty} \frac{\pi(\Delta_N)}{\pi(\ms E^j_N)} \;=\; 0
\end{equation*}
for all $j\in S$. Fix $\ell\in S$, and let $\{\nu_N : N\ge 1\}$ be a
sequence of probability measures satisfying \eqref{h-12a}. Then, for
every $t>0$,
\begin{equation*}
\lim_{N\to\infty} \bb E_{\nu_N}
\Big[\int_0^t \chi_{\Delta_N} (\eta (s\theta_N))
\, ds \,\Big]\, =\,0\,.
\end{equation*}
\end{lemma}

\begin{remark}
\label{h-rm3}
The introduction of the enlarged process is inspired by the definition
of the \emph{\color{blue} soft hitting time} of Bianchi and
Gaudilli\`ere \cite{bg16}.
\end{remark}

\begin{remark}
\label{h-rm5}
Hypothesis \eqref{h-16} can be divided in two: Assume \eqref{h-7}, and
suppose that there exist constants $0<c_0<C_0<\infty$ such that $c_0\,
\pi_{N} (\ms E^k_N) \le \pi_{N} (\ms E^\ell_N) \le C_0\, \pi_{N} (\ms
E^k_N)$ for all $k$, $\ell\in S$.
\end{remark}

\begin{remark}
\label{h-rm6}
Hypothesis \eqref{h-16} can be weaken as follows. Instead of fixing
the same rate $\gamma_N$ for all valleys, we may choose a
valley-dependent rate. This does not alter the stationary
state, and it permits to choose larger parameters $\gamma_N$ for
deeper valleys. Assumption \eqref{h-16} may also be weaken to admit a
deep valley, all the other ones being shallow (cf. \cite{bl9}).
\end{remark}

\begin{remark}
\label{h-rm4}
In Subsection \ref{14-4}, we apply the method presented above to a
polymer model examined by Caputo et al. in \cite{cmt, clmst}. It can
also be employed to derive the reduced model of the random walk
presented in Section \ref{sec1}. We refer to in \cite{bl9}.  Lacoin
and Teixeira \cite{lt15} followed this scheme to prove the metastable
behavior of a polymer interface which interacts with an attractive
substrate.
\end{remark}

\begin{proof}[Proof of Theorem \ref{h-s1}]
Fix $\ell\in S$ and a sequence $\nu_N$ of probability measures
satisfying \eqref{h-12a}.
Fix a function $F: S\to\bb R$ and a limit point $\mb P$ of the
sequence $\mb P^N$ satisfying \eqref{h-14}. Assume, without loss of
generality, that $\mb P^N$ converges to $\mb P$. We claim that
\begin{equation}
\label{h-15}
M^F(t) \;:=\; F(X_t) \;-\; F(X_0) \;-\; \int_0^t (L\, F)(X_s)\, ds
\end{equation}
is a martingale under $\mb P$, where $L$ is the generator associated
to the jump rates $\bs r(j,k)$ introduced in \eqref{h-6}.

Fix $0\le s< t$, $q\ge 1$, $0\le t_1< \cdots < t_q\le s$, and a bounded
function $g: S^q \to \bb R$. Let $G = g(X(t_1), \dots, X(t_q))$, where
$X(s)$ represents the coordinate process of $D([0,\infty), S)$. We
shall prove that
\begin{equation*}
\mb E \,\big[\, M^F (t) \, G\, \big] \;=\; 
\mb E\,\big[\, M^F (s) \, G \, \big]\;,
\end{equation*}
where $\mb E$ stands for the expectation with respect to $\mb P$. 

Fix a sequence $\gamma_N$ such that for all $j$, $k\in S$,
\begin{equation*}
\frac{\pi_{N} (\ms E^k_N)}{\pi_{N} (\ms E^\ell_N)} 
\; t^{j,N}_{\rm rel} \;\ll\; \gamma_N \;\ll\; \theta_N \;,   
\end{equation*}
which is possible in view of \eqref{h-16}, and recall that we denote
by $u^k_N$ the solution of \eqref{h-1}.  Let
\begin{equation*}
H_N(\eta) \;=\; \sum_{k\in S} F(k) \, u^k_N (\eta)\;, \quad
\eta\in \ms E_N\;.
\end{equation*}
By the Markov property of the trace process $\eta^{\ms E}(t)$,
\begin{equation*}
M^N_t \;=\; H_N(\eta^{\ms E}(t\theta_N)) \;-\; H_N(\eta^{\ms E}(0))
\;-\; \int_0^{t \theta_N} (\ms L_{\ms E_N} H_N) (\eta^{\ms E}(s))\, ds
\end{equation*}
is a martingale. In particular, if 
\begin{equation*}
G_N \;=\;  g\big(\, \bs X^T_N(t_1)
\,,\, \dots \,,\, \bs X^T_N(t_q)\, \big)
\;=\;  g\big(\, X^T_N(t_1\theta_N)
\,,\, \dots \,,\, X^T_N(t_q\theta_N)\, \big)\;,
\end{equation*}
we have that
\begin{equation*}
\bb E_{\nu_N} \big[\, M^N_t \, G_N\, \big]
\, = \, \bb E_{\nu_N} \big[\, M^N_s \, G_N\, \big]\;,
\end{equation*}
so that
\begin{equation*}
\bb E_{\nu_N} \Big[ \, G_N \, \Big\{ H_N (\eta^{\ms E}(t\theta_N))
\;-\; H_N(\eta^{\ms E}(s \theta_N))
\;-\; \int_{s \theta_N}^{t \theta_N} (\ms L_{\ms E_N} H_N) (\eta^{\ms E}(r))\, dr
\Big\} \, \Big ] \,=\, 0\,.
\end{equation*}

By Lemma \ref{h-s2}, 
\begin{equation*}
\lim_{N\to \infty} \sup_{r\ge 0}\,
\bb E_{\nu_N} \Big[ \, \big| \, H_N (\eta^{\ms E}(r\theta_N))
\,-\, (F\circ\Psi_N) (\eta^{\ms E} (r\theta_N))  \,\big|\,
\Big] \;=\;0\;.
\end{equation*}
Thus, by the penultimate equation and since $\bs X^T_N (t) = X^T_N
(t\theta_N) = \Psi_N(\eta^{\ms E_N}(t \theta_N))$,
\begin{equation*}
\lim_{N\to\infty} \bb E_{\nu_N} \Big[ G_N\, \Big\{ 
F(\bs X^T_N(t)) \,-\, F(\bs X^T_N(s))
\;-\; \int_{s \theta_N}^{t \theta_N} 
(\ms L_{\ms E_N} H_N) (\eta^{\ms E}(r))\, dr
\Big\} \Big ] \,=\, 0\,.
\end{equation*}

By definition of $H_N$ and $w^k_N$, introduced just above \eqref{h-4},
$\ms L_{\ms E_N} H_N = \sum_k F(k) \, w^k_N$. Hence, by Lemma
\ref{h-s3} and \eqref{h-3},
\begin{equation*}
\lim_{N\to\infty} \bb E_{\nu_N} \Big[\, G_N\,  \Big\{ 
F(\bs X^T_N(t)) \,-\,  F(\bs X^T_N(s))
\;-\; \int_{s}^{t}   \, (L^\star_N F)
(\bs X^T_N(r))\, dr \Big\} \, \Big ] \,=\, 0\,.
\end{equation*}
where $L^\star_N$ is the generator of a $S$-valued Markov chain given
by
\begin{equation*}
(L^\star_N F) (j) \;=\;  \sum_{k\in S} \theta_N\, 
\bs r^\star_N(j,k) \, [\, F(k) - F(j)\,] \;.
\end{equation*}

At this point, the martingale has been expressed as a function of the
process $\bs X^T_N(t)$. By definition of the measure $\mb P^N$, the
previous expectation is equal to
\begin{equation*}
\mb E^N \Big[\, g\big(\, X(t_1)
\,,\, \dots \,,\, X (t_q)\, \big) \,  \Big\{
F(X(t)) \,-\,  F(X(s))
\;-\; \int_{s}^{t}   \, (L^\star_N F)
(X(r))\, dr \Big\} \, \Big ]\;, 
\end{equation*}
where, recall, $X(t)$ represents the coordinate process in
$D([0,\infty), S)$ and $\mb E^N$ expectation with respect to $\mb
P^N$.

By assumption \eqref{h-6}, $(L^\star_N F)(k)$ converges to $(L F)(k)$
for all $k \in S$. Therefore, as $\mb P^N$ converges to $\mb P$ and in
view of \eqref{h-14} [which guarantees that the finite-dimensional
projections are continuous], passing to the limit, we get that
\begin{equation*}
\mb E \Big[\, g\big(\, X(t_1)
\,,\, \dots \,,\, X (t_q)\, \big) \,  \Big\{
F(X(t)) \,-\,  F(X(s))
\;-\; \int_{s}^{t}   \, (L F)
(X(r))\, dr \Big\} \, \Big ] \;=\; 0\;. 
\end{equation*}
This shows that \eqref{h-15} holds, and completes the proof of the
theorem.  
\end{proof}

\subsection{The resolvent equation}  

We examine in this subsection the asymptotic behavior of the solution
of resolvent equation \eqref{h-1}.

Fix $\gamma_N>0$ and consider the $\gamma_N$-enlargement of the
process $\eta^{\ms E_N}(t)$.  Let
$h^k_N: \ms E_N \cup \ms E^\star_N \to [0,1]$ be the equilibrium
potential between the sets $\ms E^{\star, k}_N$ and
$\breve{\ms E}^{\star, k}_N$:
\begin{equation}
\label{h-2}
{ \color{blue} h^k_N (\eta)} \;:=\; \bb P^{\star, \gamma_N}_\eta \big[ \, 
H(\ms E^{\star, k}_N) \,<\, H (\breve{\ms E}^{\star, k}_N) \,\big ]\;.
\end{equation}
Since $\ms L_{\ms E_N, \star} h^k_N =0$ on $\ms E_N$, we deduce that
the restriction of $h^k_N$ to $\ms E_N$ solves the resolvent equation
\eqref{h-1}. Since the solution is unique, $u^k_N = h^k_N$ on
$\ms E_N$ and we have a simple stochastic representation of the
solution of the resolvent equations.

\begin{remark}
\label{h-rm1}
The enlargement of the chain $\eta^{\ms E_N}(t)$ thus provides a
stochastic representation of the resolvent equation \eqref{h-1}.
\end{remark}

\begin{lemma}
\label{h-s4}
There exists a finite constant $C_0$, independent of $N$, such that for
all $k\in S$, 
\begin{equation*}
\begin{aligned}
& \frac 1{\gamma_N} \, \sum_{\eta\in \ms E^{k}_N} 
\pi_{\ms E}(\eta) \, [1- u^k_N(\eta)]^2 
\;+\; \< \, (-\ms L_{\ms E_N}) \, u^k_N 
\,,\, u^k_N \, \>_{\pi_{\ms E}} \;+\;
\frac 1{\gamma_N} \, \sum_{\eta\in \breve{\ms E}^{k}_N} 
\pi_{\ms E}(\eta) \, u^k_N  (\eta)^2 \\
&\qquad \;\le\; 
\frac{C_0}{\theta_N} \, \pi_{\ms E} (\ms E^{k}_N) \;.
\end{aligned}
\end{equation*}
\end{lemma}

\begin{proof}
Denote the left-hand side of the inequality by $A_N$, and by $B_N$
the same expression with $\pi_{\ms E}$ in place of $\pi^\star_{\ms E}$. Since
$\pi^\star_{\ms E} (\eta) = (1/2) \, \pi_{\ms E}(\eta)$,
$\eta\in \ms E_N$, $A_N=2B_N$.  As $u^k_N$ and $h^k_N$ coincide on
$\ms E_N$, we may replace the former by the latter. On the other hand,
as $h^k_N = \chi_{\ms E^{k,\star}_N}$ on $\ms E^\star_N$,
$B_N = D_{N,\star}(h^k_N)$, where $D_{N,\star}(f)$ represents the
Dirichlet form of $f$ with respect to the enlarged process
$\zeta_N(t)$. 

By \eqref{J-3}, 
\begin{equation*}
D_{N,\star}(h^k_N) \;=\;
\Cap_\star(\ms E^{\star, k}_N, \breve{\ms E}^{\star, k}_N) \;. 
\end{equation*}
Thus,
$A_N = 2 \, \Cap_\star(\ms E^{\star, k}_N, \breve{\ms E}^{\star,
  k}_N)$.
By assumption \eqref{h-12b}, the capacity is less than or equal to
$C_0\, \pi^\star_{\ms E} (\ms E^{\star, k}_N) / \theta_N$ for some
finite constant $C_0$.  This proves the assertion because
$ \pi^\star_{\ms E} (\ms E^{\star, k}_N) = (1/2)\, \pi_{\ms E} (\ms
E^{k}_N)$.
\end{proof}

\begin{proof}[Proof of Lemma \ref{h-s2}]
Fix $\ell\in S$, a sequence of probability measures $\nu_N$ satisfying
the hypotheses of the lemma and $t>0$.
Denote by $S_{\ms E}(t)$, $t\ge 0$, the semigroup associated
to the trace process $\eta^{\ms E}(t)$, and by $f^N_t$ the Radon-Nikodym
derivative $d\nu_N S_{\ms E}(t)/ d\pi_{\ms E}$. By \eqref{M-6},
$E_{\pi_{\ms E}}[(f^N_t)^2] \le E_{\pi_{\ms E}}[(f^N_0)^2]$.  Hence, by
Schwarz inequality, the square of the expectation appearing in the
statement of the lemma is bounded above by
\begin{equation*}
E_{\pi_{\ms E}} \Big[ \Big( \frac{d\nu_N}{d\pi_{\ms E}} \Big)^2 \Big]
\; E_{\pi_{\ms E}} \Big[\,\big( \chi_{\ms E^\ell_N}  - u^\ell_N \big)^2\, \Big]\;.
\end{equation*}
By Lemma \ref{h-s4}, the second term is bounded by $C_0 \, \gamma_N\,
\pi_{\ms E} (\ms E^{\ell}_N) / \theta_N$. Thus, by the assumption on
the sequence of probability measures $\nu_N$, the previous displayed
formula is bounded by $C_0 \, \gamma_N / \theta_N$. This expression
vanishes as $N\to\infty$ by the hypothesis on $\gamma_N$.
\end{proof}

\subsection{Local ergodicity} 

The proof of Lemma \ref{h-s3} is divided in several steps.  Denote by
$\color{blue} \<\, \cdot\,,\,\cdot\,\>_{\pi_{\ms E}}$ the scalar
product in $L^2(\pi_{\ms E})$. For a zero-mean function $f: \ms E_N\to
\bb R$, let $\color{blue} \Vert f\Vert_{-1}$ be the $\mc H_{-1}$ norm
of $f$ associated to the generator $\ms L_{\ms E_N}$:
\begin{equation*}
\Vert f \Vert^2_{-1}
\;=\; \sup_{h} \Big\{ 2\, \<\, f \,,\, h \>_{\pi_{\ms E}} 
\,-\, \<\, h \,,\, (- \ms L_{\ms E_N}) \, h \,\>_{\pi_{\ms E}} \Big\}\; ,
\end{equation*}
where the supremum is carried over all functions $h:\ms E_N\to \bb
R$. By \cite[Lemma 2.4]{klo12}, for every function $f: \ms E_N \to \bb
R$ which has zero-mean with respect to $\pi_{\ms E}$, and every $T>0$,
\begin{equation}
\label{h-11}
\bb E_{\pi_{\ms E}} \Big[ \sup_{0\le t\le T}
\Big( \int_0^{t} f(\eta^{\ms E_N} (s)) 
\, ds  \, \Big)^2 \, \Big] \;\le\; 24\, T\, \Vert f \Vert^2_{-1}\;.
\end{equation}

Recall that we denote by $\pi_{\ms E^k}$ the stationary measure
$\pi_N$ conditioned to $\ms E^k_N$. Let
$\color{blue} \ms L_{R, \ms E^k_N}$ be the generator of the reflected
process $\eta_N(t)$ at $\ms E^k_N$.  For a function
$f: \ms E^k_N\to\bb R$ which has zero-mean with respect to
$\pi_{\ms E^k}$, denote by $\color{blue} \Vert f\Vert_{k,-1}$ the
$\mc H_{-1}$ norm of $f$ with respect to the generator
$\ms L_{R, \ms E^k_N}$:
\begin{equation*}
\Vert f \Vert^2_{k,-1}
\;=\; \sup_{h} \Big\{ 2 \, \<\, f \,,\, h \, \>_{\pi_{\ms E^k}} 
- \<\, h \,,\, (- \ms L_{R, \ms E^k_N}) \,  h \,\>_{\pi_{\ms E^k}} \Big\}\; ,
\end{equation*}
where the supremum is carried over all functions $h:\ms E^k_N\to \bb
R$. It is clear that
\begin{equation*}
\sum_{j\in S} \pi_{\ms E}(\ms E^j_N) \, 
\< \,  h \,,\, (-\ms L_{R, \ms E^j_N}) \, h \,\>_{\pi_{\ms E^j}} 
\;\le\; \<\, h \,,\, (- \ms L_{\ms E_N}) \, h \, \>_{\pi_{\ms E}}
\end{equation*}
for any function $h: \ms E_N \to\bb R$. These expression are not equal
because two kinds of jumps appear on the right-hand side and do not on
the left: The trace process may jump between valleys, and it may also
perform a jump inside a valley (crossing the set $\Delta_N$) which is
not possible in the original dynamics.

It follows from the previous inequality and from the formulae for the
$\mc H_{-1}$ norms that for every function $f:\ms E_N\to \bb R$ which
has zero-mean with respect to each measure $\pi_{\ms E^j}$,
\begin{equation}
\label{h-10}
\Vert f \Vert^2_{-1} \;\le\; \sum_{j\in S} 
\pi_{\ms E}(\ms E^j_N)\,
\Vert f \Vert^2_{j,-1}\;.
\end{equation}

\begin{lemma}
\label{h-s5}
Let $\{\nu_N : N\ge 1\}$ be a sequence of probability measures on $\ms
E_N$. Then, for every function $f:\ms E_N\to \bb R$ which has 
zero-mean with respect to each measure $\pi_{\ms E^j}$ and for every $T>0$,
\begin{equation*}
\Big( \, \bb E_{\nu_N} \Big[ \sup_{t\le T} \, \Big|\int_0^{t} 
f(\eta^{\ms E}(s)) \, ds \Big| \,\Big] \, \Big)^2
\;\le\; 24 \, T\, 
E_{\pi_{\ms E}} \Big[ \Big( \frac {\nu_N}{\pi_{\ms E}} \Big)^2 \Big] 
\, \sum_{j\in S} \pi_{\ms E}(\ms E^j_N) \, \Vert f \Vert^2_{j,-1} \; .
\end{equation*}
\end{lemma}

\begin{proof}
By Schwarz inequality, the expression on the left hand side
is bounded above by
\begin{equation*}
E_{\pi_{\ms E}} \Big[ \Big( \frac {\nu_N}{\pi_{\ms E}} \Big)^2 \Big] 
\,  \bb E_{\pi_{\ms E}} \Big[ \sup_{t\le T} \Big( \int_0^{t} 
f (\eta^{\ms E}(s)) \,  ds \Big)^2 \,\Big] \;.
\end{equation*}
By \eqref{h-11} and by \eqref{h-10}, the second expectation is bounded by
\begin{equation*}
24 \, T\, \sum_{j\in S} \pi_{\ms E}(\ms E^j_N) 
\, \Vert f \Vert^2_{j,-1} \; ,
\end{equation*}
as claimed.
\end{proof}

\begin{proof}[Proof of Lemma \ref{h-s3}]
Fix $\ell\in S$, and a sequence of probability measures $\nu_N$
satisfying the hypotheses of the lemma. Fix $k\in S$.
Since $w^k_N - \widehat w^k_N$ has zero-mean with respect to each
$\pi_{\ms E^j}$, by the assumption on the sequence $\nu_N$ and
Lemma \ref{h-s5}, the square of the expectation appearing in the
statement of the lemma is bounded by
\begin{equation}
\label{h-9}
\frac{C_0 \, T\, \theta_N}{\pi_{\ms E}(\ms E^\ell)} \, 
\sum_{j\in S}  \pi_{\ms E}(\ms E^j_N) 
\, \Vert  w^k_N - \widehat w^k_N \Vert^2_{j,-1} 
\end{equation}
for some finite constant $C_0$.  

By \eqref{h-1}, on the set $\ms E^k_N$,
$\ms L_{\ms E_N} u^k_N = - \, (1/\gamma_N) \,(1- u^k_N)$, so that
$w^k_N - \widehat w^k_N = (1/\gamma_N) \,(u^k_N - \widehat u^k_N)$.
Hence, by the spectral gap of the reflected process,
\begin{equation*}
\Vert \, w^k_N - \widehat w^k_N \, \Vert^2_{k,-1} \;=\; 
\frac 1{\gamma^2_N} \, \Vert  \, u^k_N - \widehat u^k_N \, \Vert^2_{k,-1} 
\;\le \;  \frac {t^{k,N}_{\rm rel}}{\gamma^2_N} \, 
\Vert\, u^k_N - \widehat u^k_N \, \Vert^2_{\pi_{\ms E^k}}\;.
\end{equation*}
Since $\Vert\, u^k_N - \widehat u^k_N \, \Vert^2_{\pi_{\ms E^k}} \le
\Vert\, u^k_N - 1 \, \Vert^2_{\pi_{\ms E^k}}$, by Lemma \ref{h-s4},
\begin{equation*}
\Vert \, w^k_N - \widehat w^k_N \, \Vert^2_{k,-1} \;\le\; 
C_0 \, \frac{t^{k,N}_{\rm rel}}{\gamma_N \, \theta_N} 
\end{equation*}
for some finite constant $C_0$. 

Similarly, since $\ms L_{\ms E_N} u^k_N = (1/\gamma_N) \, u^k_N$ on
the sets $\ms E^j_N$, $j\not = k$,
\begin{equation*}
\Vert \, w^k_N - \widehat w^k_N \, \Vert^2_{j,-1} \;=\; 
\frac 1{\gamma^2_N} \, \Vert  \, u^k_N - \widehat u^k_N \, \Vert^2_{j,-1} 
\;\le\; C_0 \,  \frac{\pi_{\ms E} (\ms E^k_N)}{\pi_{\ms E} (\ms E^j_N)}\, 
\frac{t^{j,N}_{\rm rel}} {\gamma_N\, \theta_N}\;\cdot
\end{equation*}
Therefore, the sum appearing in \eqref{h-9} is bounded by 
\begin{equation*}
C_0\, T \, |S|\, \frac{\max_{j\in S} t^{j,N}_{\rm rel}}{\gamma_N}\,
\frac{\pi_{N} (\ms E^k_N)}{\pi_{N} (\ms E^\ell_N)} \;\cdot
\end{equation*}
By the hypotheses of the lemma, this expression vanishes as
$N\uparrow\infty$, which completes the proof. 
\end{proof}

\begin{proof}[Proof of Lemma \ref{h-s6}.]
Fix $\ell\in S$, and let $\nu_N$ be a sequence of probability measures
satisfying \eqref{h-12a}.  By Schwarz inequality, the square of the
expectation appearing in the statement of the lemma is bounded above
by
\begin{equation*}
\frac 1{\pi_N (\ms E_N)} \, E_{\pi_{\ms E}} \Big[ \,
\Big( \frac{d\nu_N}{d\pi_{\ms E}} \Big)^2 \Big] \,
\bb E_{\pi_N}
\Big[ \Big(\int_0^t \chi_{\Delta_N} \big(\, \eta (s\theta_N) \,\big)
\, ds \Big)^2 \, \Big]
\end{equation*}
By assumption \eqref{h-12a}, the first expectation is bounded by
$C_0/\pi_{\ms E}(\ms E^\ell_N)$. On the other hand, by Schwarz
inequality, the second expectation is less than or equal to
\begin{equation*}
t \, \bb E_{\pi_N}
\Big[ \int_0^t \chi_{\Delta_N} \big( \, \eta (s\theta_N) \,\big)
\, ds \, \Big] \;=\; t^2 \pi_N (\Delta_N)\;.
\end{equation*}
The expression appearing in the penultimate displayed formula is thus
bounded above by
$C_0 \, t^2\, [\, \pi_N (\Delta_N)/\pi_{N}(\ms E^\ell_N)\,]$, which
concludes the proof of the lemma.
\end{proof}

\section{Tightness}
\label{sec14}

In this section, we present sufficient conditions for the tightness of
the sequence $\mb P^N$ introduced in Theorems \ref{3-t1}, \ref{5-p1}
and \ref{h-s1}. We need a slight generalization of Lemma \ref{5-p2}.
Recall the notation introduced just before this lemma.  We proved
there that for each $t\ge 0$ and $\eta\in E_N$, $S_{\ms E_N} (t)$ is a
stopping time with respect to the filtration $(\mc F^\eta_t: t\ge 0)$.

\begin{lemma}
\label{I-s2}
Let $\{\mc G_r : r\ge 0\}$ be the filtration given by $\mc G_r = \mc
F^\eta_{S_{\ms E_N}(r)}$, and let $\tau$ be a stopping time with respect to
$\{\mc G_r\}$. Then, $S_{\ms E_N}(\tau)$ is a stopping time with
respect to $\{\mc F^\eta_t\}$.
\end{lemma}

\begin{proof}
Fix a stopping time $\tau$ with respect to the filtration
$\{\mc G_r\}$. This means that for every $t\ge 0$,
$\{\tau\le t\} \in \mc G_t = \mc F^\eta_{S_{\ms E_N}(t)}$. Hence, for all
$r\ge 0$,
\begin{equation*}
\{\tau\le t\} \,\cap\, \{S_{\ms E_N}(t) \le r \} \,\in\,  \mc F^\eta_{r}\;.
\end{equation*}

We claim that $\{S_{\ms E_N}(\tau) <t \} \,\in\,  \mc F^\eta_{t}$. Indeed, by
\eqref{5-2}, this event is equal to $\{T_{\ms E_N} (t)> \tau \}$, which
can be written as
\begin{align*}
& \bigcup_{q\in \bb Q}\, \{\tau \le q \} \, \cap \, \{T_{\ms E_N} (t) > q \}
\;=\; \bigcup_{q\in \bb Q} \, \{\tau \le q \} \, \cap \, \{S_{\ms E_N} (q)<t \} \\
&\qquad \;=\; \bigcup_{q\in \bb Q} \, \bigcup_{n\ge 1}
\{\tau \le q \} \, \cap \, \{S_{\ms E_N} (q) \le t - (1/n)\}\;.
\end{align*}
By the penultimate displayed equation, each term belongs to $\mc
F^\eta_{t-(1/n)} \subset \mc F^\eta_t$, which proves the claim.

We may conclude. Since 
\begin{equation*}
\{S_{\ms E_N}(\tau) \le t \} \;=\; \bigcap_{q} \, \{S_{\ms E_N} (\tau) < t + q\}\;,
\end{equation*}
where the intersection is carried out over all
$q\in (0,\infty)\cap \bb Q$, and since the filtration $\{\mc F^\eta_t\}$ is
right continuous, by the previous claim, $\{S_{\ms E_N}(\tau) \le t \}
\in \mc F^\eta_t$.
\end{proof}

Recall that $\xi_N(t) = \eta_N(t\theta_N)$, and the definition of the
measure $\bb Q^N_\eta$ introduced just before Lemma
\ref{5-p2}. Expectation with respect to this measure is denoted by
$\bb Q^N_\eta$, as well. Note that $\eta^{\ms E_N}(t\theta_N) =
\xi^{\ms E_N}(t)$.

\begin{lemma}
\label{I-s1}
Suppose that for all $t>0$,
\begin{equation}
\label{I-1}
\lim_{N\to\infty} \max_{\xi\in \ms E_N} \bb Q^N_\xi \Big[  
\int_{0}^{t} \chi_{\Delta_N}(\xi_N(s))\, ds \, \Big] 
\;=\; 0\;,
\end{equation}
and that
\begin{equation}
\label{I-2}
\lim_{\delta\to 0}
\limsup_{N\to\infty} \max_{j\in S} \max_{\xi \in\ms E^N_j} 
\bb Q^N_\xi \big[\, H(\breve{\ms E}^j_N) \le \delta  \,\big]  
\;=\; 0 \;.
\end{equation}
Then, the sequence of measures $\mb P^N$ is tight. Moreover, every
limit point $\mb P$ is such that
\begin{equation*}
\mb P \big[\,  X(t) \,\not =\,  X(t-)\,\big] \;=\; 0
\end{equation*}
for every $t>0$.
\end{lemma}

\begin{proof}
Fix $\eta\in \ms E_N$.  According to Aldous' criterion \cite{Bil99}, we
have to show that for every $\delta>0$, $R>0$,
\begin{equation*}
\lim_{a_0\to 0} \limsup_{\epsilon \to 0} \, 
\sup \bb P^N_{\eta} \big[ \, \big |\, \bs X^T_N (\tau+a)  
\,-\, \bs X^T_N (\tau) \,\big | \,>\, \delta\, \big] \;=\; 0\;,
\end{equation*}
where the supremum is carried over all stopping times $\tau$ bounded
by $R$ and all $0\le a<a_0$.  Since $\bs X^T_N (t) = \Psi_N ( \xi^{\ms
  E_N} (t))$, the previous probability can be written as
\begin{equation*}
\bb Q^N_{\eta} \big[ \, \big |\, \Psi_N \big( \xi^{\ms E_N}
(\tau+a) \big)  \,-\, \Psi_N \big( \xi^{\ms E_N}
(\tau) \big) \,\big | \,>\, \delta\, \big] \;.
\end{equation*}

Since $|\Psi_N \big( \xi^{\ms E_N} (\tau+a) \big) \,-\,
\Psi_N \big( \xi^{\ms E_N} (\tau) \big)|>\delta$ entails
that $\Psi_N \big( \xi^{\ms E_N} (\tau+a) \big) \not =
\Psi_N \big( \xi^{\ms E_N} (\tau) \big)$, the expression
in the previous displayed equation is bounded by
\begin{equation*}
\bb Q^N_{\eta} \big[\,  \Psi_N \big( \xi^{\ms E_N} (\tau+a) 
\big) \not = \Psi_N \big( \xi^{\ms E_N} (\tau) \big)\, \big]\;.
\end{equation*}
Fix $b=2a_0$ so that $b-a\ge a_0$. Decompose this probability
according to the event $\{S_{\ms E_N} (\tau+a) - S_{\ms
  E_N} (\tau) > b \}$ and its complement.

Suppose that $S_{\ms E_N} (\tau+a) - S_{\ms E_N} (\tau) > b$. In this
case, $ S_{\ms E_N} (\tau) + b < S_{\ms E_N} (\tau+a)$, so that
$T_{\ms E_N}(S_{\ms E_N} (\tau) + b) \le T_{\ms E_N}(S_{\ms E_N}
(\tau+a)) = \tau+a$. Hence, as $T_{\ms E_N}(S_{\ms E_N} (t)) = t$,
$T_{\ms E_N}(S_{\ms E_N} (\tau) + b) - T_{\ms E_N}(S_{\ms E_N} (\tau))
\le a$, that is,
\begin{equation*}
\int_{S_{\ms E_N} (\tau)}^{S_{\ms E_N} (\tau) + b} 
\chi_{\ms E_N}(\xi_N(s))\, ds \;\le\; a \;. 
\end{equation*}
In other words,
\begin{equation*}
\int_{S_{\ms E_N} (\tau)}^{S_{\ms E_N} (\tau) + b }
\chi_{\Delta_N}(\xi_N (s))\,
ds \;\ge\; b - a \;.  
\end{equation*}

By Lemma \ref{I-s1}, $S_{\ms E} (\tau)$ is a stopping time for the
filtration $\{\mc F^\eta_t\}$. Hence, by the strong Markov property
and since $\xi_N(S_{\ms E_N} (t))$ belongs to $\ms E_N$ for all $t\ge
0$,
\begin{align*}
& \bb Q^N_\eta \big[ \,  S_{\ms E_N} (\tau+a) - S_{\ms E_N}
(\tau) \,> b \, \big]  \;\le\;
\bb Q^N_\eta \Big[ \, 
\int_{S_{\ms E_N} (\tau)}^{S_{\ms E_N} (\tau) + b} 
\chi_{\Delta_N}(\xi_N(s))\, ds \;\ge\; b - a \, \Big]
\\
&\quad \;\le\; \max_{\xi \in \ms E_N} \bb Q^N_\xi \Big[  
\int_{0}^b \chi_{\Delta_N}(\xi_N(s))\, ds 
\;\ge\; b - a \, \Big] \;.
\end{align*}
By Chebychev inequality, a change of variables and by our choice of
$b$, this expression is less than or equal to
\begin{equation*}
\frac 1{(b-a)} \, \max_{\xi \in \ms E_N} \bb Q^N_\xi \Big[  
\int_{0}^{b} \chi_{\Delta_N}(\xi_N(s))\, ds \, \Big] \;\le\;
\frac 1{a_0} \, \max_{\xi\in \ms E_N} \bb Q^N_\xi \Big[  
\int_{0}^{2a_0} \chi_{\Delta_N}(\xi_N(s ))\, ds \, \Big]\;.
\end{equation*}
By assumption \eqref{I-1}, this expression vanishes as $N\to\infty$
for every $a_0>0$.

We turn to the case $\{S_{\ms E_N} (\tau+a) - S_{\ms E_N}
(\tau ) \le b \}$.  On this set we have that
\begin{align*}
& \big\{\, \Psi_N(\xi_N (S_{\ms E_N} (\tau+a))) \,\not =\,
\Psi_N(\xi_N (S_{\ms E_N} (\tau ))) \, \big\} \\
& \;\subset\;
\big\{\, \Psi_N(\xi_N (S_{\ms E_N} (\tau )+c )) \,\not = \,
\Psi_N(X (S_{\ms E} (\tau ))) \text{ for some } 0\le c\le b\,\big
\}\;. 
\end{align*}
Since $S_{\ms E_N} (\tau )$ is a stopping time for the
filtration $\{\mc F_t\}$ and since $\xi_N (S_{\ms E_N} (t))$ belongs
to $\ms E_N$ for all $t$,
\begin{align*}
& \bb Q^N_\eta \Big[ \Psi_N(\xi_N (S_{\ms E_N} (\tau+a))) \not = 
\Psi(\xi_N (S_{\ms E_N} (\tau )))  \,,\,
S_{\ms E_N} (\tau + a) - S_{\ms E} (\tau ) \le b  \Big] \\
&\quad \le\;
\max_{\xi\in\ms E_N} \bb Q^N_\xi \Big[ \, \Psi_N(\xi_N (c )) \not
  = \Psi_N(\xi) \text{ for some } 0\le c\le b \Big] \;.  
\end{align*}
If $\xi\in\ms E^j_N$, this later event corresponds to the event
$\{H(\breve{\ms E}^j_N) \le b \}$. The maximum is thus bounded
by
\begin{equation*}
\max_{j\in S} \max_{\xi \in\ms E^j_N} \bb Q^N_\xi
\big[\, H(\breve{\ms E}^j_N) \le b  \,\big] \;=\;
\max_{j\in S} \sup_{\xi \in\ms E^j_N} \bb Q^N_\xi
\big[\, H(\breve{\ms E}^j_N) \le 2a_0 \,\big]  \;.
\end{equation*}
By assumption \eqref{I-2}, this expression vanishes as $N\to \infty$
and then $a_0\to 0$. This completes the proof of the tightness.

The same argument shows that for every $t>0$, 
\begin{equation*}
\lim_{a_0\to 0} \limsup_{N\to\infty} 
\mb P^N \big[ \, X(t-a) \not = X(t) \text{ for some
} 0\le a\le a_0 \big] \;=\; 0\;.
\end{equation*}
Hence, if $\mb P$ is a limit point of the sequence $\mb
P^N$,
\begin{equation*}
\lim_{a_0\to 0} \mb P \big[ \, X(t-a) \not = X(t) \text{ for some
} 0\le a\le a_0 \big] \;=\; 0\;.
\end{equation*}
This completes the proof of the second assertion of the lemma since
$\{X(t) \not = X(t-)\}\subset \{X(t-a) \not = X(t) \text{ for
  some } 0\le a\le a_0\}$ for all $a_0>0$. 
\end{proof}

Conditions \eqref{I-1}, \eqref{I-2}, can be formulated in terms of
capacities. Next results is Theorem 2.6 in \cite{bl2} and Theorem 2.1
in \cite{bl7}. Note that we do not require the process to be
reversible.

\begin{theorem}
\label{I-s4}
Assume that condition \eqref{6-2} is in force: For all $j\in S$, there
exists $\xi^{j,N} \in \ms E^j_N$ such that
\begin{equation*}
\lim_{N\to\infty} \max_{\eta\in \ms E^j_N \,,\, \eta\not = \xi^{j,N}} 
\frac{\Cap_N(\ms E^j_N, \breve{\ms E}^j_N)}{\Cap_N (\eta, \xi^{j,N})}
\;=\; 0\;.
\end{equation*}
Assume, furthermore, that the coarse-grained jump rates converge: For
all $j\not = k\in S$, there exists $\bs r(j,k) \in [0,\infty)$ such
that
\begin{equation*}
\lim_{N\to\infty} \bs r_N(j,k) \;=\; \bs r(j,k) \;.
\end{equation*}
Let $A\subset S$ be the set of absorbing points of the Markovian dynamics
induced by the rates $\bs r(j,k)$. Assume that for all $j\in A$,
$t>0$, 
\begin{equation*}
\limsup_{N\to\infty} \max_{\xi \in\ms E^N_j} 
\bb Q^N_\xi \Big[  
\int_{0}^{t} \chi_{\Delta_N}(\xi_N(s))\, ds \, \Big] 
\;=\; 0\;.
\end{equation*}
Assume that for all $k\in S\setminus A$,
\begin{equation*}
\lim_{N\to\infty} \frac{\pi_N(\Delta_N)}{\pi_N(\ms E^k_N)} \;=\; 0\;.
\end{equation*}
Then, conditions \eqref{I-1}, \eqref{I-2} hold.
\end{theorem}

This result, which guarantees tightness, together with Theorems
\ref{3-t1}, \ref{l6-1} and Remark \ref{rm6-2}, which provide
uniqueness, yield the convergence of the sequence $\bs X^T_N$.

\begin{theorem}
\label{I-s3}
Fix $k\in S$, a sequence $\eta_N\in \ms E^k_N$, and denote by
$\mb P^N$ the probability measure on $D([0,\infty), S)$ induced by the
process $\bs X^T_N(t)$ and the measure $\bb P^N_{\eta_N}$.  Assume the
hypotheses of Theorem \ref{I-s4}. Then, the sequence $\mb P^N$
converges to the solution of the $(L,\delta_k)$ martingale problem,
where $L$ is the generator of the $S$-valued Markov chain whose jump
rates are $\bs r(j,k)$.
\end{theorem}

\section{The last passage}
\label{sec8}

We prove in this section that the last passage process, introduced in
Definition \ref{l1-1}, converges if conditions (T1), (T2) hold.  In
order to prove this statement, we first define a metric in the path
space $D([0,\infty),S\cup \{{\mf d}\})$ which induces the Skorohod
topology. Assume that $0\not\in S$ and identify the point $\mf d$ with
$0\in \bb Z$ so that $S\cup \{{\mf d}\}$ is a metric space with the
metric induced by $\bb Z$.

For each integer $m\ge 1$, let $\Lambda_m$ denote the class of
strictly increasing, continuous mappings of $[0,m]$ onto itself. If
${\lambda} \in \Lambda_m$, then $\lambda_0=0$ and $\lambda_m=m$. In
addition, consider the function
\begin{equation*}
g_m(t)\;=\;
\left\{
\begin{array}{ll}
1 & \textrm{if \;$t\le m-1$}\;,\\
m-t & \textrm{if \;$m-1 \le t\le m$\;,}\\
0 & \textrm{if \;$t\ge m$}\;.
\end{array}
\right.
\end{equation*}

For any integer $m\ge 1$ and $\omega, \omega' \in D([0,\infty) ,
S\cup\{\mf d\})$, define $d_m(\omega,\omega')$ to be the infimum of
those positive $\epsilon$ for which there exists $\lambda\in \Lambda_m$
satisfying
\begin{equation*}
\sup_{t\in [0,m]} |\lambda_t-t| \;<\; \epsilon \quad\text{and}\quad
\sup_{t\in [0,m]}|\, g_m(\lambda_t) \,\omega(\lambda_t) \,-\, g_m(t) 
\, \omega'(t) \,| \;<\; \epsilon\;.
\end{equation*}
Define the metric $d$ in $D([0,\infty),S\cup \{\mf d\})$ by
\begin{equation*}
d( \omega , \omega') \;=\; \sum_{m=1}^{\infty} 
\frac 1{2^{m}}\, \big\{\, 1\land d_m( \omega , \omega')\, \big\}\;.
\end{equation*}
This metric induces the Skorohod topology in the path space
$D([0,\infty) ,S \cup \{\mf d\})$ \cite{Bil99}. Next result is
Proposition 4.4 in \cite{bl2}.

Recall from \eqref{A-2} the definition of $X^V_N(t)$ and let $\bs
X^V_N(t) = X^V_N(t \theta_N)$. Recall from assumption (T1) the
definition of $\bs X^T_N(t)$

\begin{theorem}
\label{8-t1}
Suppose that $(\eta_N(t) :t\ge 0)$, $N\ge 1$, satisfies condition {\rm
  (T2)}. Then, for any sequence $(\eta_N : N\ge 1)$,  $\eta_N \in \ms E_N$,
\begin{equation*}
\lim_{N\to\infty} {\bb E}_{\eta_N}\big[ \,d(\bs X^V_N , \bs X^T_N)\, \big] \;=\; 0\;.
\end{equation*}
\end{theorem}

It follows from this result that the last-passage process $\bs
X^V_N(t)$ converges whenever the trace process $\bs X^T_N(t)$
converges and (T2) is in force.

\section{The finite-dimensional distributions}
\label{sec9}

Recall the definition of the process $X_N(t)$ defined in \eqref{A-1},
and the one of the reduced model $\bs X(t)$ introduced in Definition
\ref{l1-1}.  Next result is Proposition 1.1 of \cite{llm}.

\begin{theorem}
\label{9T1}
Assume that conditions (T1) and (T2) of Definition \ref{l1-2} are in
force, and that
\begin{equation}
\label{9-1}
\lim_{\delta\to 0} \limsup_{N\to 0} \sup_{\eta \in \ms E_N} 
\sup_{\delta\le s\le 2\delta} \bb P^N_{\eta} 
\big[\, \eta(s\theta_N) \in \Delta_N \,\big]\;=\;0\;. 
\end{equation}
Then, the finite-dimensional distributions of $\bs X_N (t) =
X_N(t\theta_N)$ converge to the finite-dimensional distributions of
$\bs X(t)$.
\end{theorem}

With further mixing conditions one can prove that the state of the
process at time $t\theta_N$ is a time-dependent convex combinations of
states supported in the valleys. 

Denote by $p_t(j,k)$ the transition probabilities of the reduced model
$\bs X(t)$, by $\pi^k_N$ the measure $\pi_N$ conditioned to $\ms
E^k_N$, and by $\Vert \mu - \nu \Vert_{\rm TV}$ the total variation
distance between two probability measures $\mu$ and $\nu$ defined on
$E_N$. Let $(S_N(t): t\ge 0)$ be the semigroup associated to the
Markov chain $\eta_N(t)$. Then, under mixing conditions specified in
\cite{llm}, for every $j\in S$ and sequence $\eta_N\in\ms E^N_j$,
\begin{equation*}
\lim_{N\to\infty} \big\| \delta_{\eta_N} \, S_N(t\theta_N)
\,-\, \sum_{k\in S} p_t(j,k)  \,\pi^k_N \, \big\|_{\rm TV}\;=\;0\;,
\end{equation*}
where $\delta_{\eta}$, $\eta\in E_N$, stands for the Dirac measure
concentrated on the configuration $\eta$.

\section{Markov chains}
\label{sec13}

We briefly present in this section some results on Markov chains used
in the article.  Fix a
finite set $\color{blue} E$. Consider a continuous-time, $E$-valued,
Markov chain $\color{blue} ( \eta(t) : t \geq 0 )$. Assume that the
chain $\eta(t)$ is irreducible and denote by $\color{blue} \pi$ the
unique stationary state.

Elements of $E$ are represented by the letters $\eta$, $\xi$. Let
$\color{blue} \bb P_\eta$, $\eta\in E$, be the probability measure on
$D([0,\infty), E)$ induced by the Markov chain $\eta(t)$ starting from
$\eta$. Recall from \eqref{1-5} the definition of the hitting time and
the return time to a set.

Denote by $\color{blue} R(\eta,\xi)$, $\eta\not = \xi\in E$, the jump
rates of the Markov chain $\eta(t)$, and let
${\color{blue}\lambda (\eta)} = \sum_{\xi\in E} R (\eta,\xi)$ be the
holding rates. Denote by $\color{blue} p (\eta,\xi)$ the jump
probabilities, so that
$R (\eta,\xi) = \lambda (\eta) \, p (\eta,\xi)$. The stationary state
of the embedded discrete-time Markov chain is given by
$\color{blue} M (\eta) = \pi (\eta) \, \lambda (\eta)$.

Denote by $L$ the generator of the Markov chain
$\eta(t)$,
\begin{equation*}
{\color{blue} (Lf)(\eta)} \;=\; \sum_{\xi\in E} R(\eta,\xi)\, 
[\, f(\xi) - f(\eta)\,]\;.
\end{equation*}
Let $L^2(\pi)$ be the set of square-summable functions $f:E \to \bb R$
endowed with the scalar product $\<\,\cdot\,,\,\cdot\, \>_{\pi}$ given
by
\begin{equation*}
{\color{blue} \<\, f \,,\, g\, \>_{\pi}} 
\;:=\; \sum_{\eta\in E_N} f(\eta)\, g(\eta)\, \pi(\eta)\;,
\quad {\color{blue} \Vert\, f \,\Vert^2} 
\;=\; \<\, f \,,\, f\, \>_{\pi} \;.
\end{equation*}
Denote by $L^*$ the adjoint of the operator $L$ in $L^2(\pi)$: For all
functions $f$, $g:E \to \bb R$,
\begin{equation}
\label{M-2}
\<\, L^*\, f \,,\, g\, \>_{\pi} \;=\; \<\, f \,,\, L\, g\, \>_{\pi} \;.
\end{equation}
An elementary computation yields that
\begin{equation*}
{\color{blue} (L^*f)(\eta)} \;=\; \sum_{\xi\in E} R^*(\eta,\xi)\, 
[\, f(\xi) - f(\eta)\,]\;,
\end{equation*}
where the jump rates $R^*(\eta,\xi)$ satisfy
\begin{equation*}
\pi(\eta)\,  {\color{blue} R^*(\eta,\xi)} 
\;=\; \pi(\xi)\,  R(\xi, \eta)\;, \quad \eta \,\not = \, \xi\,\in\, E\; .
\end{equation*}

The chain is said to be \emph{\color{blue} reversible} if the
generator $L$ is self-adjoint: $L^*=L$. It is reversible if and only
if the jump rates satisfy the \emph{\color{blue} detailed balance
  conditions}:
\begin{equation}
\label{M-10}
\pi(\eta)\,  R (\eta,\xi)
\;=\; \pi(\xi)\,  R(\xi, \eta)\;, \quad \eta \,\not = \, \xi\,\in\, E\; .
\end{equation}

The operator $L^*$ corresponds to the generator of a Markov chain,
represented by $\color{blue} \eta^*(t)$, and called the adjoint or
time-reversed process. The holding rates
${\color{blue} \lambda^*(\eta)} = \sum_{\xi\in E} R^*(\eta,\xi)$ of
this chain coincide with the original ones,
$\lambda^*(\eta) = \lambda(\eta)$, and the jump probabilities
$p^*(\eta\,\xi)$ satisfy the balance conditions
\begin{equation}
\label{J-4}
M(\eta)\, {\color{blue} p^*(\eta,\xi) }\;=\; M(\xi)\, p(\xi, \eta)
\;, \quad \eta \,\not = \, \xi\,\in\, E \;.
\end{equation}

Let $\color{blue} L^s$ be the symmetric part of the generator $L$:
\begin{equation}
\label{M-9}
L^s \;=\; \frac 12\, \{\, L \,+\, L^*\,\} \;.
\end{equation}
The operator $L^s$ is self-adjoint in $L^2(\pi)$ and it corresponds to
the generator of the Markov chain whose jump rates, denoted by
$R^s(\eta,\xi)$, are given by ${\color{blue} R^s(\eta,\xi)} = (1/2)
\{R(\eta,\xi) + R^*(\eta,\xi)\}$. A simple computation shows that
these rates satisfy the detailed balance conditions \eqref{M-10}.

Denote by $D(f)$ the Dirichlet form of a function $f:E \to
\bb R$:
\begin{equation}
\label{M-8}
{\color{blue}  D(f)} \;:=\; \<\, (-\, L) \, f \,,\, f \,\>_\pi
\;=\; \<\, (-\, L^s) \, f \,,\, f \,\>_\pi\;. 
\end{equation}
We leave to the reader the assignment of checking the last equality.
An elementary computation shows that
\begin{equation}
\label{M-5}
D(f) \;=\; \frac 12\, \sum_{\eta\in E} \sum_{\xi\in E}
\pi(\eta)\,  R(\eta,\xi)\, [\, f(\xi) - f(\eta)\,]^2 \;.
\end{equation}
This formula holds even in the non-reversible case.  In the sum, each
unordered pair $\{\eta, \xi\} \subset E$, $\xi\not = \eta$, appears
twice.

Denote by $\color{blue} (S(t) : t\ge 0)$, the semigroup associated to
the generator $L$, so that $(d/dt) S(t) = L\, S(t) = S(t)\, L$. Fix a
probability measure $\nu$ on $E$ and let $f_t$ be the Radon-Nikodym
derivative of $\nu S(t)$ with respect to $\pi$. We claim that
\begin{equation}
\label{M-1}
\frac d{dt} f_t \;=\; L^* f_t\;.
\end{equation}
Indeed, fix a function $g:E \to \bb R$ and consider the mean $\bb
E_{\nu}[g(\eta(t))]$, where $\color{blue} \bb E_{\nu}$ represents the
expectation with respect to the measure $\color{blue} \bb P_\nu =
\sum_{\eta\in E} \nu(\eta) \, \bb P_\eta$. This expectation can be
written as
\begin{equation*}
\sum_{\eta\in E} \nu(\eta) \, [S(t) \,g] (\eta) \;=\;
\sum_{\eta\in E} [\nu\, S(t)](\eta) \, g (\eta) \;=\;
\sum_{\eta\in E} \pi(\eta) \, f_t (\eta)\, g (\eta) \;=\;
\<\, f_t \,,\, g\, \>_\pi\;.
\end{equation*}
As $(d/dt) S(t)g = S(t)\, L\, g$, taking derivative on both sides of
this identity we get that
\begin{equation*}
\sum_{\eta\in E} \nu(\eta) \, [S(t) \, L\, g] (\eta) \;=\;
\<\, \frac {d}{dt}\, f_t \,,\, g\,\>_\pi\;.
\end{equation*}
The left-hand side can be written as $\<\, f_t \,,\, L\, g\, \>_\pi \;=\;
\<\, L^* f_t \,,\, g \,\>_\pi$. Hence, for all functions $g$, $\<\,
(d/dt) f_t \,,\, g \,\>_\pi \;=\; \<\, L^* f_t \,,\, g\,\>_\pi$, which
proves claim \eqref{M-1}.

By \eqref{M-1} and \eqref{M-2},
\begin{equation*}
\frac {d}{dt}\, \<\, f_t \,,\, f_t \,\>_\pi \;=\; 
2\, \<\, L \, f_t \,,\, f_t \,\>_\pi 
\;=\; -\, 2\, D(f_t) \;\le\; 0\;.
\end{equation*}
The inequality follows from the positiveness of the Dirichlet form
derived in \eqref{M-5}. Integrating in time yields that
\begin{equation*}
\Vert\, f_t \,\Vert^2 \;+\; 
2\, \int_0^t  D(f_s)\, ds  \;\le\; \Vert\, f_0 \,\Vert^2 \;.
\end{equation*}
In particular, for all $0\le s\le t$,
\begin{equation}
\label{M-6}
\<\, f_t \,,\, f_t \,\>_\pi \;\le\; \<\, f_s \,,\, f_s \,\>_\pi \;.
\end{equation}

The \emph{\color{blue} spectral gap} of the generator, denoted by $\mf
g$, is the value of the smallest positive eigenvalue of the symmetric
part of the generator:
\begin{equation*}
\mf g \;=\; \inf_{f} \frac{\<\, (-\, L) \, f \,,\, f \,\>_\pi}
{\<\, \, f \,,\, f \,\>_\pi} \;,
\end{equation*}
where the infimum is carried over all functions $f:E\to \bb R$ which
are orthogonal to the constants, i.e., which have zero-mean with
respect to $\pi$: $E_\pi[f] = \<\, \, f \,,\, 1 \,\>_\pi =0$.

\subsection{Reflected chain}

Fix a non-empty, proper subset $F$ of $E$.  Denote by $(\eta^{R,F} (t)
: t\ge 0)$, the Markov chain $\eta (t)$ reflected at $F$. This is the
$F$-valued process obtained from $\eta(t)$ by forbidding all jumps
between $F$ and $E\setminus F$. The generator $L_{R, F}$ of this
Markov process is given by
\begin{equation*}
{\color{blue} (L_{R,F} f)\, (\eta)} \,=\, \sum_{\xi\in  F} R(\eta,\xi)
\, \big\{f(\xi)-f(\eta)\big\}\;, \quad \eta\in F\;.
\end{equation*}

Assume that the reflected process $\eta^{R,F} (t)$ is irreducible. It
is easy to show that the conditioned probability measure $\pi_{F}$
defined by
\begin{equation}
\label{M-3}
\pi_F (\eta) \;=\; \frac{\pi(\eta)}{\pi(F)}\;,
\quad \eta\in F \;, 
\end{equation}
satisfies the detailed balance conditions \eqref{M-10} for the
reflected process if the chain is reversible.

In general, $\pi_{F}$ may not be invariant. Consider, for example, an
asymmetric random walk on the circle. The uniform measure is
invariant, but its restriction to an interval $I$ is not invariant for
the process reflected at $I$.  For cycle generators, however, it is
possible to reflect the chain preserving the stationary state.

\subsection{Cycle generators}
\label{12.2}

The results of this subsection are taken from Section 4 of
\cite{lx15}. We refer to \cite{ls2018} for an application.

\smallskip\noindent\emph{\color{blue} Cycle}: A cycle is a sequence of
distinct configurations $(\eta_0, \eta_1, \dots, \eta_{n-1},
\eta_n=\eta_0)$ whose initial and final configuration coincide:
$\eta_i \not = \eta_j\in E$, $i\not = j\in \{0, \dots, n-1\}$. The
number $n$ is called the length of the cycle.

\smallskip\noindent\emph{\color{blue} Cycle generator}: A generator
$L$ is said to be a cycle generator associated to the cycle $\mf c =
(\eta_0, \eta_1, \dots, \eta_{n-1}, \eta_n=\eta_0)$ if there exists
reals $r_i>0$, $0\le i<n$, such that
\begin{equation*}
R(\eta,\xi) \;=\; 
\begin{cases}
r_i & \text{if $\eta=\eta_i$ and $\xi=\eta_{i+1}$ for some $0\le i
  <n$}\;, \\
0 & \text{otherwise}\;.
\end{cases}
\end{equation*}
We denote this cycle generator by $\mc L_{\mf c, {\bs r}}$, where $\bs
r = (r_0, \dots, r_{n-1})$. Most of the time we omit the dependence on
$\bs r$ and write $\mc L_{\mf c, {\bs r}}$ simply as $\mc L_{\mf
  c}$. Note that
\begin{equation*}
{\color{blue} (\mc L_{\mf c, {\bs r}}  f) \, (\eta) \;=\;
(\mc L_{\mf c} f)\, (\eta)} \;=\; \sum_{i=0}^{n-1} \chi_{\{\eta_i\}}
(\eta) \, r_i \, [f(\eta_{i+1}) - f(\eta_i)]\;,
\end{equation*}
and that the chain is irreducible only if $\{\eta_0, \eta_1, \dots,
\eta_{n-1}\}=E$.

Consider a cycle $\mf c = (\eta_0, \eta_1, \dots, \eta_{n-1},
\eta_n=\eta_0)$ of length $n\ge 2$ and let $\mc L_{\mf c}$ be a cycle
generator associated to $\mf c$. Denote the jump rates of $\mc L_{\mf
  c}$ by $R(\eta_i,\eta_{i+1})$. A measure $\pi$ is stationary for
$\mc L_{\mf c}$ if and only if
\begin{equation}
\label{M-7}
\pi(\eta_i) \, R(\eta_i,\eta_{i+1}) \;\; \text{is constant}\;.  
\end{equation}

\smallskip\noindent\emph{\color{blue} Sector condition}: 
Next lemma asserts that every cycle generator satisfies a sector
condition.  The proof of this result can be found in \cite[Lemma
5.5.8]{klo12}.

\begin{lemma}
\label{M-s1}
Let $\mc L_{\mf c}$ be a cycle generator associated to a cycle $\mf c$ of length
$n$. Then, $\mc L_{\mf c}$ satisfies a sector condition with constant $2n$: For all
$f$, $g:E\to \bb R$,
\begin{equation*}
\<\, \mc L_{\mf c} \, f\, ,\, g \,\>^2_\pi \;\le\; 2n \, 
\<\, (-\, \mc L_{\mf c} \, f) \,,\, f \,\>_\pi\, 
\<\, (-\,\mc L_{\mf c} \, g) \,,\, g \,\>_\pi\;. 
\end{equation*}
\end{lemma}

\smallskip\noindent\emph{\color{blue} Cycle decomposition}: Every
generator $L$, stationary with respect to a probability measure $\pi$,
can be decomposed as the sum of cycle generators which are stationary
with respect to $\pi$.

\begin{lemma}
\label{M-s2}
Let $L$ be a generator of an $E$-valued, irreducible Markov chain.
Denote by $\pi$ the unique invariant probability measure.  Then, there
exists cycles $\mf c_1, \dots, \mf c_p$ such that
\begin{equation*}
L \;=\; \sum_{j=1}^p \mc L_{\mf c_j}\;,
\end{equation*}
where $\mc L_{\mf c_j}$ are cycle generators associated to $\mf c_j$ which
are stationary with respect to $\pi$.
\end{lemma}

\begin{proof}
The proof consists in eliminating successively all $2$-cycles (cycles
of length $2$), then all $3$-cycles and so on up to the $|E|$-cycle if
there is one left. Denote by $R(\eta,\xi)$ the jump rates of the
generator $L$ and by $\bb C_2$ the set of all $2$-cycles
$(\eta,\xi, \eta)$ such that $R(\eta,\xi) R(\xi,\eta)>0$. Note that
the cycle $(\eta,\xi,\eta)$ coincides with the cycle $(\xi,\eta,\xi)$.

Fix a cycle $\mf c = (\eta,\xi,\eta)\in\bb C_2$. Let $\bar c(\eta,\xi) =
\min\{ \pi(\eta) R(\eta,\xi) , \pi(\xi) R(\xi,\eta)\}$ be the minimal
conductance of the edge $(\eta,\xi)$, and let $R_{\mf c} (\eta,\xi)$
be the jump rates given by $R_{\mf c} (\eta,\xi) =
\bar c(\eta,\xi)/\pi(\eta)$, $R_{\mf c} (\xi,\eta) =
\bar c(\eta,\xi)/\pi(\xi)$. Observe that $R_{\mf c} (\zeta,\zeta') \le R
(\zeta,\zeta')$ for all $(\zeta,\zeta')$, and that $R_{\mf c} (\xi,\eta) = R
(\xi,\eta)$ or $R_{\mf c} (\eta,\xi) = R (\eta,\xi)$.

Denote by $\mc L_{\mf c}$ the generator associated the the jump rates
$R_{\mf c}$. Since $\pi(\eta) R_{\mf c} (\eta,\xi) = \bar c(\eta,\xi) =
\pi(\xi) R_{\mf c} (\xi,\eta)$, by \eqref{M-7}, $\pi$ is a stationary
state for $\mc L_{\mf c}$ (actually, reversible). Let $\mc L_1 = \mc L - \mc L_{\mf
  c}$ so that
\begin{equation*}
L \;=\; L_1 \;+\; \mc L_{\mf c}\;.
\end{equation*}
As $R_{\mf c} (\zeta,\zeta') \le R (\zeta,\zeta')$, $L_1$ is the
generator of a Markov chain. Since both $L$ and $\mc L_{\mf c}$ are
stationary for $\pi$, so is $L_1$. Finally, if we draw an arrow
from $\zeta$ to $\zeta'$ if the jump rate from $\zeta$ to $\zeta'$ is
strictly positive, the number of arrows for the generator $L_1$ is
equal to the number of arrows for the generator $L$ minus $1$ or $2$.
This procedure has therefore strictly decreased the number of arrows
of $L$.

We may repeat the previous algorithm to $L_1$ to remove from $L$ all
$2$-cycles $(\eta,\xi, \eta)$ such that $R(\eta,\xi) R(\xi,\eta)>0$.
Once this has been accomplished, we may remove all $3$-cycles
$(\eta_0,\eta_1, \eta_2, \eta_3=\eta_0)$ such that $\prod_{0\le i <3}
R(\eta_i,\eta_{i+1}) >0$. At each step at least one arrow is removed
from the generator which implies that after a finite number of steps
all $3$-cycles are removed.

Once all $k$-cycles have been removed, $2\le k<|E|$, we have obtained
a decomposition of $L$ as
\begin{equation*}
L \;=\; \sum_{k=2}^{|E|-1} \mc L_k \;+\; \hat {L}\;,
\end{equation*}
where $\mc L_k$ is the sum of $k$-cycle generators and is stationary
with respect to $\pi$, and $\hat {L}$ is a generator, stationary with
respect to $\pi$, and with no $k$-cycles, $2\le k<|E|$. If $\hat {L}$ has
an arrow, as it is stationary with respect to $\pi$ and has no
$k$-cycles, $\hat {L}$ must be an $|E|$-cycle generator, providing the
decomposition stated in the lemma.
\end{proof}

\begin{corollary}
\label{M-s3}
The generator $L$ satisfies a sector condition with constant bounded
by $2|E|$: For all
$f$, $g:E\to \bb R$,
\begin{equation*}
\<\,  L \,  f \,,\, g \,\>^2_\pi \;\le\; 2|E| \, 
\< \, (- \, L \, f) \,,\, f \,\>_\pi\, 
\<\, (- \, L \, g) \,,\, g \,\>_\pi\;. 
\end{equation*} 
\end{corollary}

\begin{proof}
Fix $f$ and $g:E\to \bb R$. By Lemma \ref{M-s2},
\begin{equation*}
\< Lf, g\>^2_\pi \;=\; \Big( \sum_{j=1}^p \< \mc L_{\mf c_j} f, g\>_\pi
\Big)^2  \;,
\end{equation*}
where $\mc L_{\mf c_j}$ is a cycle generator, stationary with respect to
$\pi$, associated to the cycle $\mf c_j$. By Lemma \ref{M-s1} and by
Schwarz inequality, since all cycles have length at most $|E|$, the
previous sum is bounded by
\begin{equation*}
2|E| \, \sum_{j=1}^p \< (- \, \mc L_{\mf c_j} f), f\>_\pi 
\, \sum_{k=1}^p \< (-\, \mc L_{\mf c_k} g), g\>_\pi \;=\; 
2|E| \, \< (-L f), f\>_\pi \, \< (-L g), g\>_\pi\;,
\end{equation*}
as claimed
\end{proof}

\begin{remark}
\label{M-r1}
A generator $L$ is reversible with respect to $\pi$ if and only if it
has a decomposition in $2$-cycles. Given a measure $\pi$ on a finite
state space, by introducing $k$-cycles satisfying
\eqref{M-7} it is possible to define non-reversible dynamics which are
stationary with respect to $\pi$. The previous lemma asserts that this
is the only way to define such dynamics.
\end{remark}

\begin{remark}
\label{M-r2}
The decomposition in cycles is not unique. There may exist cycles and
vectors $\mf c_1, \dots, \mf c_p$, $\bs r_1, \dots, \bs r_p$ and
$\hat{\mf c}_1, \dots, \hat{\mf c}_q$, $\hat{\bs r}_1, \dots, \hat{\bs
  r}_q$ such that $\{\mf c_1, \dots, \mf c_p\} \not = \{\hat{\mf c}_1,
\dots, \hat{\mf c}_q\}$,
\begin{equation*}
L \;=\; \sum_{j=1}^p \mc L_{\mf c_j, \bs r_j} \;=\; \sum_{k=1}^q \mc
L_{\hat{\mf c}_k, \hat{\bs r}_k}\;,
\end{equation*}
and $\pi$ is a stationary state for all cycle generators.  We leave
the reader to find an example. However, in view of Lemma \ref{M-s1},
it is natural to look for one which minimizes the length of the
longest cycle.
\end{remark}

\begin{remark}
\label{M-r3}
In a finite set, the decomposition of a generator into cycle
generators is very simple. The problem for countably-infinite sets is
much more delicate. We refer to \cite{gv12} for a discussion.
\end{remark}

Let $F$ be a proper subset of $E$ and consider the chain reflected at
$F$.  The last result of this subsection provides sufficient
conditions for the measure $\pi$ conditioned to $F$ to be a stationary
state for the reflected process in the non-reversible case.

\begin{lemma}
\label{M-s4}
Assume that the generator $L$ can be written as a sum of cycle
generators:
\begin{equation*}
L \;=\; \sum_{j=1}^p \mc L_{\mf c_j}\;,
\end{equation*}
where $\mf c_1, \dots, \mf c_p$ are cycles and $\pi$ is a stationary
state for each $\mc L_{\mf c_j}$. Then, the measure $\pi$ conditioned
to $F$ is stationary for the reflected chain at $F$ if there exists a
subset $A$ of $\{1, \dots, p\}$ such that
\begin{equation*}
L_{R, F} \;=\; \sum_{j\in A} \mc L_{\mf c_j}\;.
\end{equation*}
\end{lemma}

\begin{proof}
Since $\pi$ is a stationary state for each $\mc L_{\mf c_j}$, it is
also a stationary state for $L_{R, F} = \sum_{j\in A} \mc L_{\mf
  c_j}$. As the reflected process does not leave the set $F$, the
measure $\pi$ is stationary if and only if its restriction to $F$ is
stationary. 
\end{proof}

\subsection{Enlarged chains}

Let $E^\star$ be a copy of $E$. The elements of $E^\star$ are
represented by the letters $\eta$, $\xi$. Denote by $P_\star: E \cup
E^\star \to E \cup E^\star$ the application which maps a configuration
in $E$, $E^\star$, to its copy in $E^\star$, $E$, respectively.

Following \cite{bg16}, for $\gamma >0$ denote by $\eta^\gamma (t)$ the
Markov process on $E \cup E^\star$ whose jump rates
$R^\gamma (\eta,\xi)$ are given by
\begin{equation*}
R^\gamma (\eta,\xi) \;=\; 
\begin{cases}
R(\eta,\xi) & \text{if $\eta$ and $\xi\in E$,} \\
1/\gamma & \text{if $\xi = P_\star \eta$,} \\
0 & \text{otherwise.} 
\end{cases}
\end{equation*}
Therefore, being at some state $\xi$ in $E^\star$, the process may
only jump to $P_\star \xi$ and this happens at rate $1/\gamma$. In
contrast, being at some state $\xi$ in $E$, the process
$\eta^\gamma (t)$ jumps with rate $R (\xi, \xi')$ to the state
$\xi'\in E$, and jumps with rate $1/\gamma$ to $P_\star\xi$.  We call
the process $\eta^\gamma (t)$ the \emph{\color{blue}
  $\gamma$-enlargement} of the process $\eta(t)$.

Let $\pi_\star$ be the probability measure on $E \cup E^\star$ defined
by
\begin{equation*}
\pi_\star (\eta) \;=\; \pi_\star (P_\star\eta) 
\;=\; (1/2)\, \pi (\eta)  \; , \quad \eta\in E\;.
\end{equation*}
The probability measure $\pi_\star$ is invariant for the enlarged
process $\eta^\gamma (t)$ and it is reversible whenever $\pi$ is
reversible.  

Let $F$ be a subset of $E$. Think of $F$ as a valley.  If $\gamma$ is
much larger than the mixing time, the distribution of $\eta (
H_{F_\star})$, where $F_\star = \{P_\star\eta : \eta \in F\}$, is very
close the stationary state conditioned to $F$.

\subsection{Collapsed chains}

The collapsed chain consists in collapsing a subset of the state-space
to a point and in the defining a dynamics which keeps the properties
of the original evolution as much as possible.  This is a well-known
technique, see for instance \cite{br58, ab1}.

Fix a subset $A$ of $E$, and let ${\color{blue} E_A} := [E \setminus
A] \cup \{\mf d\}$, where $\color{blue} \mf d$ stands for an extra
configuration added to $E$ and meant to represent the collapsed set
$A$. Denote by $\color{blue} (\eta^{C,A} (t) : t\ge 0)$ the chain
obtained from $\eta(t)$ by collapsing the set $A$ to the singleton
$\{\mf d\}$. This is the continuous-time Markov chain on $E_A$ with
jump rates $R^{C,A}(\eta,\xi)$, $\eta$, $\xi\in E_A$, given by
\begin{equation}
\label{L-1}
\begin{gathered}
{\color{blue} R^{C,A}(\eta,\xi)} \;=\; R(\eta,\xi) \;, 
\quad R^{C,A}(\eta,\mf d) \;=\; \sum_{\zeta\in A} R(\eta,\zeta) \;, 
\quad \eta\;,\; \xi \;\in\; E \setminus A\;, \\
R^{C,A}(\mf d,\eta) \;=\; \frac 1{\pi(A)} \sum_{\xi\in A} \pi(\xi) \, 
R(\xi,\eta) \;, \quad \eta \in E \setminus A \;. 
\end{gathered}
\end{equation}

The collapsed chain $\{\eta^{C,A}(t) : t\ge 0\}$ inherits the
irreducibility from the original chain. Denote by $\pi^{C,A}$ the
probability measure on $E_A$ given by
\begin{equation}
\label{L-2}
{\color{blue} \pi^{C,A} (\mf d)} \;=\; \pi(A)\;, \quad 
\pi^{C,A}(\eta) \;=\; \pi(\eta)\;, \quad \eta \in E \setminus A \;.
\end{equation}
Since
\begin{equation*}
\sum_{\xi\not\in A , \zeta\in A} \pi(\xi)\, R(\xi,\zeta) \;=\; 
\sum_{\xi\not\in A , \zeta\in A} \pi(\zeta)\, R (\zeta,\xi)\;,
\end{equation*}
one checks that $\pi^{C,A}$ is a stationary state, and therefore the
unique invariant probability measure, for the collapsed chain
$\eta^{C,A}(t)$.

The collapsed chain has to be understood as follows. Until the process
hits the set $A$, it evolves as the original one. When it reaches this
set, it immediately equilibrates and its position is replaced by the
stationary distribution conditioned to $A$.

In particular, we may couple the collapsed process with the original
one until the set $A$ is reached, so that, for every $\eta\in
E\setminus A$, and $B\subset E\setminus A$,
\begin{equation}
\label{L-4}
\bb P^{C,A}_\eta \big[\, H_{\mf d} < H^+_B\,\big] \;=\;
\bb P_\eta \big[\, H_{A} < H^+_B\,\big]\;,
\end{equation}
provided $\bb P^{C,A}_\eta$ represents the distribution of the
collapsed chain $\eta^{C,A}(t)$ starting from $\eta$. It follows from
this identity and the explicit formulae for the jump rates and the
stationary state that for every $B\subset E\setminus A$,
\begin{equation*}
\Cap (A, B) \;=\; \Cap^{C,A} (\mf d, B)\;,
\end{equation*}
where $\Cap^{C,A} (\mf d, B)$ represents the capacity between $\mf d$
and $B$ for the collapsed chain. 

This identity ceases to hold if we replace $A$ by a set in
$E\setminus A$ because \eqref{L-4} is incorrect if $\mf d$, $A$ are
replaced by a set $D\subset E\setminus A$.

\smallskip

Denote by $\color{blue} L^{C,A}$ the generator of the chain
$\eta^{C,A}(t)$. Fix two functions $f$, $g:E_A\to \bb R$. Let $F$,
$G:E\to \bb R$ be defined by 
\begin{equation*}
F(\eta) \;=\; f(\eta)\;, \quad \eta\in E\setminus
A \;, \quad F(\zeta) = f(\mf d)\;, \quad \zeta\in A\;,
\end{equation*}
with a similar definition for $G$. We claim that
\begin{equation}
\label{L-3}
\<\, L^{C,A} f \,,\, g \,\>_{\pi^{C,A}} 
\;=\; \<\, L F \,,\, G \,\>_{\pi} \;.
\end{equation}
Conversely, if $F$, $G:E\to \bb R$ are two functions constant over
$A$, \eqref{L-3} holds if we define $f$, $g: E_A\to \bb R$ by 
\begin{equation*}
f(\eta) \;=\; F(\eta)\;, \quad \eta\in E\setminus A\;, 
\quad f(\mf d) \;=\; F(\zeta) \quad \text{for some}
\quad \zeta\in A\;,
\end{equation*}
with an analogous equation for $f$, $F$ replaced by $g$, $G$,
respectively. 

To prove \eqref{L-3}, fix two functions $f$, $g: E_A\to \bb R$. By
definition of $L^{C,A}$,
\begin{equation*}
\< \,L^{C,A} f \,,\, g \,\>_{\pi^{C,A}} \;=\; 
\sum_{\eta, \xi \in E_A} \pi^{C,A}(\eta) \,
R^{C,A}(\eta,\xi) \, [\, f(\xi)-f(\eta)\,]\, g(\eta)\;. 
\end{equation*}
In view of \eqref{L-1}, \eqref{L-2}, this expression is equal to
\begin{equation*}
\begin{split}
& \sum_{\eta\in E\setminus A} \pi (\eta) \,
\Big\{ \sum_{\xi\in E\setminus A} 
R(\eta,\xi) \, [\, f(\xi) - f(\eta) \,]  \;+\; 
\sum_{\zeta\in A} R(\eta,\zeta) \, [\, f(\mf d) - f(\eta)\,]
\, \Big\} \, g(\eta) \\
& \qquad \;+\; \sum_{\xi\in E\setminus A} 
\sum_{\zeta\in A} \pi(\zeta)\, R (\zeta,\xi) \,
[\,f(\xi) \,-\, f(\mf d)\,] \, g(\mf d) \;.
\end{split}
\end{equation*}
Since $F(\eta)=f(\eta)$ for $\eta\in E\setminus A$, and $F(\xi)=f(\mf
d)$ for $\xi\in A$, with similar identities with $G$, $g$ replacing
$F$, $f$, the last sum is equal to
\begin{equation*}
\begin{split}
& \sum_{\eta\in E\setminus A} \pi (\eta) \,
\Big\{ \sum_{\xi\in E\setminus A} 
R(\eta,\xi) \, [\, F(\xi) - F(\eta)\,] \,+\,
\sum_{\zeta\in A} R(\eta,\zeta) \, [\,F(\zeta) - F(\eta)\,] \, \Big\}
\, G(\eta) \\
&\qquad \;+\; \sum_{\zeta\in A}  \sum_{\xi\in E\setminus A} 
\pi(\zeta) \, R(\zeta,\xi) \, [\,F(\xi)-F(\zeta)\,] \, G(\zeta)\;.
\end{split}
\end{equation*}
Since $F$ is constant on $A$, we may add to this expression
\begin{equation*}
\sum_{\eta\in A}  \sum_{\xi\in A} \pi(\eta) \, 
R(\eta,\xi) \, [F(\xi)-F(\eta)] \, G(\eta)
\end{equation*}
to obtain that the last displayed expression is equal to $\<LF ,
G\>_\pi$, which concludes the proof of the first assertion of
\eqref{L-3}. The second statement is obtained following the computation
in the reverse order. 

\section{Potential theory}
\label{sec10}

In this section, we present general results on the potential theory of
continuous-time Markov chains used throughout the article. 

Reversible Markov chains can be interpreted in terms of electrical
circuits. This description may provide some intuition on the notions
introduced below, as Dirichlet form, capacity or equilibrium
potential. We refer to the monographs of Doyle and Snell
\cite{DoySne84} and Gaudilli\`ere \cite{Gau09}. The analogy has been
extended to the non-reversible context by Bal\'azs and Folly
\cite{BalFol16}.

\subsection{The capacity}

Fix two non-empty subsets $A$, $B$ of $E$ such that $A \cap B =
\varnothing$.  The capacity between $A$ and $B$, denoted by
$\Cap(A,B)$, is given by
\begin{equation} 
\label{J-1}
{\color{blue} \Cap (A,B)} \;:=\; \sum_{\eta \in A} M (\eta) \, 
\bb P_\eta [ H_B < H_A^+]\; . 
\end{equation}

The capacity is monotone in the second coordinate. Let $B'$ be a
subset of $E$ such that $A\cap B' = \varnothing$, $B\subset B'$. Since
$\bb P_\eta [ H_B < H_A^+] \le \bb P_\eta [ H_{B'} < H_A^+]$, we have
that
\begin{equation}
\label{J-6}
\Cap (A,B) \;\le\; \Cap (A,B') \;.
\end{equation}

By \eqref{J-4}, for any sequence of configurations
$\eta_0, \eta_1, \dots, \eta_n$ such that $p(\eta_i, \eta_{i+1})>0$,
$0\le i<n$,
\begin{equation*}
M(\eta_0) \, \prod_{i=0}^{n-1} p(\eta_i, \eta_{i+1}) \;=\;
M(\eta_n) \, \prod_{i=0}^{n-1} p^*(\eta_{i+1},\eta_i)\;.
\end{equation*}
In particular, for any $\eta\in A$, $\xi\in B$,
\begin{equation*}
M(\eta) \, \bb P_\eta\big[\, H_B< H^+_A \,,\, H_B=H_\xi \,\big] 
\;=\; M(\xi) \, \bb P^*_\xi \big[\, H_A< H^+_B \,,\, H_A=H_\eta
\,\big] \;. 
\end{equation*}
Therefore, since
\begin{equation*}
\sum_{\eta\in A} M(\eta) \, \bb P_\eta \big[\, H_B<
H^+_A \,\big]  \;=\;
\sum_{\eta\in A} \sum_{\xi\in B} M(\eta) \, \bb P_\eta \big[\, H_B<
H^+_A, H_B=H_\xi \,\big] \;,
\end{equation*}
by \eqref{J-1} and the penultimate identity we have that
\begin{equation}
\label{J-5}
\Cap (A,B) \;=\; \sum_{\xi\in B} M(\xi) \, \bb P^*_\xi[ H^+_A< H^+_B]\;
\;=\; \Cap^*(B,A)\;,
\end{equation}
where $\Cap^* (A,B)$ represents the capacity between the sets $A$, $B$
for the adjoint process. 

It follows from \eqref{J-6} and \eqref{J-5} that the capacity is
monotone in the first coordinate as well: if $A'$ is a subset of $E$
such that  $A\subset A'$, $A'\cap B=\varnothing$,
\begin{equation*}
\Cap (A,B) \;\le\; \Cap (A',B) \;.
\end{equation*}

\subsection{A formula for the capacity}

Recall the formula \eqref{M-5} for the Dirichlet form $D(f)$ of a
function $f:E\to \bb R$.  Fix two disjoint subsets $A$, $B$ of $E$:
$A\cap B = \varnothing$.  Denote by $\color{blue} h_{A,B}: E \to \bb
R$ the \emph{\color{blue} equilibrium potential} between $A$ and
$B$. It is the unique solution of the boundary-value elliptic problem
\begin{equation}
\label{J-10}
\left\{
\begin{aligned}
& \; (\, L \, h \,)\, (\eta) \;=\;
0\;,\;\; \eta\not \in A \cup B \;, \\
& \; h (\eta) \;=\; \chi_{A} (\eta) 
\;,\;\; \eta \,\in\, A \cup B \;.
\end{aligned}
\right.
\end{equation}
It has a stochastic representation as
\begin{equation}
\label{J-2}
h_{A,B} (\eta) \;=\;  \bb P_\eta \big[\, H_{A} < H_{B} 
\,\big]\;.
\end{equation}

Since $h_{A,B}$ is harmonic on $(A\cup B)^c$, it vanishes over $B$ and
it is equal to $1$ at $A$,
\begin{equation*}
D(h_{A,B}) \;=\; \<\, (-L h_{A,B}) \,,\, h_{A,B}\, \>_{\pi} \;=\; 
\sum_{\eta \in A} \sum_{\xi\in E} \pi (\eta)\, R (\eta,\xi)\, 
[\, 1 - h_{A,B}(\xi) \,] \;.
\end{equation*}
By the representation \eqref{J-2} of the equilibrium potential,
$1 - h_{A,B}(\xi) = \bb P_\xi [\, H_{B} < H_{A} \,]$. By the strong
Markov property at the first jump, for every $\eta\in A$,
\begin{equation*}
\bb P_\eta [\, H_{B} < H^+_{A} \,] \;=\; 
\sum_{\xi\in E} p (\eta,\xi) \, \bb P_\xi [\, H_{B} < H_{A} \,] \;.
\end{equation*}
Hence,
\begin{equation}
\label{J-3}
D(h_{A,B}) \;=\; \sum_{\eta \in A} \pi (\eta)\, \lambda (\eta)\, 
\bb P_\eta [\, H_{B} < H^+_{A} \,] \;=\; \Cap(A,B)\;.
\end{equation}

The capacity is symmetric: By \eqref{J-2}, $h_{B,A} = 1 - h_{A,B}$,
and, by \eqref{M-5}, $D(h_{A,B}) = D(1 - h_{A,B})$. Hence,
\begin{equation}
\label{J-8}
\Cap(A,B) \;=\; D(h_{A,B}) \;=\; D(1 - h_{A,B}) \;=\; D(h_{B,A}) \;=\;
\Cap(B,A)\;. 
\end{equation}

\subsection{Flows} 

Denote by $c(\eta,\xi)$ the conductance of the oriented edge
$(\eta,\xi)$, and by $c_s(\eta,\xi)$ its symmetric version:
\begin{equation}
\label{J-20}
{\color{blue} c(\eta,\xi)} \;=\; \pi(\eta)\, R(\eta,\xi)\;, \quad
{\color{blue} c_s(\eta,\xi)} \;=\; \frac 12\, \big\{c(\eta,\xi) 
\,+\, c(\xi,\eta) \big\} \;.
\end{equation}
Note that $c_s(\eta,\xi) = (1/2) \, \pi(\eta)\, \{\, R(\eta,\xi) +
R^*(\eta,\xi) \,\}$. 

Let $\mf E$ be the set of oriented edges defined by
\begin{equation*}
{\color{blue} \mf E} \;: =\; \{(\eta,\xi)\in E\times E~:
\,c_{s}(\eta,\xi)>0\}\;.
\end{equation*}
An anti-symmetric function $\phi:\mf E\to\mathbb{R}$ is called a
$\textit{flow}$. The \emph{\color{blue}divergence} of a flow $\phi$ at
$\eta\in E$ is defined as
\begin{equation*}
{\color{blue} (\mbox{div}\,\phi)(\eta)} \;=\;
\sum_{\xi:(\eta,\xi)\in \mf E} \phi (\eta,\xi)\;,
\end{equation*}
while its divergence on a set $A\subset E$ is given by
\begin{equation*}
(\text{div }\phi)(A) \;=\; \sum_{\eta\in A}
(\text{div}\,\phi)(\eta)\;. 
\end{equation*}
The flow $\phi$ is said to be \textit{\color{blue} divergence-free at}
$\eta$ if $(\mbox{div}\,\phi)(\eta)=0$.

Denote by $\color{blue}\mf F$ the set of flows endowed with the scalar
product given by
\begin{equation*}
\left\langle \phi,\psi\right\rangle \;=\; 
\frac{1}{2}\, \sum_{(\eta,\xi)\in \mf E}
\frac 1{c_{s}(\eta,\xi)} \, \phi(\eta,\xi)\, \psi(\eta,\xi)
\;,\quad
\text{and let} \quad \left\Vert \phi\right\Vert^{2} 
\;=\; \left\langle \phi,\phi\right\rangle\;. 
\end{equation*}

\begin{remark}
\label{J-r1}
If the Markov chain is irreducible, the set of oriented edges $\mf E$
represents the set $\{(\eta, \xi) \in E \times E : R(\eta, \xi) +
R(\xi,\eta)>0\}$. Define the flow $\phi_R: \mf E \to \bb R$ by $\phi_R
(\eta, \xi) = R(\eta, \xi) - R(\xi,\eta)$. In this language, the
stationary state corresponds to the non-negative function $m: E \to
\bb R_+$ defined on the vertices which makes the function $\varphi_R:
\mf E \to \bb R$, defined by $\varphi_R (\eta, \xi) = m(\eta) \phi_R
(\eta, \xi)$ divergence free at every vertex.
\end{remark}

\subsection{The Dirichlet and the Thomson principles}

For a function $f:E\rightarrow\mathbb{R}$, define the flows
$\Phi_{f}$, $\Phi_{f}^{*}$ and $\Psi_{f}$ by
\begin{equation}
\label{J-11}
\begin{aligned}
& \Phi_{f}(\eta,\xi) \;=\; 
f(\eta)\, c(\eta,\xi) \,-\,
f(\xi)\, c(\xi,\eta)\;, \\
& \quad \Phi_{f}^{*}(\eta,\xi) \;=\; 
f(\eta)\, c(\xi,\eta) \,-\,
f(\xi)\, c(\eta,\xi)\;,\\
& \qquad \Psi_{f}(\eta,\xi) \;=\; 
c_s(\eta,\xi)\, [\, f(\eta)-f(\xi)\,]\;.  
\end{aligned}
\end{equation}
It follows from the definition of these flows that for all functions
$f:E\to\bb R$, $g:E\to\bb R$,
\begin{equation}
\label{J-12}
\begin{gathered}
\langle \Psi_{f} , \Phi_g \rangle  \;=\;
\< \, (- L) \, f \,,\, g \,\>_{\pi} \;, \qquad 
\langle \Psi_{f} , \Phi^*_g \rangle \;=\;
\<\, (- L^*) \, f \,,\, g \, \>_{\pi} \;, \\
\langle \Psi_{f} , \Psi_g \rangle  \;=\;
\< \, (- L^s) \, f \,,\, g \,\>_{\pi}
\;. 
\end{gathered}
\end{equation}

Fix two disjoint subsets $A$, $B$ of $E$ and two real numbers $a$,
$b$. Denote by $\mathfrak{C}_{a,b}(A,B)$ the set of functions
$f:E\to \bb R$ which are equal to $a$ on $A$ and $b$ on $B$:
\begin{equation*}
{\color{blue}\mathfrak{C}_{a,b}(A,B)} \;:=\; 
\big \{\, f:E\rightarrow\mathbb{R}:
f|_{A}\equiv a,\,f|_{B}\equiv b  \, \big \} \;.
\end{equation*}
Let $\mf F_{a}(A,B)$ be the set of flows from $A$ to $B$ with strength
$a\in \bb R$:
\begin{equation*}
\begin{aligned}
{\color{blue} \mf F_{a}(A,B)} \;=\;
\big \{\, \phi\in\mf F \,:\,  & (\mbox{div }\phi)(A)
\,=\, a\,=\, -\, (\mbox{div }\phi)(B)\,,\\
&\qquad (\mbox{div}\,\phi)(\eta)=0\,,\,
\eta\in(A\cup B)^{c}\big\}\;.
\end{aligned}
\end{equation*}
In particular, $\mf F_{1}(A,B)$ is the set of \emph{\color{blue}
  unitary} flows from $A$ to $B$.

Let $\color{blue} h_{A,B}^{*}$ be the equilibrium potential
corresponding to the adjoint dynamics. It is the solution of the
elliptic problem \eqref{J-10} with the adjoint generator $L^*$ in
place of $L$. It can be represented through the adjoint chain
$\eta^*(t)$ by equation \eqref{J-2} with the obvious modifications.

\begin{theorem}[Dirichlet principle]
\label{J-t1}
For any disjoint and non-empty subsets $A$, $B$ of $E$,
\begin{equation*}
\Cap (A,B) \;=\; 
\inf_{f\in\mathfrak{C}_{1,0}(A,B)}\,
\inf_{\phi\in\mf F_{0}(A,B)}
\left\Vert \Phi_{f}-\phi\right\Vert ^{2}\;.
\end{equation*}
Furthermore, the unique optimizers of the variational problem 
are given by 
\begin{equation*}
f \,=\,\frac{1}{2}(h_{A,B}+h_{A,B}^{*})\;\;\mbox{ and }\;\;
\phi \,=\,\frac{1}{2}(\Phi_{h_{A,B}^{*}}-\Phi_{h_{A,B}}^{*})\;.
\end{equation*}
\end{theorem}

\begin{theorem}[Thomson principle]
\label{J-t2}
For any disjoint and non-empty subsets $A$, $B$ of $E$,  
\begin{equation*}
\frac 1{\Cap (A,B)} \;=\; \inf_{\psi\in\mf F_{1}(A,B)}\, 
\inf_{g\in\mathfrak{C}_{0,0}(A,B)}\,
\Vert \Phi_{g}-\psi\Vert ^{2}\;.
\end{equation*}
Furthermore, the unique optimizers of the variational problem are
given by
\begin{equation*}
g\,=\, \frac 12\, \frac{h_{A,B}^{*}-h_{A,B}}{\Cap (A,B)}
\;\;\mbox{ and }\;\;\psi \,=\,
\frac 12\, \frac{\Phi_{h_{A,B}^{*}} + \Phi_{h_{A,B}}^{*}}
{\Cap (A,B)}\;\cdot
\end{equation*}
\end{theorem}

Theorem \ref{J-t1} appeared in Gaudilli\`ere and Landim \cite{gl14},
and Theorem \ref{J-t2} is due to Slowik \cite{Slo}.  Similar Dirichlet
and Thomson principles are available in the context of diffusions
processes, \cite{lms17, l-icm}.

\begin{remark}
\label{J-rm5}
Both theorems require an explicit knowledge of the invariant measure
which is not always available in non-reversible dynamics. An important
open problem consists therefore to derive formulas for the capacity
which do not involve the stationary state.
\end{remark}

\begin{remark}
\label{J-rm4}
These variational formulae, expressed as infima, provide simple lower
and upper bounds for the capacity. To obtain sharp bounds, 
good approximations of the harmonic functions are needed to produce
test functions and test flows close to the optimal ones. In concrete
examples, one of the difficulties is that the test flows constructed
are never divergence free, and a correction has to be introduced to
remove the divergence of the test flow, \cite{lmt2015, ls2018, s2018}.
\end{remark}

\subsection{Reversible dynamics}  

In the reversible case, the conductance is symmetric: $c(\eta,\xi) =
c(\xi,\eta)$. In particular, all flows $\Phi_f$, $\Phi^*_f$, $\Psi_f$,
introduced in \eqref{J-11}, coincide, and the optimal flow $\phi$ of
Theorem \ref{J-t1} vanishes because the equilibrium potentials
$h^*_{A,B}$, $h_{A,B}$ are equal. Hence, in the reversible case,
\begin{equation*}
\Cap (A,B) \;=\; 
\inf_{f\in\mathfrak{C}_{1,0}(A,B)}\,
\left\Vert \Phi_{f}\right\Vert ^{2}
\;=\; \inf_{f\in\mathfrak{C}_{1,0}(A,B)}\,
\< \, (- L) \, f \,,\, f \,\>_{\pi}  \;.
\end{equation*}
where the last identity follows from \eqref{J-12}. We recover in this
way the Dirichlet principle for reversible dynamics:
\begin{equation}
\label{J-13}
\Cap (A,B) \;=\; \inf_{f\in\mathfrak{C}_{1,0}(A,B)}\, D(f)
\end{equation}

In the Thomson principle, the optimal function $g$ vanishes, and we
recover the Thomson principle for reversible dynamics:
\begin{equation*}
\frac 1{\Cap (A,B)} \;=\; \inf_{\psi\in\mf F_{1}(A,B)}\, 
\Vert \psi\Vert ^{2}\;.
\end{equation*}

In the reversible case, the Thomson principle can also be expressed in
terms of functions.

\begin{lemma}
\label{J-l3}
We have that
\begin{equation*}
\frac 1{\Cap (A,B)} \;=\; \inf_{f}\, 
\frac{D(f)}{\Big(\sum_{\eta \in A} \pi (\eta)\, (L\, f) (\eta) \Big)^2} \;,
\end{equation*}
where the infimum is carried over all functions $f:E\to \bb R$ such
that $(Lf)(\eta)=0$ for all $\eta\in E \setminus (A\cup B)$.
\end{lemma}

\begin{proof}
Fix a function $f:E\to \bb R$ such that $(Lf)(\eta)=0$ for all
$\eta\in E \setminus (A\cup B)$.  By Schwarz inequality and equation
\eqref{M-5} for the Dirichlet form,
\begin{equation*}
\Big( \, \frac 12 \sum_{\eta,\xi \in E} \pi (\eta)\, R (\eta,\xi)\, 
[\, f(\xi) \,-\, f(\eta) \,] \, [\, h_{A,B}(\xi) \,-\, h_{A,B}(\eta) \,] \,\Big)^2
\;\le\; D(f)\, D(h_{A,B})\;. 
\end{equation*}
As the chain is reversible, the jump rates satisfy the detailed
balance conditions \eqref{M-10}. We may thus rewrite the sum appearing
on the left-hand side as
\begin{equation*}
-\, \sum_{\eta,\xi \in E} \pi (\eta)\, R (\eta,\xi)\, 
[\, f(\xi) \,-\, f(\eta) \,] \, h_{A,B}(\eta)  \;=\;
-\, \sum_{\eta \in E} \pi (\eta)\, (L\, f) (\eta) \, h_{A,B}(\eta) \;.
\end{equation*}
Since $h_{A,B} = \chi_A$ on $A\cup B$ and $Lf=0$ on the complement,
the previous sum is equal to
\begin{equation*}
-\, \sum_{\eta \in A} \pi (\eta)\, (L\, f) (\eta) \;.
\end{equation*}

We have thus proved that 
\begin{equation*}
\sup_f \Big( \sum_{\eta \in A} \pi (\eta)\, (L\, f) (\eta) \,\Big)^2 \,
\frac 1{D(f)} \,
\;\le\; D(h_{A,B})\;, 
\end{equation*}
where the supremum is carried over all functions $f$ satisfying the
assumptions of the lemma. This inequality is actually an identity
because the equilibrium potential $h_{A,B}$ belongs to the class of
functions considered [it is harmonic on $(A\cup B)^c$] and
\begin{equation*}
\sum_{\eta \in A} \pi (\eta)\, (L\, h_{A,B}) (\eta) \;=\; D(h_{A,B})\;.
\end{equation*}
To complete the proof of the lemma, it remains to recall that 
$\Cap (A,B) = D(h_{A,B})$.
\end{proof}

\begin{remark}
\label{M-rm1}
By inserting test functions, the previous lemma provides lower bounds
for the capacity between two sets. In practical situations, however,
it is almost impossible to find functions which are harmonic at every
point of $(A\cup B)^c$. But it might be possible to find functions
which are almost harmonic in the sense that $Lf$ is small. The
previous proof applied to any test function yields that for every
$\epsilon>0$, 
\begin{equation*}
(1-\epsilon)\,
\Big( \sum_{\eta \in A} \pi (\eta)\, (L\, f) (\eta) \,\Big)^2 \,
\;-\; \frac{1}{\epsilon} \, \Big( \sum_{\eta \in (A\cup B)^c} \pi (\eta)\, \big|\,
(L\, f) (\eta) \,\big| \,\Big)^2  \;\le\; D(f) \, D(h_{A,B})
\end{equation*}
where we used Young's inequality $2ab \ge - \epsilon a^2 -
\epsilon^{-1} b^2$ and the fact that the absolute value of the
harmonic function is bounded by $1$. The advantage of this inequality
with respect to the Thomson principle lies in the fact that it holds
for all functions $f:E\to \bb R$ and not only for the harmonic ones in
$(A\cup B)^c$. However, the resulting lower bound for the capacity
will be sharp only if $f$ is almost harmonic on $(A\cup B)^c$.
\end{remark}

\begin{remark}
\label{M-rm2}
The previous remark can be extended to all principles stated in the
previous and in the next section. It is this version which is used in
concrete examples. We refer to Theorem 5.3 of \cite{s2018}.
\end{remark}

\subsection{Dirichlet principle II}  

We provide in this subsection an alternative variational formula for
the capacity in terms of functions only.

Fix two disjoint subsets $A$, $B$ of $E$. Let
$\color{blue} \mf F_{0}(A,B)^\perp$ be the set of flows in $\mf F$
which are orthogonal to all flows in $\mf F_{0}(A,B)$. By
\cite[Theorem 8.7]{lax}, for every function $f$ in
$\mathfrak{C}_{1,0}(A,B)$,
\begin{equation*}
\inf_{\phi\in\mf F_{0}(A,B)}
\left\Vert \Phi_{f}-\phi\right\Vert ^{2} \;=\; \sup_{\psi\in\mf F_{0}(A,B)^\perp}
\frac{\<\, \Phi_{f} \,,\, \psi\,\>^2 }{\<\, \psi \,,\, \psi\,\>}  \;,
\end{equation*}
where the supremum is carried over all $\psi\not = 0$. We may rewrite
the right-hand side to obtain that
\begin{equation}
\label{J-16}
\inf_{\phi\in\mf F_{0}(A,B)}
\left\Vert \Phi_{f}-\phi\right\Vert ^{2} \;=\; \sup_{\psi\in\mf F_{0}(A,B)^\perp}
\big\{\, 2 \, \<\, \Phi_{f} \,,\, \psi\,\> \;-\; \<\, \psi \,,\,
\psi\,\> \, \big\}  \;,
\end{equation}
which is more convenient.

\begin{lemma}
\label{J-l1b}
We have that
\begin{equation*}
\mf F_{0}(A,B)^\perp \;=\; {\color{blue} \mf C(A,B)} \;:=\;
\big\{\, \Psi_f : f \in \mf C_{a,b} 
\;\text{ for some } a\,,\, b \in \bb R\,\big\}\;.
\end{equation*}
\end{lemma}

\begin{proof}
Denote by $\mf A$ the set on the right-hand side. Its is clear that
$\mf A \subset \mf F_{0}(A,B)^\perp$. Indeed, fix $\phi\in \mf
F_{0}(A,B)$ and $f$ in $\mf C_{a,b}$ for some $a$, $b\in \bb R$. Then,
\begin{equation*}
\<\, \Psi_{f} \,,\, \phi\,\> \;=\; \frac{1}{2}\, \sum_{(\eta,\xi)\in
  \mf E} [\, f(\eta)-f(\xi)\,] \, \phi (\eta,\xi) \;=\;
\sum_{\eta \in E} f(\eta)\, (\mbox{div}\,\phi) (\eta)\;.
\end{equation*}
As $f$ is constant equal to $a$, $b$ on $A$, $B$, respectively, this
sum can be written as
\begin{equation}
\label{J-18}
a\, \sum_{\eta \in A} (\mbox{div}\,\phi) (\eta)
\;+\; 
\sum_{\eta \not \in A \cup B} f(\eta)\, (\mbox{div}\,\phi) (\eta)
\;+\; 
b\, \sum_{\eta \in B} (\mbox{div}\,\phi) (\eta)
\;.
\end{equation}
Each of these sums vanish because $\phi$ belongs to $\mf F_{0}(A,B)$.

It remains to show that $\mf A^\perp \subset \mf F_{0}(A,B)$. Let
$\phi$ be a flow in $\mf A^\perp$. Then, for all $a$, $b\in \bb R$,
$f$ in $\mf C_{a,b}$, 
\begin{equation*}
\<\, \Psi_{f} \,,\, \phi\,\> \;=\; 0\;.
\end{equation*}
In the first part of the proof, we showed that the left-hand side of
this identity is equal to \eqref{J-18}. Hence, for all $a$, $b\in \bb
R$ and all $f: E \setminus (A\cup B) \to \bb R$, \eqref{J-18}
vanishes. From this we conclude that for all $\xi\not \in A \cup B$,
\begin{equation*}
\sum_{\eta \in A} (\mbox{div}\,\phi) (\eta)
\;=\; 
(\mbox{div}\,\phi) (\xi)
\;=\; 
\sum_{\eta \in B} (\mbox{div}\,\phi) (\eta) \;=\; 0\;.
\end{equation*}
This proves that $\phi$ belongs to $\mf F_{0}(A,B)$ and completes the
proof of the lemma.
\end{proof}

It follows from \eqref{J-16}, the previous lemma and \eqref{J-12}
that
\begin{equation*}
\inf_{\phi\in\mf F_{0}(A,B)}
\left\Vert \Phi_{f}-\phi\right\Vert ^{2} \;=\; \sup_{g\in\mf C(A,B)}
\big\{\, 2 \, \<\, f \,,\, L\, g\,\>_\pi \;-\; \<\, (-\, L^s g) \,,\,
g\,\>_\pi \, \big\}  \;,
\end{equation*}
where the set $\mf C(A,B)$ has been introduced in the statement of
Lemma \ref{J-l1b}. We replaced $g$ by $-g$ in the previous expression
to remove the minus sign in the first term. 

The previous argument permitted to formulate in terms of functions a
variational formula originally expressed through flows.  Since, by
\eqref{M-8}, $\<\, L^s g \,,\, g\,\>_\pi = \<\, L g \,,\, g\,\>_\pi$, in
the previous formula we may replace $L^s$ by $L$. This identity
together with Theorem \ref{J-t1} provides a Dirichlet principle in
terms of functions only. This is the content of the next result. In
contrast with the one formulate in terms of flows, it involves an
$\inf\, \sup$ instead of an $\inf\, \inf$ which is simpler to
estimate.

\begin{theorem}
\label{J-t4}
Let $A$, $B$ be disjoint, non-empty subsets of $E$. Then,
\begin{equation*}
\Cap (A,B) \;=\; 
\inf_{f\in\mathfrak{C}_{1,0}(A,B)}\, \sup_{g\in \mf C(A,B)}
\Big\{\, 2\, \<\,  f \,,\, L \, g \,\>_{\pi} \;-\; 
\<\, (- L) \, g \,,\, g\,\>_{\pi} \Big\} \;.
\end{equation*}
Moreover, the optimal function is given by $f = (1/2) \{ h_{A,B} +
h^*_{A,B}\}$.
\end{theorem}

Theorem \ref{J-t4} has been proved by Doyle \cite{doy} and,
independently, by Gaudilli\`ere and Landim \cite{gl14}. A version in
the context of diffusions is due to Pinsky \cite{p1, p2}.

\begin{remark}
\label{J-rm2}
It is also possible to transform the variational problem
\begin{equation*}
\inf_{g\in\mathfrak{C}_{0,0}(A,B)}\, \Vert \Phi_{g}-\psi\Vert ^{2}
\end{equation*}
into a supremum over flows satisfying certain identities. The
resulting variational formula does not seem to be useful.
\end{remark}

\subsection{Sector condition}

Recall from \eqref{M-9} that we denote by $L^s$ the symmetric part of
the operator $L$ in $L^2(\pi)$: $L^s = (1/2)(L+L^*)$. This operator is
self-adjoint in $L^2(\pi)$ and the corresponding Markov chain, denoted
by $\color{blue} \eta^s(t)$ is reversible.  Moreover, for every
function $f:E\to \bb R$,
\begin{equation*}
\< \, (- L^s) \, f \,,\, f \,\>_{\pi} \;=\; \< \, (- L) \, f \,,\, f
\,\>_{\pi} \;=\; D(f)\;.
\end{equation*}
Therefore, the Dirichlet form associated to the operator $L^s$,
denoted by $D^s(f)$ and defined by the leftmost term of the previous
equation, coincides with the Dirichlet form of the original process.

In particular, if we represent by $\Cap^s (A,B)$ the capacity between
two disjoint, non-empty subsets $A$, $B$ with respect to the chain
$\eta^s(t)$, by \eqref{J-13},
\begin{equation*}
\Cap^s (A,B) \;=\; \inf_{f\in\mathfrak{C}_{1,0}(A,B)}\, D_s(f)
\;=\; \inf_{f\in\mathfrak{C}_{1,0}(A,B)}\, D(f)\;.
\end{equation*}
Hence, as $h_{a,B}$ belongs to $\mf C_{1,0}(A,B)$, by \eqref{J-3} and
the previous identity,
\begin{equation}
\label{J-14}
\Cap^s (A,B) \;\le \; \Cap (A,B) \;.
\end{equation}

It turns out that a converse inequality holds if the generator
satisfies a sector condition.  Recall that a generator $L$ satisfies a
sector condition with constant $C_0$ if for every functions $f$,
$g:E\to \bb R$,
\begin{equation*}
\<\, Lf \,,\, g \,\>_\pi^2 \;\le\; C_0 \, 
\< \, (-L) \, f \,,\, f \,\>_\pi \, \<\, (-L) \, g \, ,\, g \, \>_\pi\;. 
\end{equation*}
Next result states that the capacity between two sets can be estimated
by by the symmetric capacity between these set if the generator
satisfies a sector condition

\begin{lemma}
\label{J-l2}
Suppose that the generator $L$ satisfies a sector condition with
constant $C_0$. Then, for every pair of disjoint subsets $A$, $B$ of
$E$, 
\begin{equation*}
\Cap (A,B) \;\le\; C_0 \, \Cap^s (A,B)\;.
\end{equation*}
\end{lemma}

\begin{remark}
\label{J-rm6}
By equation \eqref{7-2}, the height of a valley is proportional to the
inverse of the capacity. Thus, equation \eqref{J-14} asserts that the
height of a valley in non-reversible dynamics is smaller than the one
in the reversible version. Therefore, non-reversible dynamics mix
faster than their reversible counterpart.
\end{remark}

\begin{remark}
\label{J-rm1}
When the state space $E$ is finite, the generator always satisfies a
sector condition (cf. Corollary \ref{M-s3}), but Lemma \ref{J-l2}
holds in the context of countably-infinite state spaces and diffusions.
\end{remark}

\subsection{Recurrence}

We assume in this section that the set $E$ is countably infinite. A
classical problem in the theory of Markov chains is to determine
wether a chain is recurrent or not. Potential theory is a powerful tool in
this framework.

Here is an open problem, for instance. Consider the random walk in
random environment evolving on $\bb Z^2$ as follows. For each line
$l(k) = \{ (x,k) : x\in \bb Z\}$ flip a fair coin. If it comes
head, on this line the random walk may only jump to the right, while
it may only jump to the left if it comes tail. This represented by
drawing an arrow from $(x,k)$ to $(x+1,k)$ for each $x\in \bb Z$ if
the side shown is head, or from $(x,k)$ to $(x-1,k)$ if it is tail.
Do the same thing for each column to obtain a graph as in Figure
\ref{fig6}.

\begin{figure}[h]
\centering
\begin{tikzpicture}[scale = .6]
\foreach \x in {0, ..., 5}
\draw (\x,-1) -- (\x,6);
\foreach \y in {0, ..., 5}
\draw (-1,\y) -- (6,\y);
\draw[very thick, blue, ->] (0,5) -- (0,6);
\draw[very thick, blue, ->] (1,5) -- (1,6);
\draw[very thick, blue, ->] (2,0) -- (2,-1);
\draw[very thick, blue, ->] (3,5) -- (3,6);
\draw[very thick, blue, ->] (4,0) -- (4,-1);
\draw[very thick, blue, ->] (5,0) -- (5,-1);
\draw[very thick, blue, ->] (0,0) -- (-1,0);
\draw[very thick, blue, ->] (5,1) -- (6,1);
\draw[very thick, blue, ->] (5,2) -- (6,2);
\draw[very thick, blue, ->] (0,3) -- (-1,3);
\draw[very thick, blue, ->] (5,4) -- (6,4);
\draw[very thick, blue, ->] (0,5) -- (-1,5);
\draw[very thick,red,->] (2,3) -- (2,2);
\draw[very thick,red,->] (2,3) -- (1,3);
\end{tikzpicture}
\caption{A random walk in random environment evolving on $\bb Z^2$. At
  the tail of the two red arrows, the random walk may only jump, with
  equal probability, to the left or to the bottom}
\label{fig6}
\end{figure}
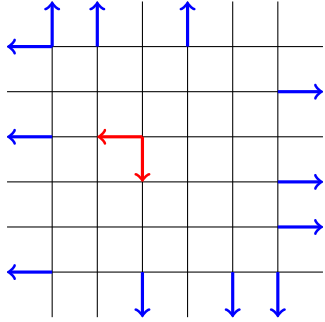 

As illustrated in Figure \ref{fig6}, each point $(x,y)$ in $\bb Z^2$
is the tail of two arrows. Denote by $\eta(t)$ the random walk on $\bb
Z^2$ which waits a mean-one exponential time ate each site of $\bb
Z^2$ and which jumps with equal probability along one of the two arrows.

It is clear that for almost all environments the random walk is
irreducible and that the uniform measure is stationary [because the
flow formed by the arrows is divergence-free]. It is an open problem
to determine if this random walk is almost-surely recurrent or
transient. \smallskip

In view of this example, consider a chain $\eta(t)$ defined on a
countably infinite space $E$ which is irreducible and assume that
there exists a stationary state, denoted by $\pi$. Note that $\pi$ may
not be summable, as in the example above. But we assume that $\pi$ is
explicitly known because all estimates below involve $\pi$. This is
clearly a strong hypothesis and in many cases a stationary state is
not known.

Recall that the Markov chain $\eta(t)$ is recurrent if and only if
there exist a configuration $\eta\in E$ such that $\bb P_\eta [
H^+_\eta = \infty ] =0$. There is nothing special about $\eta$. If
this identity holds for some configuration $\eta$, due to the
irreducibility, it holds for every.  Let $(B_n : n\ge 1)$ be a
sequence of \emph{finite} subsets of $E$ containing $\eta$ and
increasing to $E$, $\eta\in B_n\subset B_{n+1}$, $\cup_n B_n =
E$. Then,
\begin{equation*}
\bb P_\eta \big [\, H^+_\eta = \infty \, \big] \;=\;
\lim_{n\to\infty} \bb P_\eta \big [ H_{B_n^c} < H_\eta^+ \big] \;.
\end{equation*}

By definition \eqref{J-1} of the capacity, for any finite set $B$
containing the site $0$,
\begin{equation*}
\frac 1{M (\eta)} \, \bb P_\eta \big [ H_{B^c} < H_\eta^+ \big] \;=\; 
\Cap (\eta, B^c)\;,
\end{equation*}
where $M(\xi) = \pi(\xi)\, \lambda(\xi)$, $\lambda(\xi)$ being the
holding rate at $\xi$. Hence, the Markov chain $\eta(t)$ is recurrent
if and only if there exist a configuration $\eta\in E$ and a sequence
of finite subsets $B_n$ containing $\eta$ and increasing to $E$ such that
\begin{equation}
\label{J-19}
\lim_{n\to\infty} \Cap (\eta, B_n^c) \;=\; 0\;.
\end{equation}
The proof of the recurrence is thus reduced to the estimation of the
capacity between a configuration and the complement of a finite set. 

Of course, if condition \eqref{J-19} holds for some configuration
$\eta\in E$ and for some sequence of finite subsets $B_n$ containing
$\eta$ and increasing to $E$, it also holds for all configurations
$\xi\in E$ and for all sequences of finite subsets $C_n$ containing
$\xi$ and increasing to $E$.

The next two results, taken from \cite{gl14}, follow from the previous
observation and the estimate \eqref{J-14} and Lemma \ref{J-l2}. Recall
from the previous subsection that $\eta^s(t)$ stands for the
reversible version of the process $\eta(t)$ whose generator is given
by $L^s$ introduced in \eqref{M-9}.

\begin{theorem}
\label{J-t5}
Let $\eta(t)$ be a irreducible Markov chain on a countable state space
$E$ which admits a stationary measure. The process is transient if so
is the Markov chain $\eta^s(t)$.
\end{theorem}

\begin{theorem}
\label{J-t6}
Let $\eta(t)$ be a irreducible Markov chain on a countable state
space $E$ which admits a stationary measure. The process is recurrent
if its generator satisfies a sector condition and if the Markov
chain $\eta^s(t)$ is recurrent.
\end{theorem}

It follows from these results, cf. \cite{gl14}, that a irreducible
Markov chain on a countable state space $E$ which admits a stationary
measure is recurrent if the Markov chain $\eta^s(t)$ is recurrent and
if
\begin{equation*}
\sum_{(\eta,\xi)\in \mf E} \frac{c_a(\eta,\xi)^2}
{c_s(\eta,\xi)}  \;<\; \infty\;,
\end{equation*}
where the symmetric conductance $c_s$ has been introduced in
\eqref{J-20}, and the asymmetric one is given by $c_a(\eta,\xi) =
(1/2) \, [\, c(\eta,\xi) - c(\xi, \eta)\,]$.

Benjamini and Hermon \cite{Her17, BenHer17} used Theorem \ref{J-t5} to
investigate the recurrence of non-backtracking random walks and to
show that for every transient, nearest-neighbor Markov chain on a
graph, the graph formed by the vertices it visited and edges it
crossed is a.s. recurrent for simple random walk.

\section{Examples}
\label{sec12}

We present in this section some dynamics whose metastable behavior has
been derived with the arguments presented in the article.

\subsection{Random walks in a potential field}

We describe the reversible version of the dynamics. The non-reversible
one is obtained by replacing $2$-cycles, in the terminology of
Subsection \ref{12.2}, by $k$-cycles.

Let $\Xi$ be an open and bounded subset of $\bb R^d$, and denote by
$\partial \,\Xi$ its boundary, which is assumed to be a smooth
manifold. Fix a twice continuously differentiable function
$F:\Xi\cup \partial \,\Xi\to\bb R$.  We assume that the second partial
derivatives of $F$ are Lipschitz continuous; that all the eigenvalues
of the Hessian of $F$ at the critical points which are local minima
are \emph{strictly} positive; that the Hessian of $F$ at the critical
points which are not local minima or local maxima has one strictly
negative eigenvalue, all the other ones being strictly positive. In
dimension $1$ this assumption requires the second derivative of $F$ at
the local minima to be strictly negative. Finally, we assume that for
every $\bs x\in \partial\, \Xi$, $(\nabla F)(\bs x) \cdot \bs n (\bs
x) > 0$, where $\bs n (\bs x)$ represents the exterior normal to the
boundary of $\Xi$, and $\bs x\cdot\bs y$ the scalar product of $\bs
x$, $\bs y\in\bb R^d$. This hypothesis guarantees that $F$ has no
local minima at the boundary of $\Xi$.

Denote by $\Xi_N$ the discretization of $\Xi$: $\Xi_N = \Xi \cap
(N^{-1} \bb Z^d)$, $N\ge 1$, where $N^{-1} \bb Z^d = \{\bs k/N : \bs
k\in \bb Z^d\}$. The elements of $\Xi_N$ are represented by the
symbols $\bs x=(\bs x_1, \dots, \bs x_d)$, $\bs y$ and $\bs z$.  Let
$\mu_N$ be the probability measure on $\Xi_N$ defined by
\begin{equation*}
\mu_N(\bs x) \;=\; \frac 1{Z_N} e^{-N F(\bs x)}\;,\quad\bs x\in
\Xi_N\;, 
\end{equation*}
where $Z_N$ is the partition function $Z_N = \sum_{\bs x\in \Xi_N}
\exp\{-N F(\bs x)\}$.  Let $\{\eta_N(t) : t\ge 0\}$ be the continuous-time
Markov chain on $\Xi_N$ whose generator $\ms L_N$ is given by
\begin{equation}
\label{v54}
(\ms L_N f)(\bs x) \;=\; \sum_{\substack{\bs y \in\Xi_N\\
\Vert\bs y - \bs x \Vert = 1/N}}  e^{-(1/2) N [F(\bs y) - F(\bs x)]} 
\, [f(\bs y) - f(\bs x)]\;,
\end{equation}
where $\Vert \,\cdot\,\Vert$ represents the Euclidean norm of $\bb
R^d$.  The rates were chosen for the measure $\mu_N$ to be reversible
for the dynamics. 

We restrict our atention here to the evolution among the shallowest
valleys. One can infer from this discussion the general case which can
be found in \cite{lmt2015}.  Denote by $\mf M$ the set of local minima
and by $\mf S$ the set of saddle points of $F$ in $\Xi$.  Let $\mf
S_1$ be the set of the lowest saddle points:
\begin{equation*}
\mf S_1 \;=\; \Big\{ \bs z\in \mf S : F (\bs z) = \min\{ F (\bs y) :
\bs y\in \mf S\}\, \Big\}\;.
\end{equation*}
We represent by $\bs z^{1}, \dots , \bs z^{n}$ the elements of
$\mf S_1$, $\mf S_1 = \{ \bs z^{1}, \dots , \bs z^{n} \}$. Denote by
$H$ the height of the saddle points in $\mf S_1$:
\begin{equation*}
H\;=\; F(\bs z^{1}) \;.
\end{equation*}

Let $\widehat{\varOmega}$ be the level set of $\Xi$ defined by
\begin{equation*}
\widehat{\varOmega} \;=\; \big\{\bs x\in \Xi : F(\bs x) 
\le H \big\} \;.
\end{equation*}
The set $\widehat{\varOmega}$ can be written as a disjoint union of
connected components: $\widehat{\varOmega} = \cup_{1\le j\le
  \varkappa} \widehat{\varOmega}_j$, where $\widehat{\varOmega}_j \cap
\widehat{\varOmega}_k = \varnothing$, $j\not = k$, and where each set
$\widehat{\varOmega}_j$ is connected. Some connected component may not
contain any saddle point in $\mf S_1$, and some may contain more than
one saddle point. Denote by $\varOmega_j$, $1\le j\le m$, the
connected components $\widehat{\varOmega}_{j'}$ which contain a point
in $\mf S_1$.  

Each component $\varOmega_j$ is a union of valleys,
$\varOmega_j = W_{j,1} \cup \cdots \cup W_{j,m_j}$. The sets
$W_{j,a}$ are defined as follows. Let $\mathring{\varOmega}_j$ be the
interior of $\varOmega_j$. Each set $W_{j,a}$ is the closure of a
connected component of $\mathring{\varOmega}_j$.  The intersection of
two valleys is a subset of the set of saddle points:
$W_{j,a} \cap W_{j,b} \subset \mf S_1$. Figure \ref{fig7} illustrates
the valleys of two connected components. 

\begin{figure}[htb]
  \centering
  \def\svgwidth{350pt}
  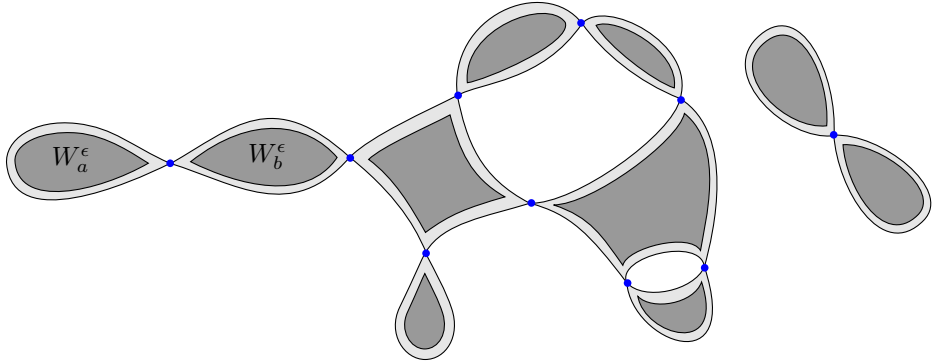
  \caption{Some valleys which form two connected components $\Omega_1$
    and $\Omega_2$. The blue dots represent the saddle points and the
    gray regions the points $\bs x$ in the valleys such that $F(\bs x)
    < H - \epsilon$.}
\label{fig7}
\end{figure}

Fix $1\le j\le m$ and a connected component $\varOmega =
\varOmega_j$. Let $S = \{1, \dots, \ell\}$ denote the set of the
indices of the valleys forming the connected component $\varOmega$:
$\varOmega = W_{1} \cup \cdots \cup W_{\ell}$.  Recall that $F(\bs z)
= H$, $\bs z\in \mf S_1$.  For $\epsilon>0$, $1\le a\le \ell$, let
$W^\epsilon_a = \{\bs x \in W_a : F(\bs x) < H - \epsilon\}$, and let
\begin{equation*}
\ms E^a_N \;=\; W^\epsilon_a \,\cap\, \Xi_N\;,\;\; 1\le a\le
\ell\;.
\end{equation*}
Each valley $W_{a}$ contains exactly one local minimum of $F$, denoted
by $\bs m_{a}$. Let $h_a = F(\bs m_{a})$.

Let $\hat \theta_{a} = H - h_a >0$, $a\in S$, be the depth of the
valley $W_{a} $.  The depths $\hat \theta_{a}$ provide the time-scale
at which a metastable behavior is observed. Let
$\theta_{1} < \theta_{2} < \cdots < \theta_{p}$, $p\le \ell$, be the
increasing enumeration of the sequence $\hat \theta_{a}$,
$1\le a\le \ell$:
\begin{equation*}
\{\hat \theta_{1}, \dots, \hat \theta_{\ell}\} \;=\; 
\{\theta_{1}, \dots, \theta_{p}\}\;.
\end{equation*}

The chain exhibits a metastable behavior on $p$ different time scales
in the set $\varOmega$. Let $T_q = \{a\in S : \hat \theta_{a} =
\theta_q\}$, $1\le q\le p$, so that $T_1, \dots, T_p$ forms a
partition of $S$, and let
\begin{equation*}
S_q \;=\; T_q \cup \cdots \cup T_p\;,\;\; 1\le q\le p\;.
\end{equation*}
Define the projection $\Psi^q_N : \Xi_N\to S_q \cup \{ 0\}$, $1\le
q\le p$, as
\begin{equation*}
\Psi^q_N (\bs x) \;=\; \sum_{a\in S_q} a \, \chi_{\ms E^{a}_N} (\bs x) \;.
\end{equation*}
Note that $\Psi^q_N(\bs x)=0$ for all points $\bs x$ which do not
belong to $\cup_{a\in S_q} \ms E^{a}_N$.  Denote by $X^q_N(t)$ the
projection of the Markov chain $\eta_N(t)$ by $\Psi^q_N$:
\begin{equation*}
X^q_N(t) \;=\; \Psi^q_N (X_N(t))\;.
\end{equation*}

The theory presented in Sections \ref{sec3}--\ref{sec7} yields the
existence, for each $1\le q\le p$, of a time-scale $\beta^q_N$ and a
$S_q$-valued Markov chain $\bs X^q(t)$ with the following property.
For each $a\in S_q$ and sequence of configurations $\bs x_N$ in
$\ms E^{a}_N$, starting from $\bs x_N$, the finite-dimensional
distributions of the projected process
$\bs X^q_N(t) = X^q_N(t \beta^q_N)$ converge to the ones of
$\bs X^q(t)$. The time-scales $\beta^q_N$ can be explicitly computed
and are related to the capacity between valleys.

We refer to \cite{lmt2015, ls2016, ls2018, lms17} for more details.
This model is at the origin of the study of metastability from a
dynamical point of view.  The first results can be traced back at
least to Hood \cite{h78}, van't Hoff \cite{hoff96}, Arrhenius
\cite{a89}, Eyring \cite{e35} and Kramers \cite{kra40}. We refer to
the recent books by Olivieri and Vares \cite{ov05} and Bovier and den
Hollander \cite{bh15} and to the review by Berglund \cite{b13} for
references and alternative derivations of these results.

\subsection{Spin dynamics}

Since the seminal paper by Cassandro, Galves, Olivieri and Vares
\cite{cgov84}, which introduced the pathwise approach to
metastability, the metastable behavior of many spin dynamics have been
derived in different ways.  We do not review here the main results,
but just illustrate the theory developed in the previous sections with
one example. We again refer the reader to \cite{ov05, bh15} for a
complete list of references on the subject.

Denote by $\Lambda_L = \lbrace 1,\dots,L \rbrace^2$ the
two-dimensional discrete torus with $L^2$ elements, and let $\Omega_L
= \lbrace -1, 0, 1 \rbrace^{\Lambda_L}$. Elements of $\Omega_L$ are
represented by the Greek letter $\sigma$. For $x \in \Lambda_L$,
$\sigma(x) \in \{-1,0,1\}$ stands for the value at $x$ of the
configuration $\sigma$ and is called the spin at $x$ of $\sigma$.

The Blume--Capel model was introduced in \cite{Blu66,Cap66} to study
the $^3$He --$^4$He phase transition. One can think as a system of
particles with spins. The value $\sigma(x)=0$ corresponds to the
absence of particles, while $\sigma(x)=\pm 1$ to the presence of a
particle with spin equal to $\pm 1$.

Fix an external field $h\in\bb R$, a magnetic field $\lambda\in\bb
R$, and denote by $\bb H : \Omega_L \rightarrow \bb R$ the Hamiltonian
given by
\begin{equation*} 
\bb H (\sigma) \;=\; \sum \left( \sigma(y) - \sigma(x) \right)^2 
\;-\; h \sum_{x \in \Lambda_L} \sigma(x)
\;-\; \lambda \sum_{x \in \Lambda_L} \sigma(x)^2\;,
\end{equation*}
where the first sum is carried over all unordered pairs of
nearest-neighbor sites of $\Lambda_L$.

Denote by $\mu_\beta$ the Gibbs measure associated to the Hamiltonian
$\bb H$ at inverse temperature $\beta$. This is the probability
measure on $\Omega_L$ given by
\begin{equation} 
\label{N-2}
\mu_\beta(\sigma) = \frac{1}{Z_\beta} e^{-\beta
    \bb H(\sigma)}, 
\end{equation} 
where $Z_\beta$ is the partition function, the normalization constant
which turns $\mu_\beta$ into a probability measure. 

We refer to \cite{co96} for a description of the ground states, the
configurations which minimize the Hamiltonian $\bb H$, according to
the values of the parameters $h$ and $\lambda$.  In all cases, the
ground states form a subset of the set $\{{\bf -1}, {\bf 0}, {\bf +1}
\}$, where ${\bf -1}, {\bf 0}, {\bf +1}$ represent the configurations
of $\Omega_L$ with all spins equal to $-1,0,+1$, respectively.

The continuous-time Metropolis dynamics at inverse temperature $\beta$
is the Markov chain on $\Omega_L$, denoted by $\lbrace \sigma_t : t
\geq 0 \rbrace$, whose infinitesimal generator $L_\beta$ acts on
functions $f : \Omega_L \rightarrow \bb R$ as
\begin{equation*}
\begin{aligned} 
(L_\beta f)(\sigma) &= \sum_{x \in \Lambda_L} R_\beta (\sigma,\sigma^{x,+}) 
\, [ f(\sigma^{x,+}) - f(\sigma) ] \\
&+ \sum_{x \in \Lambda_L} R_\beta (\sigma,\sigma^{x,-}) \, [ f(\sigma^{x,-}) - f(\sigma) ]\;.  
\end{aligned}
\end{equation*}
In this formula, $\sigma^{x,\pm}$ represents the configuration obtained
from $\sigma$ by modifying the spin at $x$ as follows,
\begin{equation*}
\sigma^{x,\pm} (z) := 
\begin{cases}
\sigma (x) \pm 1 \ \textrm{mod}\ 3 & \textrm{ if \ } z=x\;, \\
\sigma (z) & \textrm{ if \ } z\neq x\;, 
\end{cases}
\end{equation*}
where the sum is taken modulo $3$, and the jump rates $R_\beta$ are given
by
\begin{equation*}
R_\beta(\sigma,\sigma^{x,\pm}) \, =\, 
\exp\Big\{ -\beta \, \big[\mathbb H (\sigma^{x ,\pm}) 
- \mathbb H (\sigma) \big]_+\Big\}\, , \quad x\in \Lambda_L\, ,
\end{equation*}
where $a_+$, $a\in \bb R$, stands for the positive part of $a$: $a_+ =
\max\{a, 0\}$.  

The Gibbs measure $\mu_\beta$ introduced in \eqref{N-2} satisfies the
detailed balance conditions \eqref{M-10}, and is therefore reversible
for the dynamics.

Assume from now on that the chemical potential vanishes, $\lambda =0$,
and that the magnetic field $h$ is small and positive, $0<h<2$.  In
this situation, the configurations ${\bf -1}$, ${\bf 0}$ are local
minima of the Hamiltonian, while the configuration ${\bf +1}$ is a
global minimum. Moreover, $\bb H(\textbf{0}) < \bb
H(\textbf{-1})$. 

Assume that $2/h$ is not an integer and let $n_0 = \lfloor 2/h
\rfloor$, where $\lfloor a \rfloor$ stands for the integer part of $a
\in \bb R_+$.  Denote by $\mf R_c$ the set of configurations with $n_0
(n_0 +1) + 1$ $0$-spins forming, in a background of $-1$-spins, a $n_0
\times (n_0 +1)$ rectangle with an extra $0$-spin attached to the
longest side of this rectangle. This means that the extra $0$-spin is
surrounded by three $-1$-spins and one $0$-spins which belongs to the
longest side of the rectangle.

It is proved in \cite{ll18, llm-18} that, as the temperature vanishes,
starting from $\bf -1$ the process visits the set $\mf R_c$ before
hitting $\bf 0$ or $\bf +1$:
\begin{equation*}
\lim_{\beta \to \infty} \bb P_{\bf -1} [ H_{{\mf R}_c} < H_{\{{\bf 0},
  {\bf +1} \}}] \;=\; 1 \;.  
\end{equation*}

The set $\mf R_c$ represents the energetic barrier which has to be
surmounted to pass from $\bf -1$ to $\{{\bf 0}, {\bf +1} \}$. 
Fix $\xi \in \mf R_c$, let
\begin{equation*}
\Delta \;=\; \bb H(\xi) \,-\,  \bb H({\bf -1})  \;=\; 4(n_0+1) -
[n_0(n_0+1) +1] \;,
\end{equation*}
and let $\theta_\beta$ be given by
\begin{equation*}
\theta_\beta \;=\; \frac {\mu_\beta({\bf -1})}
{\Cap({\bf -1}, \{ {\bf 0}, {\bf +1}\})} 
\;=\; \big[ 1+ o_\beta(1) \big]\, \frac{3}{4(2n_0+1)}\,
\frac 1{|\Lambda_L|} \, e^{\Delta\, \beta}\;,
\end{equation*}
where $o_\beta(1)$ is a remainder which vanishes as $\beta\to\infty$.

Fix $\mf d \not = 0$, $\pm 1$, and denote by $\Psi : \Omega_L\to
\{-1,0, 1, \mf d \}$ the projection defined by $\Psi({\bf -1}) = -1$,
$\Psi({\bf 0}) = 0$, $\Psi({\bf +1}) = +1$, and $\Psi(\sigma) = \mf
d$, otherwise. The main results in \cite{ll18, llm-18} state that,
starting from $\bf -1$, the finite-dimensional distributions of the
coarse-grained chain $\bs X_\beta(t) = \Psi\big(\sigma(\theta_\beta
t)\big)$ converge to the ones of the $\{-1,0, 1\}$-valued,
continuous-time Markov chain $\bs X(t)$ in which $1$ is an absorbing
state, and whose jump rates are given by
\begin{equation*}
r(-1,0) \;=\;  r(0,1) \;=\; 1\;, \quad
r(-1,1) \;=\; r(0,-1)\;=\; 0\;. 
\end{equation*}

The metastable behavior of this model has been explored by Cirillo and
Olivieri \cite{co96}, Manzo and Olivieri \cite{mo01}, and more
recently by Cirillo and Nardi \cite{cn13}, and Cirillo, Nardi and
Spitoni \cite{cns17}. The mean-field Potts model is another spin
dynamics in which the spin may take more than two values. It has been
examine recently in \cite{ls2016} and by Nardi and Zocca in
\cite{nz17}.

\subsection{Zero range processes}

Denote by $\bb N$ the set of non-negative integers, $\bb N
=\{0,1,2,...\}$, by $\bb T_L$, $L\ge 1$, the discrete, one-dimensional
torus with $L$ points, and by $\eta$ the elements of $\bb N^{\bb T_L}$
called configurations. The total number of particles at $x\in \bb T_L$
for a configuration $\eta \in \bb N^{\bb T_L}$ is represented by
$\eta_x$. Let $E_N$, $N\ge 1$, be the set of configurations with $N$
particles:
\begin{equation*}
E_{N}\;:=\;\big\{ \eta\in\bb N^{\bb T_L} : \sum_{x\in \bb T_L} \eta_x
= N \big\}\;.
\end{equation*}

Fix $\alpha>1$, and define $g:\bb N\to \bb R_+$ as
\begin{equation*}
g(0)=0\; , \quad g(1)=1  \quad
\textrm{and}\quad g(n)=\frac{a(n)}{a(n-1)}\;, \;\; n\ge 2\;,  
\end{equation*}
where $a(0)=1$, $a(n)=n^{\alpha}$, $n\ge 1$.  In this way,
$\prod_{i=1}^{n}g(i)=a(n)$, $n\ge 1$, and $\{ g(n) : n\ge 2\}$ is a
strictly decreasing sequence converging to $1$ as $n\uparrow\infty$.

Fix $1/2\le p\le 1$, and denote by $p(x)$ the transition probability
given by $p(1)=p$, $p(-1)=1-p$, $p(x)=0$, otherwise.  Let
$\sigma^{x,y}\eta$ be the configuration obtained from $\eta$ by moving
a particle from $x$ to $y$:
\begin{equation*}
(\sigma^{x,y}\eta)_z\;=\;\left\{
\begin{array}{ll}
\eta_x-1 & \textrm{for $z=x$} \\
\eta_y+1 & \textrm{for $z=y$} \\
\eta_z & \rm{otherwise}\;. \\
\end{array}
\right.  
\end{equation*}

The nearest-neighbor, zero-range process associated to the jump rates
$\{g(k) : k\ge 0\}$ and the transition probability $p(x)$ is the
continuous-time, $E_N$-valued Markov process $\{\eta_N(t) : t\ge 0\}$
whose generator $L_N$ acts on functions $f: E_N\to\bb R$ as
\begin{equation*}
(L_N f) (\eta) \;=\; \sum_{\stackrel{x,y\in \bb T_L}{x\not = y}}
g(\eta_x) \, p(y-x) \, \big\{ f(\sigma^{x,y}\eta) - f(\eta) \big\} \;.
\end{equation*}
Hence, if there are $k$ particles at site $x$, at rate $p g(k)$, resp.
$(1-p) g(k)$, one of them jumps to the right, resp. left. Since $g(k)$
decreases to $1$ as $k\to\infty$, the more particles there are at some
site $x$ the slower they jump, but the rate remains bounded below by
$1$. 

This Markov process is irreducible. The stationary probability
measure, denoted by $\pi_N$, is given by
\begin{equation*}
\pi_N(\eta) \;=\; \frac {N^{\alpha}} {Z_{N}} \, 
\prod_{x\in \bb T_L} \frac{1}{ a(\eta_x)} \;,
\end{equation*}
where $Z_{N}$ is the normalizing constant.

Fix a sequence $\{\ell_N : N\ge 1\}$ such that $1\ll \ell_N \ll N$,
and let $\ms E^x_N$, $x\in \bb T_L$, be the set of configurations in
which all but $\ell_N$ particles sit at $x$:
\begin{equation*}
\ms E^x_N  \;:=\; \Big\{\eta\in E_N : \eta_x \ge N - \ell_N \Big\}\;.
\end{equation*}
According to equation (3.2) in \cite{bl2}, for each $x\in \bb
T_L$, $\pi_N(\ms E^x_N) \to 1/L$ as $N\uparrow\infty$.

Denote by $\eta^{\ms E_N}(t)$ the trace of the process $\eta_N(t)$ on
$\ms E_N = \cup_x \ms E^x_N$, and let $\Psi_N:\ms E_N \mapsto S$ be
given by
\begin{equation*}
\Psi_N(\eta) \;=\; \sum_{x\in S} x\; \chi_{_{ \ms E^x_N}}
(\eta)\;.
\end{equation*}

Under some further conditions on the sequence $\ell_N$, it can be
proven, following the method presented in Sections
\ref{sec3}--\ref{sec7}, that the time-rescaled coarse-grained process
$\bs X_N(t) = X_N(t N^{1+\alpha}) = \Psi_N(\eta^{\ms E_N}(t
N^{1+\alpha}))$ converges to a $S$-valued Markov chain $\bs X(t)$.
The jump rates of the reduced model $\bs X(t)$ are proportional to the
capacity of the random walk on the discrete torus with $L$ points
which jumps to the right with probability $p$ and to the left with
probability $1-p$. Moreover, in the time scale $N^{1+\alpha}$ the time
spent by the process $\eta_N(t)$ on $\Delta_N = E_N \setminus \ms E_N$
is negligible.

This model has been introduced by Evans \cite{E-00} Godr\`eche
examined the dynamics of the condensate in \cite{G-03}.  Its
metastable behavior has been derived in \cite{bl3, l2014, s2018}. The
reduced model is a $\bb T_L$-valued Markov chain whose jump rates are
proportional to the capacities of the underlying random walk
associated to $p(\cdot)$.

The nucleation phase of this model has been described in \cite{BJL17}.
Armend\'ariz, Grosskinsky and Loulakis \cite{agl2017} considered the
case in which the total number of sites increases with the number of
particles, keeping a constant density. In this situation, the reduced
model is a L\'evy-process.

Grosskinsky, Redig and Vafayi \cite{GRV-13}, Cao, Chleboun and
Grosskinsky \cite{CCG-14} and Bianchi, Dommers and Giardin{\`a}
\cite{bdg17} proved the metastable behavior of the inclusion process,
another interacting particle system which exhibits condensation.

Static aspects of condensation for this zero-range process and other
dynamics have been examined by Jeon, March and Pittel \cite{JMP-00},
Grosskinsky, Sch\"utz and Spohn \cite{GSS-03}, Armend\'ariz and
Loulakis \cite{AL-09, AL-11}, Chleboun and Grosskinsky \cite{CG-10,
  CG-14, CG-15}, Grosskinsky, Redig and Vafayi \cite{GRV-11},
Godr\`eche and Luck \cite{GL-12}, Armend\'ariz, Grosskinsky and
Loulakis \cite{agl2013}, Fajfrov\'a, Gobron and Saada
\cite{FajGobSaa16}.

In some dynamics the condensate is formed instantaneously as the size
of the system grows, Waclaw and Evans \cite{WE-12}, Chau1, Connaughton
and Grosskinsky \cite{CCG-15}.

\subsection{Random walks among random traps}

Let $(G_N : N\ge 1)$, $G_N = (V_N, E_N)$, be a sequence of possibly
random, finite, connected graphs defined on a probability space
$(\Omega, \mc F, \bb P)$, where $V_N$ represents the set of vertices
and $E_N$ the set of unoriented edges. Assume that the number of
vertices, $|V_N|$, converges to $+\infty$ in $\bb P$-probability.
To fix ideas, one can consider the $d$-dimensional discrete torus with
$N^d$ points.

Assume that on the same probability space $(\Omega, \mc F, \bb P)$, we
are given an i.i.d collection of random variables $\{W^N_j : j\ge
1\}$, $N \geq 1$, independent of the random graph $G_N$ and whose
common distribution belongs to the basin of attraction of an
$\alpha$-stable law, $0<\alpha <1$. Hence, for all $N\ge 1$ and $j\ge
1$,
\begin{equation*}
\bb P[W^N_j>t] \;=\; \frac{L(t)}{t^\alpha}\;, \quad t>0\;,
\end{equation*}
where $L$ is a slowly varying function at infinity.

For each $N\ge 1$, re-enumerate in decreasing order the weights
$W^N_1, \dots , W^N_{|V_N|}$: $\hat W^N_j = W^N_{\sigma(j)}$, $1\le
j\le |V_N|$ for some permutation $\sigma$ of the set $\{1, \dots,
|V_N|\}$ and $\hat W^N_j \ge \hat W^N_{j+1}$ for $1\le j < |V_N|$. Let
$(x^N_1, \dots, x^N_{|V_N|})$ be a random enumeration of the vertices
of $G_N$ and define $W^N_{x^N_j} = \hat W^N_j$, $1\le j\le |V_N|$,
turning $G_N = (V_N, E_N, W^N)$ into a finite, connected,
vertex-weighted graph.

Consider for each $N \geq 1$, a continuous-time random walk
$\{\eta_N(t) : t\ge 0\}$ on $V_N$, which waits a mean $W^N_x$
exponential time at site $x$, after which it jumps to one of its neighbors
with uniform probability. The generator $\ms L_N$ of this walk is given by:
\begin{equation*}
(\ms L_N f)(x) \;=\; \frac 1{\deg(x)} \, \frac 1{W^N_x} \,
\sum_{y\sim x} [f(y) - f(x)]
\end{equation*}
for every $f:V_N\to\bb R$, where $y\sim x$ means that $\{x,y\}$
belongs to the set of edges $E_N$ and where $\deg(x)$ stands for the
degree of $x$: $\deg(x) = \# \{y\in V_N : y\sim x\}$.

Let $\Psi_N : V_N \to \{1, \dots, |V_N|\}$ be given by
$\Psi_N(x^N_j)=j$.  It has been proved for a class of random graphs
that there exists a time-scale $\theta_N$ for which time-rescaled
process $\bs X_N(t) = \Psi_N(\eta_N(t\theta_N))$ converges to a
$K$-process.

To describe the dynamics of the $K$-process, consider two sequences of
positive real numbers $\mb u = (u_k : k \ge 1)$ and $\mb Z = (Z_k : k
\ge 1)$ such that
\begin{equation*}
\sum_{k \ge 1} Z_k \, u_k \;<\; \infty \;, \quad
\sum_{k \ge 1} u_k \;=\; \infty\;.
\end{equation*}

Consider the set $\overline{\mathbb{N}}_* = \{1, 2, \dots \} \cup
\{\infty\}$ of non-negative integers with an extra point denoted by
$\infty$.  We endow this set with the metric induced by the isometry
$\phi:\overline{\mathbb{N}}_* \to \mathbb{R}$, which sends $n \in
\overline{\mathbb{N}}_*$ to $1/n$ and $\infty$ to $0$. This makes the
set $\overline{\mathbb{N}}_*$ into a compact metric space.

The $K$-process with parameter $(Z_k , u_k)$ can be informally
described as follows. Being at $k\in \bb N$, the process waits a mean
$Z_k$ exponential time, at the end of which it jumps to
$\infty$. Immediately after jumping to $\infty$, the process returns
to $\bb N$. The hitting time of any finite subset $A$ of $\bb N$ is
almost surely finite. Moreover, for each fixed $n \ge 1$, the
probability that the process hits the set $\{1, \dots, n\}$ at the
point $k$ is equal to $u_k/\sum_{1\le j\le n} u_j$. In particular, the
trace of the $K$-process on the set $\{1, \dots, n\}$ is the Markov
process which waits at $k$ a mean $Z_k$ exponential time at the end of
which it jumps to $j$ with probability $u_j/\sum_{1\le i\le n} u_i$.

In contrast with the theory presented in the previous sections, here
the reduced model takes value in a countably infinite space. Moreover,
as $\Psi_N$ is a bijection, the process $\bs X_N(t)$ is Markovian, and
we do not need to remove a piece of the state space by considering the
trace, and we prove the convergence of the projection to the reduce
model.

The $K$-process has been introduced by Fontes and Mathieu \cite{fm1}
who also proved the convergence to the $K$-process of the trap model
in the complete graph. Fontes and Lima \cite{fli1} considered the case
of the hypercube. These results have been extended to $d$-dimensional
torus, $d\ge 2$, and to random graphs in \cite{jlt11, jlt14}. More
recently, Cortines, Gold and Louidor considered a continuous time
random walk on the two-dimensional discrete torus, whose motion is
governed by the discrete Gaussian free field \cite{CorGolLou18}.

\subsection{A polymer in the depinned phase}
\label{14-4}

Fix $N\ge 1$ and denote by $E_N$ the set of all lattice paths starting
at $0$ and ending at $0$ after $2N$ steps:
\begin{equation*}
E_N\;=\; \{ \eta\in \bb Z^{2N+1}: \eta_{-N} = \eta_N = 0 \,,\,
\eta_{j+1}-\eta_j = \pm 1 \,,\, -N \le j < N \}\;.
\end{equation*}
Fix $0<\alpha<1$ and denote by $\eta_N(t)$ the $E_N$-valued Markov
chain whose generator $L_N$ is given by
\begin{equation*}
(L_N f)(\eta) \;=\; \sum_{j=-N+1}^{N-1} c_{j,+} (\eta) \, [f(\eta^{j,+})
- f(\eta)] \;+\; \sum_{j=-N+1}^{N-1} c_{j,-} (\eta)\, [f(\eta^{j,-})
- f(\eta)]\;.
\end{equation*}
In this formula $\eta^{j,\pm}$ represents the configuration which is
equal to $\eta$ at every site $k\not = j$ and which is equal to
$\eta_j \pm 2$ at site $j$. 

The jump rate $c_{j,+} (\eta)$ vanishes at configurations $\eta$ which
do not satisfy the condition $\eta_{j-1} = \eta_{j+1} = \eta_j +1$,
and it is given by
\begin{equation*}
c_{j,+} (\eta) \;=\; 
\begin{cases}
1/2 & \text{if $\eta_{j-1} = \eta_{j+1} \not = \pm 
  1$,} \\
1/(1+\alpha) & \text{if $\eta_{j-1} = \eta_{j+1} = 
  1$,} \\
\alpha/(1+\alpha) & \text{if $\eta_{j-1} = \eta_{j+1} = -
  1$} 
\end{cases}
\end{equation*}
for configurations which fulfill the condition $\eta_{j-1} =
\eta_{j+1} = \eta_j +1$. Let $-\eta$ stand for the configuration
$\eta$ reflected around the horizontal axis, $(-\eta)_j = - \eta_j$,
$-N\le j\le N$. The rates $c_{j,-}(\eta)$ are given by $c_{j,-}(\eta)
= c_{j,+}(-\eta)$.

Denote by $\Sigma(\eta)$ the number of zeros in the path $\eta$,
$\Sigma (\eta) = \sum_{-N\le j\le N} \mb 1\{\eta_j=0\}$. The
probability measure $\pi_N$ on $E_N$ defined by $\pi_N(\eta) =
(1/Z_{2N})\, \alpha^{\Sigma(\eta)}$, where $Z_{2N}$ is a normalizing
constant, is easily seen to be reversible for the dynamics generated
by $L_N$.

Denote by $\mf g_N$ the spectral gap of the chain. The exact
asymptotic behavior of $\mf g_N$ is not known, but, by \cite[Theorem
3.5]{cmt}, $\mf g_N \le C(\alpha) (\log N)^8 / N^{5/2}$ for some
finite constant $C(\alpha)$.

Fix a sequence $\ell_N$ such that $1\ll \ell_N \ll N$, and let
\begin{gather*}
\ms E^1_N  \;=\;  \big \{\eta\in E_N : \eta_j>0 \text{ for all } 
-(N - \ell_N) < j < (N-\ell_N) \,\big\}\;, \\
\ms E^2_N = \{\eta\in E_N : -\eta \in \ms E^1_N\}\;,
\quad \Delta_N \;=\; E_N \setminus (\ms E^1_N \cup \ms E^2_N)\;.
\end{gather*}
By equation (2.27) in \cite{clmst}, $\pi_N(\ms E^1_N) = \pi_N(\ms
E^1_N) = (1/2) + O(\ell_N^{-1/2})$.

Denote by $\mf g^{R, j}_N$ the spectral gap of the chain reflected at
$\ms E^j_N$, $j=1$, $2$. By \cite[Proposition 2.6]{clmst}, taking
$\ell_N = (\log N)^{1/4}$, for every $\epsilon >0$, there exists $N_0$
such that for all $N\ge N_0$, $\mf g^{R, j}_N \ge
N^{-(2+\epsilon)}$. In particular, choosing $\epsilon$ small enough
and $\ell_N = (\log N)^{1/4}$,
\begin{equation*}
\mf g _N \;\ll\; \mf g^{R, 1}_N
\end{equation*}
for all $N$ large enough. This shows that the chain equilibrates
inside each valley in a much shorter time-scale than the one in which
it jumps between valleys.

Let $\nu_N$ be a sequence of probability measures concentrated on $\ms
E^1_N$ and which fulfills conditions \eqref{h-12a}. Set $\theta_N = 1/ \mf
g _N$. The method presented in Section \ref{sec11} yields that the
time-rescaled coarse-grained process $\bs X^T_N (t) =
X^T_N(t\theta_N)$, introduced in condition (T1) of Definition
\ref{l1-2}, converges to the $\{1,2\}$-valued Markov chain which
starts from $1$ and jumps from $m$ to $3-m$ at rate $1/2$. Moreover,
in the time scale $\theta_N$, the time spent by the process
$\eta_N(t)$ outside the set $\ms E_N$ is negligible. We refer the
reader to \cite{bl9} for the proofs.

The interest of this model is that the entropy plays an important
role. In contrast with the models presented in the previous
subsections, the metastable behavior is not determined by an energy
landscape, but by a repulsion in a bottleneck region of the space.  In
particular, in the terminology introduced in Remark \ref{5-rm1}, this
dynamics does not visit points and the method presented in Sections
\ref{sec3}--\ref{sec7} does not apply.

Note that the metastable behavior has been derived without a precise
knowledge of the time-scale at which it occurs. Of course, the jumps
between valleys take place in the time-scale $\theta_N$, the inverse of the
spectral gap, but the exact asymptotic behavior of $\mf g_N$ is not
known, and not needed in the proof of the metastable behavior of the
dynamics.

This model has been introduced in \cite{cmt, clmst}. The results
described in this subsection are taken from \cite{bl9}.

\subsection{Coalescing random walks}

Fix $d\ge 2$. Denote $\{e_1, \dots, e_d\}$ the canonical basis of $\bb
R^d$, and by $p$ the probability measure on $\bb Z^d$ given by
\begin{equation*}
\label{22}
p(x) \;=\; \frac 1{2d} \text{ if $x\in \{\pm \, e_1, \dots, \pm \, e_d\}$}\;,
\quad \text{$p(x)= 0$\, otherwise}\;.
\end{equation*}
Let $\bb T^d_N$ be the discrete $d$-dimensional torus with $N^d$
points. Denote by $E_N$ the family of nonempty subsets of $\bb
T^d_N$. Consider coalescing random walks on $\bb T^d_N$. This is the
$E_N$-valued, continuous-time Markov chain, represented by $(A_N(t) :
t\ge 0)$, whose generator $L_N$ is given by
\begin{equation*}
(L_N f)(A) \;=\; \sum_{x\in A} \sum_{y\not \in A} p(y-x) \{ f(A_{x,y}) -
f(A) \} \;+\; \sum_{x\in A} \sum_{y\in A} p(y-x) \{ f(A_{x}) -
f(A) \}\;,
\end{equation*}
where $A_{x,y}$, resp. $A_x$, is the set obtained from $A$ by
replacing the point $x$ by $y$, resp. removing the element $x$:
\begin{equation*}
A_{x,y} \;=\; [A \setminus \{x\}] \cup \{y\} \;, \quad
A_{x} \;=\; A \setminus \{x\} \;.
\end{equation*}

In contrast with the previous dynamics, in this example the reduced
model takes value in a countably infinite state space. Let $S=\{1,1/2
, 1/3, \dots \} \cup\{ 0\}$, and let $C^{1}(S)$ be the set of
functions $f:S\to\mathbb{R}$ of class $C^{1}$, that is $f\in C^{1}(S)$
is the restriction to $S$ of a continuously differentiable function
defined on $\bb R$. For each $f\in C^{1}(S)$ define $\ms
Lf:S\to\mathbb{R}$ as
\begin{equation*}
\label{def1}
(\ms Lf)(y):=
\begin{cases}
\binom{n}{2}\, \Big\{ f\Big(\frac 1{n-1}\Big) - f\Big(\frac 1n\Big) \Big\}\;,
&\quad\textrm{if } y=\frac{1}{n}\;\textrm{and } n\ge 2\;,\\
0\;,&\quad\textrm{if }y=1\;,\\
(1/2)f'(0)\;,&\quad\textrm{if }y=0\;.
\end{cases}
\end{equation*}
Proposition 2.1 in \cite{bcl18} asserts that for each $x\in S$ there
exists a unique solution to the $(\ms L, \delta_x)$-martingale
problem.

Consider the partition of $E_{N}$ given by
\begin{equation*}
\label{parti}
E_{N}=\bigcup_{n\in\mathbb{N}}\ms E_{N}^{n}\;,\quad\text{where}
\quad \ms E_{N}^{n}:=\{A\subset \bb T^d_N :|A|=n\}\,,\;\; n\in\bb N\;.  
\end{equation*}
In this formula, $|A|$ stands for the number of elements of $A$. Let
$\Psi_N:E_N \to S$ be the corresponding projection:
\begin{equation*}
\Psi_N(A) = 1/|A|\;,\quad A\in E_{N}\;.
\end{equation*}

To define the metastable time-scale, consider two independent random
walks $(x_{t}^{N})_{t\geq0}$ and $(y_{t}^{N})_{t\geq0}$ on $\bb
T_{N}^{d}$, both with jump probability given by $p(\cdot)$, starting
at the uniform distribution. Let $\theta_{N}$ be the expected meeting
time:
\begin{equation}
\label{N-1}
\theta_{N}\;:=\; E\big[\,\min\{t\geq0 : x_{t}^{N}=y_{t}^{N}\}\,\big]\,.
\end{equation}
Since $x_{t}^{N}-y_{t}^{N}$ evolves as a random walk speeded-up by
$2$, $\theta_{N}$ represents the expectation of the hitting time of
the origin for a simple symmetric random walk speeded-up by $2$ which
starts from the uniform measure. In a general graph, though, the
time-scale should be given by \eqref{N-1} mutatis mutandis.

Consider a continuous-time, random walk $(x_t)_{t\ge 0}$ on $\bb Z^d$
with jump probabilities given by $p(\cdot)$ and which starts from the
origin. Assume that $d\ge 3$, and denote by $v_d$ the escape
probability: $v_d = P_0[H^{+}_0 = \infty]$. It can be shown that
\begin{equation*}
\label{thetan}
\begin{aligned}
\lim_{N\to \infty} \frac{\theta_{N}}{N^{d}}
\;=\;\frac{1}{2\, v_d}\, \quad \text{in dimension } d\ge 3\;, \\  
\lim_{N\to \infty} \frac{\theta_{N}}{N^{2}\log N} \;=\; 
\frac 1{\pi}  \quad \text{in dimension } d=2\;.
\end{aligned}
\end{equation*}
The factor $2$ in the denominator appears because the process has been
speeded-up by $2$.  In particular, in $d=2$, $1/\pi$ should be
understood as $(1/2) (2/\pi)$. We refer to \cite{bcl18} for a proof of
this result.

Consider the time-rescaled coarse-grained process
\begin{equation*}
\bb X_N(t) \;=\; \Psi_N(A_N(\theta_N t)) \;, \quad t\ge 0 \;.
\end{equation*}
Note that in this example we do not take the trace of the process on
some set, but we just project it on a smaller state space. 

Applying the ideas presented in the previous sections, it is proved in
\cite{bcl18} that, starting from the configuration in which each site
is occupied by a particle, $\bb X_N(t)$ converges in the Skorohod
topology to the Markov chain whose generator is given by $\ms L$ and
which starts from $0$.

This model has been first considered by Cox \cite{c89}, who proved
that the coalescence time [the time all particles coalesced into one]
is asymptotically equal to a sum of independent exponential random
variables. This result has been extended by Oliveira \cite{Oli12,
  Oli13} to the case of transitive graphs. Related questions have been
examined by Aldous and Fill \cite{af01}, Durrett \cite{dur10}, Cooper,
Frieze and Radzik \cite{cfr09}, Chen, Choi and Cox \cite{ccc16}.

\subsection{Further examples}

We mention in this last subsection other models whose metastable
behavior has been derived with the tools presented in the previous
sections. 

The metastable behavior of sequences of continuous-time Markov chains
on a fixed finite state-space has been examined in \cite{bl4,
  lx15}. This problem has been addressed with large deviations
techniques by Scopolla \cite{Sco93}, Olivieri and Scopolla in
\cite{OliSco95, OliSco96}, Manzo, Nardi, Olivieri and Scoppola
\cite{{mnos04}} and Cirillo, Nardi and Sohier \cite{cns14}.

Properties of hitting times of rare events have been considered in
\cite{blm13}. Fernandez, Manzo, Nardi, Scoppola and Sohier
\cite{fmnss15}, and Fernandez, Manzo, Nardi and Scoppola \cite{fmns15}
examined this question through the pathwise approach.

The evolution, in the zero-temperature limit, of a droplet in the
Ising model under the conservative Kawasaki dynamics in a large
two-dimensional square with periodic boundary conditions has been
derived in \cite{bl5, GoiLan15}. The reduced model in this example is
a two-dimensional Brownian motion on the torus.

Misturini \cite{Mis16} considered the ABC model on a ring in a
strongly asymmetric regime. He derived the metastable behavior of the
dynamics among the segregated configurations in the zero-temperature
limit. Here, the reduced model is a Brownian motion.

\smallskip\noindent{\bf Acknowledgments.} The results presented in
this review are the outcome of long standing collaborations. The
author wishes to thank J. Beltr\'an, A. Gaudilli\`ere, M. Jara,
M. Loulakis, M. Mariani, R. Misturini, M. Mourragui, I. Seo,
A. Teixeira, K. Tsunoda.

M. Ayala, B. van Ginkel, F. Sau and I. Seo read parts of a preliminary
version of this review. Their comments permitted to correct some
errors and to clarify some statements. 

This work has been partially supported by FAPERJ CNE
E-26/201.207/2014, by CNPq Bolsa de Produtividade em Pesquisa PQ
303538/2014-7, by ANR-15-CE40-0020-01 LSD of the French National
Research Agency and by the European Research Council (ERC) under the
European Union’s Horizon 2020 research and innovative programme (grant
agreement No 715734).

\end{document}